\documentclass[11pt,a4paper,oneside,openany]{amsart}
\usepackage{amsfonts,amsmath,amscd,amssymb,amsbsy,amsthm,amstext,amsopn,amsxtra,mathtools}
\usepackage{color,mathrsfs,verbatim,subfigure}
\usepackage{extarrows}
\usepackage[all]{xy}
\usepackage{cleveref}
\usepackage{tikz}
\usepackage{tikz-cd}
\usepackage[enableskew, vcentermath]{youngtab}
\usetikzlibrary{shapes.geometric, calc,matrix,arrows,decorations.pathmorphing,arrows.meta}
\usepackage{caption}
\usepackage{enumitem}
\usepackage{pdfpages}
\usepackage[toc,page]{appendix}
\usepackage[T1]{fontenc}
\usepackage{textcomp}
\usepackage{amsmath}
\usepackage{stmaryrd,amssymb}
\usepackage{amscd}
\usepackage[utf8x]{inputenc}
\usepackage{t1enc}
\usepackage[mathscr]{eucal}
\usepackage{indentfirst}
\usepackage{todonotes}
\usepackage{booktabs}
\usepackage{graphicx}
\usepackage{comment}
\usepackage{epstopdf}

\theoremstyle{plain}
\newtheorem{thm}{Theorem}[section]
\newtheorem{lemma}[thm]{Lemma}
\newtheorem{cor}[thm]{Corollary}
\newtheorem{prop}[thm]{Proposition}
\newtheorem*{claim}{Claim}
\newtheorem*{thmA1}{Theorem A.1}
\newtheorem*{thmB1}{Theorem B.1}
\newtheorem*{thmA2}{Theorem A.2}
\newtheorem*{thmB2}{Theorem B.2}

\theoremstyle{definition}
\newtheorem{defn}[thm]{Definition}

\newtheorem{rmk}[thm]{Remark}

\newtheorem{example}[thm]{Example}

\renewcommand{\Re}{\operatorname{Re}}
\renewcommand{\Im}{\operatorname{Im}}

\newcommand{\CC}{\mathbb{C}}
\newcommand{\PP}{\mathbb{P}}
\newcommand{\QQ}{\mathbb{Q}}
\newcommand{\RR}{\mathbb{R}}
\newcommand{\ZZ}{\mathbb{Z}}

\newcommand{\cH}{\mathcal{H}}

\newcommand{\cO}{\mathcal{O}}

\newcommand{\cX}{\mathcal{X}}
\newcommand{\cY}{\mathcal{Y}}

\newcommand{\Aut}{\operatorname{Aut}}

\newcommand{\ch}{\operatorname{ch}}

\newcommand{\id}{\operatorname{id}}
\newcommand{\oH}{\operatorname{H}}
\newcommand{\Hilb}{\operatorname{Hilb}}
\newcommand{\Hom}{\operatorname{Hom}}

\newcommand{\Mon}{\operatorname{Mon}}

\newcommand{\Or}{\operatorname{O}}

\newcommand{\Pic}{\operatorname{Pic}}

\newcommand{\rk}{\operatorname{rk}}

\newcommand{\sfW}{\mathsf{W}}

\date{}

\title[Locally trivial monodromy of moduli spaces of sheaves]{Locally trivial monodromy of moduli spaces of sheaves on K3 surfaces}
\author{Claudio Onorati}
\address{Dipartimento di Matematica "F. Enriques", Universit\`a degli Studi di Milano, via C. Saldini 50, 20133 Milano, Italia.}
\email{claudio.onorati@unimi.it}
\author{Arvid Perego}
\address{Dipartimento di Matematica, Universit\`a di Genova, Via Dodecaneso 35, 16146 Genova, Italia.}
\email{perego@dima.unige.it}
\author{Antonio Rapagnetta}
\address{Dipartimento di Matematica, Universit\`a di Roma Tor Vergata, via della ricerca scientifica 1, 00133 Roma, Italia.}
\email{rapagnet@mat.uniroma2.it}

\begin{document}

\begin{abstract}
In this paper we study monodromy operators on moduli spaces $M_v(S,H)$ of sheaves on K3 surfaces with non-primitive Mukai vectors $v$. If we write $v=mw$, with $m>1$ and $w$ primitive, then our main result is that the inclusion $M_w(S,H)\to M_v(S,H)$ as the most singular locus induces an isomorphism between the monodromy groups of these symplectic varieties, allowing us to extend to the non-primitive case a result of Markman.
\end{abstract}

\maketitle
\tableofcontents

\section*{Introduction}

Singular symplectic varieties (Section~\ref{section:singular symplectic}) have gained much interest lately, especially after the  proof of a global Torelli Theorem in \cite{BakkerLehn:PrimitiveSymplecticVarieties}. The outcome of these results can be summarised by saying that their geometry behaves very much like the geometry of irreducible holomorphic symplectic manifolds. 
Roughly, 
if $X$ is either a smooth or singular irreducible symplectic variety, most of the geometry of $X$ is controlled by the 
second integral cohomology group $\oH^2(X,\ZZ)$ together with its pure weight two Hodge structure and the Beauville--Bogomolov--Fujiki lattice structure. We recall that the Beauville--Bogomolov--Fujiki lattice $\oH^2(X,\ZZ)$ has always signature $(3,b_2(X)-3)$, where the positive three-space is generated by a K\"ahler class and the real and imaginary part of the symplectic form. The bimeromorphic classification of irreducible symplectic varieties in the same locally trivial deformation class is then encoded in the (locally trivial) \emph{monodromy group}, which is a finite index subgroup $\Mon^2_{\operatorname{lt}}(X)$ of the group $\Or(\oH^2(X,\ZZ))$ of isometries of the lattice $\oH^2(X,\ZZ)$.

In this paper we consider a special class of irreducible symplectic varieties, namely those that are locally trivially deformation equivalent (Definition~\ref{def:lt}) to a moduli space $M_v(S,H)$ of sheaves on a projective K3 surface $S$. Here $v$ is a non-primitive Mukai vector and $H$ is $v$-generic. The notion of $v$-genericity is technical and will be recalled in Definition~\ref{defn:v-generic}. By \cite{PR:SingularVarieties}, it is known that, under these assumptions, $M_v(S,H)$ is indeed an irreducible symplectic variety.

Let us remark that when the Mukai vector $v$ is primitive, the moduli space $M_v(S,H)$ is smooth and it is an irreducible holomorphic symplectic manifold, i.e.\ a simply connected K\"ahler manifold with a unique (up to scalar) holomorphic symplectic form. If $X$ is any manifold deformation equivalent to a smooth moduli space $M_v(S,H)$ as above, then the group $\Mon^2(X)$ of the monodromy operators in $\oH^2(X,\ZZ)$ is known by a result of Markman (\cite{Markman:IntegralConstraints}): it is the group of orientation preserving isometries that act as $\pm\id$ on the discriminant group (see Section~\ref{section:pto} for the notion of orientation). Recall that if $\Lambda$ is a lattice, then the discriminant group is the finite group $\Lambda^*/\Lambda$, where $\Lambda^*=\Hom(\Lambda,\ZZ)$ is the dual $\ZZ$-module. 

Following Markman's notation, if $\Lambda$ is any even lattice of signature $(3,n)$, then we denote by $\mathsf{W}(\Lambda)\subset\Or(\Lambda)$ the subgroup of orientation preserving isometries acting as $\pm\id$ on the discriminant group.

Our first result is the following, which extends \cite[Theorem~1.1]{Markman:IntegralConstraints} to the singular setting.

\begin{thmA1}[Corollary~\ref{cor:ogni X}]
    Let $X$ be an irreducible symplectic variety that is locally trivially deformation equivalent to a moduli space $M_v(S,H)$, where $S$ is a projective K3 surface, $v=mw$ with $m>1$ and $w$ primitive, and $H$ a $v$-generic polarisation. Then 
    \[ \Mon^2_{\operatorname{lt}}(X)=\mathsf{W}(\oH^2(X,\ZZ))\subset\Or(\oH^2(X,\ZZ)). \]
\end{thmA1}

We recall the definition of the group $\mathsf{W}(\oH^2(X,\ZZ))$ at the beginning of Section~\ref{section:W in m k}. We point out that our result does not subside Markman's one, inasmuch as we heavily use \cite[Theorem~1.1]{Markman:IntegralConstraints} in our proof. 

To the best of our knowledge, this is the first time an explicit description of the monodromy group of a class of singular symplectic varieties is exhibited. We recall that monodromy groups have been computed in all the known deformation classes of smooth irreducible holomorphic symplectic manifolds, see \cite{Markman:Monodromy,Markman:Kummer,Mon,MonRap,Onorati:Monodromy}.

As a corollary one can explicitly compute the index of $\Mon^2_{\operatorname{lt}}(X)$ in the isometry group $\Or(\oH^2(X,\ZZ))$ (see Corollary~\ref{cor:index}).
As an example, if $X$ is deformation equivalent to a moduli space $M_v(S,H)$, where $v=mw$ and $w^2=2$, then the group $\Mon^2_{\operatorname{lt}}(M_v(S,H))$ is the whole group $\Or^+(\oH^2(M_v(S,H),\ZZ))$ of orientation preserving isometries (notice that this does not depend on $m$, this feature is shared by an infinite class of examples in any dimension $2m^2+2$).
\medskip

Our second result is an equivalent reformulation of Theorem~A.1, in which the relation between the group $\Mon^2_{\operatorname{lt}}(X)$ and the monodromy group of an irreducible holomorphic symplectic manifold deformation equivalent to a smooth moduli space of sheaves is explained. First of all, if $X$ is locally trivially deformation equivalent to a moduli space $M_v(S,H)$ as before, then the most singular locus $Y$ of $X$ (cf.\ Proposition~\ref{prop:stratification}) is an irreducible holomorphic symplectic manifold deformation equivalent to the moduli space $M_w(S,H)$ (here $w$ is the primitive Mukai vector such that $v=mw$). Let us denote by $i_{Y,X}\colon Y\to X$ the closed embedding. 

The embedding $i_{Y,X}$ induces an homomorphism 
\[ i_{Y,X}^\sharp\colon\Mon^2_{\operatorname{lt}}(X)\longrightarrow\Mon^2(Y). \]
We will define this morphism carefully in Section~\ref{section:main result} (see also Section~\ref{section:moduli spaces} for the case $X=M_v(S,H)$), but we can intuitively describe it as follows. It sends a monodromy operator along a loop $\gamma$ in a family $p\colon\mathcal{X}\to T$ of deformations of $X$, to the monodromy operator along the same loop $\gamma$ on the family of deformations $q\colon\mathcal{Y}\to T$ of $Y$ obtained by restriction from $p$, i.e.\ $\mathcal{Y}\subset\mathcal{X}$ is the relative closed embedding of the most singular locus.

\begin{thmB1}[Corollary~\ref{cor:i sharp is iso gen}]
    Let $X$ be an irreducible symplectic variety that is locally trivially deformation equivalent to a moduli space $M_v(S,H)$ as above. Let $Y\subset X$ be the most singular locus and $i_{Y,X}\colon Y\to X$ the closed embedding. Then 
    \[ i_{Y,X}^\sharp\colon\Mon^2_{\operatorname{lt}}(X)\stackrel{\sim}{\longrightarrow}\Mon^2(Y) \]
    is an isomorphism.
\end{thmB1}

\subsection*{Outline of the proof}
Since the locally trivial monodromy group is invariant along locally trivial families of primitive symplectic varieties, we can reduce the proof of the the two main results to statements about moduli spaces of sheaves. Therefore from now on we will work with $X=M_v(S,H)$, where $S$ is a projective K3 surface, $v$ a Mukai vector of the form $v=mw$, with $m>1$ and $w$ primitive, and $H$ a $v$-generic polarisation.

Let us recall that $v$ belongs to the so-called \emph{Mukai lattice} and so we can consider the ortoghonal complement $v^\perp$, which is an even lattice of signature $(3,20)$. We will recall in Section~\ref{section:moduli spaces} the definitions and constructions. By \cite{PR:v perp}, there is an isometry 
\[ \lambda_{(S,v,H)}\colon v^\perp\to\oH^2(M_v(S,H),\ZZ). \] 
When the Mukai vector $v$ is primitive, the same result is due to O'Grady (see \cite{OGrady:WeightTwo,Yoshioka:ModuliAbelian}).

Theorem~A.1 is equivalent to the following.
\begin{thmA2}[Theorem~\ref{thm:Mon 2 lt}]
    Let $S$ be a projective K3 surface, $v$ a Mukai vector and $H$ a $v$-generic polarisation. Then 
    \[ \Mon^2_{\operatorname{lt}}(M_v(S,H))=\mathsf{W}(\oH^2(M_v(S,H),\ZZ))\cong\mathsf{W}(v^\perp), \]
    where the last isomorphism is induced by the isometry $\lambda_{(S,v,H)}$.
\end{thmA2}
We point out that the special case when $m=2=w^2$ already appeared in \cite[Theorem~6.1]{Onorati:Monodromy} with a different proof.

The proof follows two steps. In the first one we construct monodromy operators: this is performed in Section~\ref{section:W in m k} and the main result is Theorem~\ref{thm:W in m k}, where it is proved that $\mathsf{W}(v^\perp)\subset\Mon^2_{\operatorname{lt}}(M_v(S,H))$. This section parallels Markman's construction of monodromy operators in \cite{Markman:Monodromy}. In fact we point out that it also works for primitive Mukai vectors: our only improvement with respect to Markman's work is that we only work with polarised families of K3 surfaces (see Proposition~\ref{prop:Mon S in m k} and the remark soon after).

In the second step, performed in details in Section~\ref{subsection:i sharp}, we put a constraint on the monodromy group by using the inclusion in $M_v(S,H)$ of its most singular locus. This step is where we crucially use that $v=mw$ is not primitive. In fact in this case the most singular locus of $M_v(S,H)$ can be naturally identified with the smooth moduli space $M_w(S,H)$; let us denote by $i_{w,m}\colon M_w(S,H)\to M_v(S,H)$ the closed embedding. Then by Corollary~\ref{prop:first inclusion} the map $i_{w,m}$ induces an injective homomorphism
\[ i_{w,m}^\sharp\colon\Mon^2_{\operatorname{lt}}(M_v(S,H))\longrightarrow\Mon^2(M_w(S,H) \]
giving the desired constraint. 

The conclusion of the proof is now a straightforward combination of the two steps before and \cite[Theorem~1.1]{Markman:IntegralConstraints}.
\medskip

Working again in the case when $X=M_v(S,H)$, the isomorphism in Theorem~B.1 becomes very natural in terms of the isomorphisms
\[ \lambda_{(S,v,H)}^\sharp\colon\mathsf{W}(v^\perp)\longrightarrow\Mon^2_{\operatorname{lt}}(M_v(S,H)) \]
and
\[ \lambda_{(S,w,H)}^\sharp\colon\mathsf{W}(w^\perp)\longrightarrow\Mon^2(M_w(S,H)) \]
induced by conjugation from the isometries $\lambda_{(S,v,H)}$ and $\lambda_{(S,w,H)}$.
Let us remark that since $v=mw$, the lattices $v^\perp$ and $w^\perp$ are the same sub-lattice of the Mukai lattice of $S$; in particular there is an equality $\mathsf{W}(v^\perp)=\mathsf{W}(w^\perp)$.

\begin{thmB2}[Theorem~\ref{thm:i sharp is iso}]
Let $M_v(S,H)$ be the moduli space of shaves on a K3 surface $S$ with Mukai vector $v$. Assume that $v=mw$, with $m>1$ and $w$ primitive, and that $H$ is $v$-generic. Then the closed embedding $i_{w,m}\colon M_w(S,H)\to M_v(S,H)$ induces an isomorphism
\[ i_{w,m}^\sharp\colon\Mon^2_{\operatorname{lt}}(M_v(S,H))\stackrel{\sim}{\longrightarrow}\Mon^2(M_w(S,H)), \]
and the composition 
\[ (\lambda_{(S,w,H)}^\sharp)^{-1}\circ i_{w,m}^\sharp\circ\lambda_{(S,v,H)}^\sharp\colon\mathsf{W}(v^\perp)\longrightarrow\mathsf{W}(w^\perp) \]
is the identity.
\end{thmB2}

\subsection*{Plan of the paper}
In Section~\ref{section:preliminaries} we recall the basic facts about singular symplectic varieties and moduli spaces of sheaves. The results in this section are all known to experts but we reproduce here some proofs whenever a precise reference was missing in the literature.

Section~\ref{section:groupoid} recalls and extends Markman's construction of a groupoid representation that will be useful to construct monodromy operators and to prove surjectivity of the morphism $i_{w,m}^\sharp$. 

In Section~\ref{section:mon S pol} we provide a proof that the monodromy group of a K3 surface is generated by polarised families. This result is surely well known to experts, but we could not find any reference in the literature; its purpose is to allow us to lift monodromy operators from the K3 surface to the moduli space without deforming to a non-projective K3 surfaces, i.e.\ without the need to consider families of moduli spaces of sheaves on non-projective K3 surfaces.

Section~\ref{section:main} is the main and last section of the paper, where we prove Theorem~\ref{thm:Mon 2 lt} and Theorem~\ref{thm:i sharp is iso}.

\subsection*{Acknowledgements} 
Claudio Onorati was supported by the grant ERC-2017-CoG771507-StabCondEn with principal investigator P. Stellari. Arvid Perego and Antonio Rapagnetta were partially supported by the Research Project PRIN 2020 - CuRVI, CUP J37G21000000001. Claudio Onorati and Antonio Rapagnetta wish to thank the MIUR Excellence Department Project awarded to the Department of Mathematics, University of Rome Tor Vergata, CUP E83C1800010000. Arvid Perego wishes to thank the MIUR Excellence Department Project awarded to the Department of Mathematics, University of Genoa, CUP D33C23001110001. Claudio Onorati and Arvid Perego gratefully acknowledge support from the Simons Center for Geometry and Physics, Stony Brook University, and the Japanese--European Symposium on Symplectic Varieties and Moduli Spaces at which some of the research for this paper was performed. The authors are member of the INDAM-GNSAGA.

\section{Preliminaries}\label{section:preliminaries}

This section is dedicated to review some basic material we will need. More precisely, in the first two subsections we collect some fundamental results and definitions about primitive (resp.\ irreducible) symplectic varieties and their monodromy groups. The third subsection will be devoted to review the main results about moduli spaces of sheaves on K3 surfaces.


\subsection{Singular symplectic varieties}\label{section:singular symplectic}

We first start by recalling the notion of a symplectic variety. Let $X$ be a normal complex analytic variety, and denote by $X_{\operatorname{reg}}$ its smooth locus; notice that the complement of $X_{\operatorname{reg}}$ has codimension at least $2$. If $j\colon X_{\operatorname{reg}}\to X$ is the corresponding open embedding, then for every integer $0\leq p\leq\dim(X)$ we let $$\Omega_{X}^{[p]}:=j_{*}\Omega^{p}_{X_{\operatorname{reg}}}=\big(\wedge^{p}\Omega_{X}\big)^{**},$$whose global sections are called \textit{reflexive $p$-forms} on $X$. A reflexive $p$-form on $X$ is then a holomorphic $p$-form on $X_{\operatorname{reg}}$. 

There is a notion of \emph{singular K\"ahler form} due to Grauert and recalled in \cite[Section~2.3]{BakkerLehn:PrimitiveSymplecticVarieties}. We will not recall the definition of singular K\"ahler forms, but we point out some of their properties. 
Every singular K\"ahler form $\omega$ gives a class $[\omega]\in\oH^2(X,\RR)$; the set of classes in $\oH^2(X,\RR)$ obtained in this way form an open cone that lies in $\oH^{1,1}(S,\RR)$,  where the latter can be interpreted classically as
\[ \oH^{1,1}(X,\RR)=F^1\oH^2(X,\CC)\cap\oH^2(X,\RR), \]
where we are using the mixed Hodge structure on $\oH^2(X,\CC)$ (see \cite[Definition~2.5, Remark~2.6.(1), Proposition~2.8]{BakkerLehn:PrimitiveSymplecticVarieties}). In particular, in the cases of interest to us, the group $\oH^2(X,\CC)$ will have a pure Hodge structure and $\oH^{1,1}(X,\RR)=\oH^1(X,\Omega^{[1]}_X)\cap\oH^2(X,\RR)$.

A normal complex analytic variety admitting a K\"ahler form will be called a \emph{K\"ahler space}. By \cite[II, 1.2.1~Proposition]{Varouchas} a smooth K\"ahler space is a K\"ahler manifold in the classical sense. Moreover, if $X$ is reduced, then there exists a resolution of singularities $\widetilde{X}\to X$ such that $\widetilde{X}$ is a K\"ahler manifold.
A subspace of a K\"ahler space is again a K\"ahler space (see for example \cite[II, 1.3.1(i)~Proposition]{Varouchas}).

We now recall the definitions of a symplectic form and a symplectic variety (cf.\ \cite{Bea00}).

\begin{defn}
{\rm Let $X$ be a normal compact K\"ahler space.
\begin{enumerate}
	\item A \textit{symplectic form} on $X$ is a closed reflexive 2-form $\sigma$ on $X$ which is non-degenerate at each point of $X_{\operatorname{reg}}$.
	\item If $\sigma$ is a symplectic form on $X$, the pair $(X,\sigma)$ is a \textit{symplectic variety} if for every (K\"ahler) resolution $f\colon\widetilde{X}\to X$ of the singularities of $X$, the holomorphic symplectic form $\sigma_{\operatorname{reg}}:=\sigma_{|X_{\operatorname{reg}}}$ extends to a holomorphic 2-form on $\widetilde{X}$. By a slight abuse of notation, in this case we will say that $X$ is a symplectic variety.
\end{enumerate}}
\end{defn}

In what follows, we will be interested in two types of symplectic varieties, namely primitive symplectic varieties and irreducible symplectic varieties. We recall their definitions following \cite{BakkerLehn:PrimitiveSymplecticVarieties} and \cite{GKP}. Before doing this, we recall that if $X$ and $Y$ are two irreducible normal compact complex analytic varieties, a \textit{finite quasi-étale morphism} $f\colon Y\to X$ is a finite morphism which is étale in codimension one.

\begin{defn}
\label{defn:irrvar}
{\rm Let $X$ be a symplectic variety and $\sigma$ a symplectic form on $X$.
\begin{enumerate}
	\item The variety $X$ is a \textit{primitive symplectic variety} if $\oH^{1}(X,\mathcal{O}_{X})=0$ and $\oH^{0}(X,\Omega_{X}^{[2]})=\mathbb{C}\sigma$.
	\item The variety $X$ is an \textit{irreducible symplectic variety} if for every finite quasi-étale morphism $f\colon Y\to X$ the exterior algebra of reflexive forms on $Y$ is spanned by $f^{[*]}\sigma$.
\end{enumerate}}
\end{defn}

An irreducible symplectic variety is primitive symplectic (so in particular it is a symplectic variety), but there are examples of primitive symplectic varieties that are not irreducible symplectic. 

By \cite[Corollary~13.3]{GGK:Klt}, an irreducible symplectic variety $X$ is simply connected. In particular, the $\mathbb{Z}$-module $\oH^{2}(X,\mathbb{Z})$ is free. Moreover, the fact that irreducible symplectic varieties are simply connected together with the Bogomolov Decomposition Theorem imply that smooth irreducible symplectic varieties are irreducible holomorphic symplectic manifolds.

We conclude this section by recalling a result originally due to Kaledin about the stratification of the singularities of a symplectic variety, that together with Remark \ref{rmk:strat} will play an important role in what follows (see \cite[Theorem~2.3]{Kaledin:Poisson} or \cite[Theorem~3.4.(2)]{BakkerLehn:PrimitiveSymplecticVarieties}).

\begin{prop}
\label{prop:stratification}
Let $X$ be a symplectic variety, and consider the finite stratification by closed subvarieties
\[ X=X_0\supset X_1\supset\cdots\supset X_m, \]
where $X_{i+1}$ is the singular locus with reduced structure $(X_{i}^{\operatorname{sing}})_{\operatorname{red}}$  of $X_{i}$. 

Then for every $i=0,\cdots,m$, the normalisation of each irreducible component of $X_i$ is a symplectic variety.
\end{prop}

We notice in particular that each stratum of the stratification of the singularities is even dimensional. 

\subsection{Locally trivial deformations of symplectic varieties}

\begin{defn}\label{def:lt}
\begin{enumerate}
\item A \textit{locally trivial family} is a proper morphism $f\colon\cX\to T$ of complex analytic spaces such that $T$ is connected and, for every point $x\in\cX$, there exist open neighborhoods $V_x\subset \cX$ and $V_{f(x)}\subset T$, and an open subset $U_x\subset f^{-1}(f(x))$ such that 
$$V_x\cong U_x\times V_{f(x)},$$
where the isomorphism is an isomorphism of analytic spaces commuting with the projections over $T$.
A \textit{locally trivial deformation} of a complex analytic variety $X$ is a locally trivial family $f\colon\mathcal{X}\to T$ for which there is $t\in T$ such that $f^{-1}(t)\simeq X$. 

\item A \emph{locally trivial family of primitive (resp.\ irreducible) symplectic varieties} is a locally trivial family whose fibres are all primitive (resp.\ irreducible) symplectic. 

\item Two primitive symplectic varieties are said to be \emph{locally trivially deformation equivalent} if they are members of a locally trivial family of primitive sympelctic varieties.
\end{enumerate}
\end{defn} 

\begin{rmk}
The fibres of a locally trivial family of primitive or irreducible symplectic varieties are K\"ahler spaces by definition.
\end{rmk}

As usual, when we say that any small deformation of $X$ enjoys a property, we mean that there exists an analytic open neighborhood $U$ of the base of a versal deformation of $X$ such that every fiber over $U$ enjoys the property.

The behaviour of primitive symplectic varieties under small locally trivial deformations is known thanks to the following result. 

\begin{prop}[\protect{\cite[Corollary~4.11]{BakkerLehn:PrimitiveSymplecticVarieties}}]\label{prop:small is primitive}
Let $X$ be a primitive symplectic variety. Then any small locally trivial deformation of $X$ is again a primitive symplectic variety.
\end{prop}
 
The same result for irreducible symplectic varieties seems to be unknown. We provide here two arguments that applies to two subclasses of irreducible symplectic varieties, which will be mostly interesting for us. 

\begin{prop}\label{prop:small when simply connected}
Let $X$ be an irreducible symplectic variety with simply connected regular locus.  
Then any small locally trivial deformation $X'$ of $X$ is an irreducible symplectic variety with simply connected regular locus.
\end{prop}
\begin{proof}
Consider a locally trivial deformation $\alpha\colon\mathcal{X}\to T$ of $X$, and let $0\in T$ be such that the fiber $X_{0}$ of $\alpha$ over $0$ is isomorphic to $X$. By Proposition~\ref{prop:small is primitive} we know that, up to shrinking $T$, for every $t\in T$ we have that $X_{t}$ is a (K\"ahler) primitive symplectic variety. Since $\alpha\colon\mathcal{X}\to T$ is locally trivial, the regular loci of each fibre $X_t$ fits together to form a flat fibration $\alpha_{\operatorname{reg}}\colon\mathcal{X}_{\operatorname{reg}}\to T$ that locally on $T$ is topologically trivial, by Thom's First Isotopy Lemma (\cite[Theorem~6.5]{Dimca}). In particular, for any $t,t'\in T$, the fundamental groups $\pi_1(X_{t,\operatorname{reg}})$ and $\pi_1(X_{t',\operatorname{reg}})$ are isomorphic and trivial: hence there are no non-trivial finite quasi-étale covers of $X_t$ and the proof is completed.
\end{proof}

We can drop the strong hypothesis on the fundamental group of the smooth locus, provided that the singularities are mild enough.

\begin{prop}\label{prop:small when terminal}
    Let $X$ be an irreducible symplectic variety with at most terminal singularities. Then any small locally trivial deformation of $X$ is an irreducible symplectic variety with at most terminal singularities.
\end{prop}
\begin{proof}
    Consider again a locally trivial deformation $\alpha\colon\mathcal{X}\to T$ of $X$, and let $0\in T$ be such that the fiber $X_{0}$ of $\alpha$ over $0$ is isomorphic to $X$. Up to shrinking $T$, again by Proposition~\ref{prop:small is primitive} we can suppose that any fibre $X_t$ is a primitive symplectic variety.  We need to prove that if $f_{t}\colon Y_{t}\to X_{t}$ is a finite quasi-étale covering, then the dimension $h^{[p],0}(Y_{t})$ of $\oH^{0}(Y_{t},\Omega_{Y_{t}}^{[p]})$ is 0 if $p$ is odd, and 1 if $p$ is even. 
    
    First of all, as we saw in the proof of Proposition~\ref{prop:small when simply connected}, the fundamental groups of the smooth loci of the fibres of $\alpha$ are all isomorphic. Moreover, as the isomorphism classes of finite quasi-étale coverings of $X_{t}$ are prescribed by the subgroups of $\pi_{1}(X^{\operatorname{reg}}_{t})$, it is enough to prove the result when $f_t$ arises from a deformation of a finite quasi-étale covering (we recall that by definition the total space of a quasi-étale covering is normal).

    We make the following claim.
    \begin{claim}
       Let $f_0\colon Y_0\to X_0$ be a finite quasi-étale covering of $X_0$; then there exist a locally trivial family $\beta\colon\mathcal{Y}\to T$ and a proper morphism $f\colon\mathcal{Y}\to\mathcal{X}$ such that $Y_t=\beta^{-1}(t)$ is normal and $f_t\colon Y_t\to X_t$ is a finite quasi-étale covering, for every $t\in T$. 
    \end{claim}
        
    Let us first see how the claim would finish the proof.
    First of all, since $\beta$ is a locally trivial family, up to shrinking $T$ we can take a relative resolution of singularities $\pi\colon\widetilde{\mathcal{Y}}\to\mathcal{Y}$ such that $\widetilde{Y}_{t}$ is K\"ahler for all $t\in T$ (by \cite[Proposition~5]{Namikawa:ext}), and hence the function mapping $t\in T$ to $h^{p,0}(\widetilde{Y}_{t})$ is constant. By \cite[Theorem~2.4]{GKP} there is an equality $h^{[p],0}(Y_{t})=h^{p,0}(\widetilde{Y}_{t})$, so that $h^{[p],0}(Y_{t})$ is also constant along $T$. 
    It follows that 
    \[ \dim \oH^{0}(Y_{t},\Omega_{Y_{t}}^{[p]})=\dim \oH^{0}(Y,\Omega_{Y}^{[p]})=\left\{
    \begin{array}{ll} 
    1, & p\in 2\mathbb{N}\\ 0, & p\notin 2\mathbb{N}
    \end{array}\right. \]
    where the last equality comes from the fact that $X_0=X$ is an irreducible symplectic variety and $Y_0=Y$ is a finite quasi-étale covering of $X_0$.
    
    \proof[Proof of the Claim.] First of all, since the family $\alpha\colon\mathcal{X}\to T$ is locally trivial and the central fibre $X_0$ is terminal by hypothesis, all the other fibres $X_t$ are terminal. In particular, if $\alpha_{\operatorname{reg}}\colon\mathcal{X}_{\operatorname{reg}}\to T$ is the smooth family of the smooth loci, then $\mathcal{X}\setminus\mathcal{X}_{\operatorname{reg}}$ has codimension at least $3$ (in fact Namikawa proves in \cite[Corollary~1]{Namikawa:Terminal} that the codimension of the singular locus is always at least $4$, but we will not need this fact). Now, since the family $\alpha_{\operatorname{reg}}\colon\mathcal{X}_{\operatorname{reg}}\to T$ is smooth (and the fundamental group of the fibres is constant), there exists another smooth family $\beta'\colon\mathcal{Y}'\to T$ and a proper étale cover $f'\colon\mathcal{Y}'\to\mathcal{X}_{\operatorname{reg}}$.

    The sheaf $f'_*\mathcal{O}_{\mathcal{Y}'}$ is a locally free sheaf of algebras on $\mathcal{X}_{\operatorname{reg}}$ and, if we denote by $j\colon\mathcal{X}_{\operatorname{reg}}\to\mathcal{X}$ the open immersion, by \cite[Theorem~2]{Siu} the sheaf $j_*f'_*\cO_{\mathcal{Y}'}$ is a coherent sheaf of $\cO_{\mathcal{X}}$-algebras. Let us define
    \[ f\colon\mathcal{Y}=\underline{\operatorname{Spec}}_{\mathcal{X}}(j_*f'_*\cO_{\mathcal{Y}'})\to\mathcal{X} \]
    and set $\beta:=\alpha\circ f$. 
    For every $t\in T$ the morphism $f_t\colon Y_t \to X_t$ is a finite morphism étale in codimension $1$: the morphism $f_t$ is not necessary quasi-étale because $Y_t$ can be non-normal in general.
    We need to show that $\beta\colon\mathcal{Y}\to T$ is locally trivial.

    Let $y\in\mathcal{Y}$ be a point and consider $x=f(y)\in\mathcal{X}$. Since $\alpha\colon\mathcal{X}\to T$ is locally trivial, there exist three open subsets $V_x\subset\mathcal{X}$, $V_{\alpha(x)}\subset T$ and $U_x\subset F=\alpha^{-1}(\alpha(x))$ such that $V_x\cong V_{\alpha(x)}\times U_x$. If we put $V^0_x=(V_x\cap\mathcal{X}_{\operatorname{reg}})$, then $V_x^0=V_{\alpha(x)}\times U_x^0$, where $U_x^0=(U_x\cap\mathcal{X}_{\operatorname{reg}})$. Moreover putting $\widetilde{V}_{x}^{0}:=(f')^{-1}(V_{x}^{0})$ and $\widetilde{U}_{x}^{0}:=(f')^{-1}(U_{x}^{0})$, we have the following commutative diagram:

    \begin{equation}\label{eqn:non ce la faccio piu'} 
    \xymatrix{
    V_{\alpha(x)}\times\widetilde{U}^0_x=\widetilde{V}^0_x\ar@{->}[r]\ar@{->}[d]^-{f'_x} &  R=\underline{\operatorname{Spec}}(j_{x,*}f'_{x,*}\cO_{\widetilde{V}^0_x})\ar@{->}[d]  \\
    V_{\alpha(x)}\times U_x^0=V_x^0\ar@{->}[r]^-{j_x}\ar@{->}[d]^-{p_2^0} & V_x=V_{\alpha(x)}\times U_x\ar@{->}[d]^-{p_2}  \\
    U_x^0\ar@{->}[r]^-{j_x^0} & U_x 
    } 
    \end{equation}
    
where the horizontal morphisms are the natural inclusions, and each square is cartesian.

Since $\widetilde{V}^0_x\cong V_{\alpha(x)}\times\widetilde{U}^0_x$, we have that 
    \begin{equation}\label{eqn:urca} f'_{x,*}\cO_{\widetilde{V}^0_x}\cong(p_2^0)^*(f'_{U_x^0,*}\cO_{\widetilde{U}^0_x}), 
    \end{equation}
    where $f'_{U_x^0}\colon \widetilde{U}^0_x\to U_x^0$ is the restriction of $f'_x$ to the second factor.

    We will now prove that there exists a sheaf of algebras $G$ on $U_x$ such that $j_{x,*}f'_{x,*}\cO_{\widetilde{V}^0_x}\cong p_2^* G$. Since $R=\underline{\operatorname{Spec}}(j_{x,*}f'_{x,*}\cO_{\widetilde{V}^0_x})$, this implies that $R=V_{\alpha(x)}\times \widetilde{U}_x$, where $V_{\alpha(x)}=V_{\beta(y)}\subset\beta^{-1}(\beta(y))$ is open. By using the equality (\ref{eqn:urca}), since the bottom square of the diagram (\ref{eqn:non ce la faccio piu'}) is cartesian we have the following chain of equalities,
    \begin{align*}
        j_{x,*}f'_{x,*}\cO_{\widetilde{V}^0_x} & \cong j_{x,*}(p_2^0)^*(f'_{U_x^0,*}\cO_{\widetilde{U}^0_x}) \\
        & \cong p_2^*(j^0_{x,*}f'_{U_x^0,*}\cO_{\widetilde{U}^0_x}).
    \end{align*}
    Therefore, putting $G=j^0_{x,*}f'_{U_x^0,*}\cO_{\widetilde{U}^0_x}$ concludes the proof that $\beta\colon\mathcal{Y}\to T$ is locally trivial.

    Finally, to conclude the proof of the claim, let us notice that by the local triviality of the family $\beta$, if we take the relative normalisation $\nu\colon\mathcal{Y}^{\operatorname{nor}}\to \mathcal{Y}$, then the family $\beta^{\operatorname{nor}}=\beta\circ\nu\colon\mathcal{Y}^{\operatorname{nor}}\to T$ remains locally trivial and the composition $\mathcal{Y}^{\operatorname{nor}}\to\mathcal{X}$ is still a finite quasi-étale covering; the same holds also for the fibers.
    \end{proof}

We wish to remark that singularities are supposed to be terminal in Proposition~\ref{prop:small when terminal} only in order to deduce that the singular locus has codimension at least $3$: this allows us to apply Siu's theorem about the coherency of the pushforward in the analytic category and get that $j_*f'_*\mathcal{O}_{\mathcal{Y}'}$ is coherent. In the algebraic category the analog of Siu's theorem holds assuming that the singular locus has codimension at least $2$: it follows that in the projective setting the statement of Proposition~\ref{prop:small when terminal} holds in the general case of canonical singularities. 

\begin{prop}
    Let $X$ be a projective irreducible symplectic variety. Let $f\colon\mathcal{X}\to T$ be a locally trivial family of primitive symplectic varieties as in Definition~\ref{def:lt} such that $f$ is projective and $T$ is quasi-projective, and let $\bar{t}\in T$ a point such that $\mathcal{X}_{\bar{t}}=X$. Then there exists an analytic open neighborhood $U\subset T$ of $\bar{t}$ such that, for every $t\in U$, the fibre $\mathcal{X}_t$ of $f$ is a projective irreducible symplectic variety.
\end{prop}

\begin{rmk}\label{rmk:strat}
We remark that if $p\colon \cX\to T$ is a locally trivial family of symplectic varieties, there is a natural relative stratification 
\[ \cX=\cX_0\supset \cX_1\supset\cdots\supset \cX_m, \]
where for every $t\in T$
\[ \cX_t\supset \cX_{1,t}\supset\cdots\supset \cX_{m,t} \]
is the stratification in Proposition~\ref{prop:stratification} for the variety $\cX_t$, and for every $i=0,\cdots,m$, the restriction $p_i\colon\cX_i\to T$ is a locally trivial family. By construction, the stratum $X_m$ is smooth and, if $\mathcal{X}_{m,t}$ is irreducible and of strictly positive dimension, then $p_m\colon \cX_m\to T$ is a smooth family of irreducible holomorphic symplectic manifolds.
\end{rmk}

\subsection{Locally trivial monodromy operators}\label{section:pto}

We will now define the monodromy group of a primitive symplectic variety, following \cite{BakkerLehn:PrimitiveSymplecticVarieties}.  

First of all we recall that if $X$ is a primitive symplectic variety, the torsion free part $\oH^2(X,\ZZ)_{\operatorname{tf}}$ of the second integral cohomology group of $X$ is endowed with a nondegenerate bilinear form $q_X$ of signature $(3,b_2(X)-3)$ (see \cite[Section~5.1, Lemma~5.7]{BakkerLehn:PrimitiveSymplecticVarieties}), that will be called \textit{Beauville--Bogomolov--Fujiki form} of $X$. The pair $(\oH^2(X,\ZZ)_{\operatorname{tf}},q_X)$ will be called the \emph{Beauville--Bogomolov--Fujiki lattice} of $X$. In particular, if $X$ is irreducible symplectic, then $\oH^{2}(X,\mathbb{Z})$ is torsion free and therefore endowed with the Beauville--Bogomolov--Fujiki form.

The Beauville--Bogomolov--Fujiki form is a locally trivial deformation invariant of a primitive symplectic variety (see \cite[ Lemma~5.7]{BakkerLehn:PrimitiveSymplecticVarieties}). In particular, this implies that if $X_{1}$ and $X_{2}$ are two locally trivial deformation equivalent primitive symplectic varieties, there is an isometry between $\oH^{2}(X_{1},\mathbb{Z})_{\operatorname{tf}}$ and $\oH^{2}(X_{2},\mathbb{Z})_{\operatorname{tf}}$. We will use the notation $\Or(\oH^{2}(X_{1},\mathbb{Z})_{\operatorname{tf}},\oH^{2}(X_{2},\mathbb{Z})_{\operatorname{tf}})$ for the set of isometries from the Beauville--Bogomolov--Fujiki lattice of $X_1$ to the one of $X_2$; similarly we use the notation $\Or(\oH^2(X,\ZZ)_{\operatorname{tf}})$ for the group of isometries from the Beauville--Bogomolov--Fujiki lattice of $X$ to itself.

The isometries arising from locally trivial families will be called locally trivial parallel transport operators. More precisely, we recall the following definition.

\begin{defn}\label{defn:pto}
Let $X$, $X_1$ and $X_2$ be primitive symplectic varieties.
\begin{enumerate}
\item An isometry $g\in\Or(\oH^{2}(X_{1},\mathbb{Z})_{\operatorname{tf}},\oH^{2}(X_{2},\mathbb{Z})_{\operatorname{tf}})$ is a \emph{locally trivial parallel transport operator from $X_{1}$ to $X_{2}$} if there exists a locally trivial family $p\colon\cX\to T$ of primitive symplectic varieties and two points $t_1,t_2\in T$, with $\cX_{t_i}:=p^{-1}(t_i)=X_i$, such that $g$ is the parallel transport along a path from $t_1$ to $t_2$ in the local system $R^2p_*\ZZ$.
\item An isometry $g\in\Or(\oH^2(X,\ZZ)_{\operatorname{tf}})$ is a \emph{locally trivial monodromy operator} if it is a locally trivial parallel transport operator from $X$ to itself.
\end{enumerate}
\end{defn}

If $X_1$ and $X_2$ are two primitive symplectic varieties, we will let 
\[ \mathsf{PT}^{2}_{\operatorname{lt}}(X_1,X_2)\subset\Or(\oH^{2}(X_1,\mathbb{Z})_{\operatorname{tf}},\oH^{2}(X_2,\mathbb{Z})_{\operatorname{tf}}) \]
be the set of locally trivial parallel transport operators from $X_1$ to $X_2$.

If $X_1=X_2=X$, we put 
\[ \Mon^{2}_{\operatorname{lt}}(X):=\operatorname{PT}^{2}_{\operatorname{lt}}(X,X), \] which is then the subset of $\Or(\oH^{2}(X,\mathbb{Z})_{\operatorname{tf}})$ given by all locally trivial monodromy operators. 

\begin{lemma}\label{lemma:Mon is group}
Let $X$ be a primitive symplectic variety. The set $\Mon^2_{\operatorname{lt}}(X)$ of locally trivial monodromy operators is a subgroup of $\Or(\oH^2(X,\ZZ)_{\operatorname{tf}})$.
\end{lemma}

\begin{proof}
This follows as in the smooth case, see \cite[footnote~3]{Markman:Survey}.
\end{proof}

The group $\Mon^2_{\operatorname{lt}}(X)$ is called the \emph{locally trivial monodromy group} of $X$ and, by construction, it is a locally trivial deformation invariant. If $X$ is smooth, we will simply write $\Mon^{2}(X)$ for the monodromy group, as all smooth deformations of $X$ are locally trivial in this case.

The group $\Or(\oH^2(X,\ZZ)_{\operatorname{tf}})$ contains the subgroup $\Or^+(\oH^2(X,\ZZ)_{\operatorname{tf}})$ of \emph{orientation preserving isometries}. We refer to \cite[Section~4]{Markman:Survey} and \cite{MirandaMorrison} for a general account on orientations, but let us quickly recall the definition we use.

If $\widetilde{\mathcal{C}}_X\subset\oH^2(X,\RR)$ is the set of classes $\alpha$ such that $q_X(\alpha)>0$, then by \cite[Lemma~4.1]{Markman:Survey} we have $\oH^2(\widetilde{\mathcal{C}}_X,\ZZ)=\ZZ$, and $\Or^+(\oH^2(X,\ZZ)_{\operatorname{tf}})$ is by definition the subgroup of $\Or(\oH^{2}(X,\mathbb{Z})_{\operatorname{tf}})$ given by those isometries acting as the identity on $\oH^2(\widetilde{\mathcal{C}}_X,\ZZ)$. In general we will refer to a generator of $\oH^2(\widetilde{\mathcal{C}}_X,\ZZ)$ as an \emph{orientation}: since for any positive three-dimensional subspace $W$ of $\oH^2(X,\RR)$ the space $W\setminus\{0\}$ is deformation retract of $\widetilde{\mathcal{C}}_X$, it follows that an orientation is nothing else than an orientation on $W$.

Among the orientation preserving isometries, we find all the locally trivial monodromy operators, as the following shows.

\begin{lemma}\label{lemma:Mon is O+}
Let $X$ be a primitive symplectic variety. Then we have an inclusion $\Mon^2_{\operatorname{lt}}(X)\subset\Or^+(\oH^2(X,\ZZ)_{\operatorname{tf}})$.
\end{lemma}

\begin{proof}
The proof follows as in the classical case, and we sketch it here for the reader's convenience. 

As in the smooth case (see \cite[Corollary~8]{NamProj}) 
if $\omega\in\oH^1(\Omega^{[1]}_X)$ is a K\"ahler class and $\sigma\in\oH^0(\Omega^{[2]}_X)$ is the symplectic form, then the three-space $W$ spanned by $\omega$, $\Re(\sigma)$ and $\Im(\sigma)$ is positive definite. Notice that the choice of this basis (in this order) determines an orientation, i.e.\ a generator $\mathsf{a}$ of $\oH^2(W\setminus\{0\},\ZZ)=\oH^2(\widetilde{\mathcal{C}}_X,\ZZ)$, which does not change if we replace $\omega$ by another K\"ahler class $\omega'$ and $\sigma$ by another symplectic form $\sigma'$. 

Now if $p\colon\mathcal{X}\to T$ is any locally trivial family of primitive symplectic varieties, then $\omega$ and $\sigma$ extend to sections of the respective local systems. This gives rise to a section $t\mapsto\mathsf{a}_t\in\oH^2(\mathcal{C}_{X_t},\ZZ)$, where $\mathsf{a}_t$ is the generator determined by $\omega_t$, $\Re(\sigma_t)$ and $\Im(\sigma_t)$.

 In particular the induced action on $\oH^2(\widetilde{\mathcal{C}}_X,\ZZ)$ is the identity and we are done.
\end{proof}

Among locally trivial parallel transport operators one finds the pull-back by a birational morphism between projective primitive symplectic varieties. This result is well-known for irreducible holomorphic symplectic manifolds, and follows from Huybrechts' results (see \cite{Huybrechts:KahlerCone} and \cite[Section~3.1]{Markman:Survey}).  

\begin{prop}[\protect{\cite[Theorem~6.16]{BakkerLehn:PrimitiveSymplecticVarieties}}]
\label{prop:iso is mon}
Let $X$ and $Y$ be two projective primitive symplectic varieties and $f\colon X\dashrightarrow Y$ a birational morphism. Suppose that $f$ is defined in codimension $1$ and that $f^*\colon\Pic(Y)\otimes\mathbb{Q}\to\Pic(X)\otimes\mathbb{Q}$ is an isomorphism. Then the pullback 
\[ f^*\colon\oH^2(Y,\ZZ)_{\operatorname{tf}}\longrightarrow\oH^2(X,\ZZ)_{\operatorname{tf}} \]
is well defined and a locally trivial parallel transport operator.
\end{prop}

\begin{proof}
This statement is implicit in \cite[Theorem~6.16]{BakkerLehn:PrimitiveSymplecticVarieties}, see \cite[Section~3.1]{Markman:Survey} for a rigorous proof.
\end{proof}

We conclude this section by recalling the notion of polarised parallel transport operators. Recall that a polarised variety is a pair $(X,H)$ where $X$ is a projective variety and $H$ is a polarisation, i.e.\ an ample line bundle on $X$. If $X$ is a projective symplectic variety (resp.\ primitive or irreducible), then the pair $(X,H)$ is a polarised symplectic variety (resp.\ primitive or irreducible).

\begin{defn}\label{defn:ppto}
    Let $(X,H)$, $(X_1,H_1)$ and $(X_2,H_2)$ be polarised primitive symplectic varieties.
    \begin{enumerate}
        \item An isometry $g\in\Or(\oH^2(X_1,\ZZ)_{\operatorname{tf}},\oH^2(X_2,\ZZ)_{\operatorname{tf}})$ is a \emph{polarised locally trivial parallel transport operator} from $(X_{1},H_{1})$ to $(X_{2},H_{2})$ if it is a parallel transport operator arising from a family $p\colon\mathcal{X}\to T$ as in Definition~\ref{defn:pto} such that there exists a relatively ample line bundle $\mathcal{H}$ on $\mathcal{X}$ such that $\mathcal{H}_{t_i}=H_i$ on $\mathcal{X}_{t_i}=X_i$, for $t_1,t_2\in T$.

        \item An isometry $g\in\Or(\oH^2(X,\ZZ)_{\operatorname{tf}})$ is a \emph{polarised locally trivial monodromy operator} if it is a polarised parallel transport operator from $(X,H)$ to itself.
    \end{enumerate}
\end{defn}
As before we denote by 
\[ \mathsf{PT}^2_{\operatorname{lt}}((X_1,H_1),(X_2,H_2)) \]
the set of polarised locally trivial parallel transport operators from $(X_1,H_1)$ to $(X_2,H_2)$ and by 
\[ \Mon^2_{\operatorname{lt}}(X,H)=\mathsf{PT}^2_{\operatorname{lt}}((X,H),(X,H)) \]
the group of polarised locally trivial monodromy operators of $(X,H)$.

Notice that by definition we have
\[ \mathsf{PT}^2_{\operatorname{lt}}((X_1,H_1),(X_2,H_2))\subset\mathsf{PT}^2_{\operatorname{lt}}(X_1,X_2) \]
and
\[ \Mon^2_{\operatorname{lt}}(X,H)\subset\Mon^2_{\operatorname{lt}}(X). \]

\begin{rmk}\label{rmk:fino ad hol}
    Definition~\ref{defn:ppto} also appears in literature in a different form. Let $p\colon\mathcal{X}\to T$ be a locally trivial family as in Definition~\ref{defn:pto}: instead of asking for the existence of a relatively ample line bundle $\mathcal{H}$ on $\mathcal{X}$, one can ask for the existence of a flat section $h$ of $R^2p_*\ZZ$ such that $h_t$ is the class of an ample line bundle on $\mathcal{X}_t$ for every $t\in T$ (see for example \cite[Definition~1.1.(4)]{Markman:Survey}).

    Of course, if $\mathcal{H}$ is a relatively ample line bundle on $\mathcal{X}$, then the section $c_1(\mathcal{H})$ satisfies the last condition. The vice versa is not true in general.

    Nevertheless, we wish to point out that the parallel transport operators arising from these two definition are in fact the same, providing that the base is at least normal and quasi-projective. 

    Let us start with a locally trivial family $f\colon\mathcal{X}\to T$, where $T$ is normal, irreducible and quasi-projective; suppose that there is a flat section $h$ of $R^2p_*\ZZ$ such that $h_t$ is the class of an ample line bundle for every $t\in T$. By the Lefschetz hyperplane section Theorem (see for example \cite[Theorem~1.1.(B)]{FL:Connectivity}), for a generic complete intersection curve $C\subset T$ the morphism
    \[ \pi_1(C)\longrightarrow\pi_1(T) \]
    is surjective:  hence every parallel  transport operator induced by $f$ between fibers over  $C$ is also induced by the restriction of $f$ over $C$. Without loss of generality, we can suppose that $C$ is irreducible and smooth. From now on we  work with the restricted family $f\colon\mathcal{X}\to C$ and show the existence of a holomorphic line bundle on its total space inducing the section $h$ over $C$. 

    By applying the functor $f_*$ to the exponential sequence
    \[ 0\to\ZZ\to\mathcal{O}_{\mathcal{X}}\to\mathcal{O}^*_{\mathcal{X}}\to 0 \]
    we obtain
    \[ 0=R^1f_*\cO_{\mathcal{X}}\to R^1f_*\cO^*_{\mathcal{X}}\to R^2f_*\ZZ\to R^2f_*\cO_{\mathcal{X}}, \]
    where in the first equality we used that all the fibres are primitive symplectic varieties. 
    Taking global sections we get
    \[ 0\to\oH^0(R^1f_*\cO^*_{\mathcal{X}})\to\oH^0(R^2f_*\ZZ)\to\oH^0(R^2f_*\cO_{\mathcal{X}}). \]
    Since the Hodge structure on $\oH^2(\mathcal{X}_t,\ZZ)$ is pure for every $t$, the map on the right hand side is the relative projection on the $(2,0)$-part.
    In particular the section $h\in\oH^0(R^2f_*\ZZ)$ is mapped to $0$ and it therefore comes from a section $\tilde{h}\in\oH^0(R^1f_*\cO^*_{\mathcal{X}})$.
        
    From the Leray spectral sequence we get 
    \[ \oH^1(\cO_{\mathcal{X}}^*)\to\oH^0(R^1f_*\cO^*_{\mathcal{X}})\to\oH^2(f_*\cO^*_{\mathcal{X}})=\oH^2(\cO^*_C)=0, \]
    where the last vanishing comes from the fact that $C$ is a curve. In particular this means that there exists a holomorphic line bundle $\mathcal{H}$ on $\mathcal{X}$ lifting the section $\tilde{h}$.
    Finally, since the composition
    \[ \oH^1(\mathcal{X},\cO_{\mathcal{X}}^*)\to\oH^0(R^1f_*\cO^*_{\mathcal{X}})\to\oH^0(R^2f_*\ZZ) \]
    is the morphism mapping a line bundle $\mathcal{L}$ on $\mathcal{X}$ to the section $\{c_1(\mathcal{L}_t)\}_t\in\oH^0(R^2f_*\ZZ)$ by construction, it follows that 
    $c_1(\mathcal{H}_t)=h_t$ for every $t\in T$. 
\end{rmk}

\begin{rmk}\label{rmk:da hol a alg}
    In Remark~\ref{rmk:fino ad hol} we 
 constructed a relatively ample  holomorphic line bundle $\mathcal{H}$ on the total space $\mathcal{X}$ of the family $f\colon \mathcal{X}\rightarrow C$ of irreducible symplectic varieties over the smooth irreducible curve $C$  lifting the section $h$ of $R^2p_*\ZZ$. We wish to point out though that if the family is smooth and algebraic, then $\mathcal{H}$ can be chosen  algebraic as well.

 If $\mathcal{X}$ and $C$ are projective this follows from GAGA  principle (\cite{GAGA}),  otherwise, as $C$ is quasi-projective, it is affine and 
  there exist  smooth and projective compactifications $\overline{\mathcal{X}}$ of $\mathcal{X}$ and $\overline{C}$ of $C$ and a projective morphism $\overline{f}\colon \overline{\mathcal{X}}\rightarrow \overline{C}$.
 
 Since for $t_0\in C$ the class  $h_{t_0}\in \oH^2(X_{t_0},\ZZ)$ is a monodromy invariant class, by the global Invariant Cycle Theorem (see for example \cite[Theorem~4.24]{ICT}), 
 there exists a class $\hat{h}\in\oH^2(\overline{\mathcal{X}},\QQ)$ extending $h_{t_0}$. Moreover, as $h_{t_0}\in \oH^{1,1}(X_{t_0},\ZZ)$ and $\hat{h}_t=\ell h_t$ (for some $\ell\in\QQ$) by construction 
 we may suppose that $\hat{h}$ is of Hodge type $(1,1)$ too. Moreover, by clearing the denominator, we find a class $\hat{h}'\in\oH^{1,1}(\overline{\mathcal{X}},\ZZ)$ such that $h'_{t_{0}}=nh_{t_{0}}$ for some $n\in\mathbb{N}$. This implies that there exists an algebraic line bundle $\mathcal{L}$ on the projective variety $\overline{\mathcal{X}}$ such that $c_{1}(\mathcal{L})=\hat{h}'$, and hence the restriction $c_{1}(\mathcal{L})_{t_{0}}$ of $c_{1}(\mathcal{L})$ to the fiber $\mathcal{X}_{t_{0}}$ is an integral multiple of $h_{t_0}$.
 
 Now let $\mathcal{L}_{C}$ be the restriction of $\mathcal{L}$ to $\mathcal{X}$. 
 Since the fibers of $f$ are simply connected for every $t\in C$ there is an isomorphism $\mathcal{L}_t\simeq \mathcal{H}^{\otimes n}_t$ between the fibres of $\mathcal{L}_C$ and $\mathcal{H}^{\otimes n}$.

 It follows that there exists a holomorphic line bundle $\mathcal{N}$ on $C$ and an isomorphism of holomorphic line bundles
 $$\mathcal{H}^{\otimes m} \simeq \mathcal{L}_{C}\otimes f^*\mathcal{N}.$$
 By \cite[Theorem~30.4]{For}, every holomorphic line bundle on an affine curve
is trivial. As a consequence the holomorphic line bundle $\mathcal{H}^{\otimes m}$ is algebraic, and this implies that $\mathcal{H}$ is algebraic as well.   
\end{rmk}

\begin{rmk}
    In the notation of Definition~\ref{defn:ppto}, the family $\{\mathcal{H}_t\}_{t\in T}$ is a continuous family of ample line bundles on the fibres of $p\colon\mathcal{X}\to T$. In particular any parallel transport operator must be constant on it, i.e.\ if $g\in\mathsf{PT}^2_{\operatorname{lt}}((X_1,H_1),(X_2,H_2))$ and $h_i=c_1(H_i)$, then 
    $g(h_1)=h_2$. In particular it follows that 
    \[ \Mon^2_{\operatorname{lt}}(X,H)\subset\Or^+(\oH^2(X,\ZZ)_{\operatorname{tf}})_h, \]
    where $h=c_1(H)$ and the latter is the group of orientation preserving isometries $g$ such that $g(h)=h$.
    More precisely $\Mon^2_{\operatorname{lt}}(X,H)\subset\Mon^2_{\operatorname{lt}}(X)_h$, where again the last group is the subgroup of monodromy operators fixing the polarisation. Arguing as in the proof of \cite[Corollary~7.4]{Markman:Survey}, one can further prove that 
    \[ \Mon^2_{\operatorname{lt}}(X,H)=\Mon^2_{\operatorname{lt}}(X)_h. \]
\end{rmk}


\subsection{Moduli spaces of sheaves on K3 surfaces}\label{section:moduli spaces}

We conclude this first section by recalling the basic facts we will need about moduli spaces of sheaves on K3 surfaces, and refer the reader to \cite{PR:SingularVarieties} and \cite{PR:v perp} for a more detailed exposition about this.

Let $S$ be a projective K3 surface. We denote by $\widetilde{\oH}(S,\ZZ)$ the \textit{Mukai lattice} of $S$: as a $\mathbb{Z}$-module it is $\oH^{\operatorname{even}}(S,\ZZ)$, and the (nondegenerate) integral quadratic form on it is given by 
\[ (r,\xi,a)^2=\xi^2-2ra, \]
where $r\in\oH^{0}(S,\ZZ)$, $\xi\in\oH^{2}(S,\ZZ)$ and $a\in\oH^{4}(S,\ZZ)$. The Mukai lattice of $S$ inherits a pure weight two Hodge structure from the one on $S$ by declaring $\widetilde{\oH}^{2,0}(S):=\oH^{2,0}(S)$. 

An element $v\in\widetilde{\oH}(S,\ZZ)$ is called a \textit{Mukai vector} if $v=(r,\xi,a)$ is such that $r\geq 0$ and $\xi\in \operatorname{NS}(S)$, and in the case when $r=0$ we have that either $\xi$ is the first Chern class of a strictly effective divisor, or $\xi=0$ and $a>0$. If $v$ is a Mukai vector on $S$, then there is a coherent sheaf $\mathcal{F}$ on $S$ such that $\ch(\mathcal{F})\cdot\sqrt{td(S)}=v$: we will then say that $v$ is the \textit{Mukai vector} of $\mathcal{F}$. We notice that a Mukai vector $v$  is of type $(1,1)$ with respect to the Hodge decomposition of the Mukai lattice.

\begin{defn}[\protect{\cite[Section~2.1.2]{PR:SingularVarieties}}]\label{defn:v-generic}
Given a 
Mukai vector $v\in\widetilde{\oH}(S,\ZZ)$ of the form $v=(r,\xi,a)$, an ample 
line bundle $H$ on $S$ is \emph{v-generic} if it verifies one of the following two conditions:
\begin{enumerate}
	\item If $r>0$, then for every $\mu_{H}$-semistable sheaf $E$ such that $v(E)=v$ and every $0\neq F\subseteq E$, we have that if $\mu_{H}(E)=\mu_{H}(F)$ then $c_{1}(F)/\rk(F)=c_{1}(E)/\rk(E)$. 
	\item If $r=0$, then for every $H$-semistable sheaf $E$ such that $v(E)=v$ and every $0\neq F\subseteq E$, if $\chi(E)/(c_{1}(E)\cdot H)=\chi(F)/(c_{1}(F)\cdot H)$ then $v(F)\in\mathbb{Q}v$. 
\end{enumerate}
\end{defn}

Given a Mukai vector $v\in\widetilde{\oH}(S,\ZZ)$ and a $v$-generic ample line bundle $H$ on $S$, we denote by $M_v(S,H)$ (or simply $M_v$ if the pair $(S,H)$ is clear from the context) the moduli space of Gieseker $H$-semistable sheaves $F$ on $S$ such that $v(F)=v$. 

\begin{rmk}\label{rmk:wall and chamber}
    Let $S$ be a projective K3 surface with Picard rank is at least 2, and let $v=(r,\xi,a)$ be a Mukai vector on $S$. If $r\neq 0$, or if $r=0$ and $a\neq 0$, 
    the ample cone of $S$ has a decomposition in $v$-walls and $v$-chambers (see \cite[Section~2.1.1]{PR:SingularVarieties}). As the $v$-chambers are subcones of the ample cone, it follows that if $H'=tH$ with $t\in\ZZ$, then $H'$ belongs to the same $v$-chamber as $H$. 
    
    By \cite[Lemma~2.9]{PR:SingularVarieties} the primitive polarisations lying in a $v$-chamber are all $v$-generic according to Definition \ref{defn:v-generic}. Moreover, by \cite[Lemma~2.9]{PR:SingularVarieties} if $H_1$ and $H_2$ are two $v$-generic polarisations (according to Definition~\ref{defn:v-generic}) that belong to the closure of the same $v$-chamber, then there is an identification of the moduli spaces
    $M_v(S,H_1)=M_v(S,H_2)$, meaning that a coherent sheaf $F$ of Mukai vector $v$ on $S$ is $H_{1}$-(semi)stable if and only if it is $H_{2}$-(semi)stable. 
\end{rmk}

From now on, we will always suppose that $v^{2}>0$: under this hypothesis, the $v$-genericity of $H$ implies that $M_v(S,H)\neq\emptyset$, and that it is an irreducible normal projective variety of dimension $v^2+2$ (\cite[Theorem~4.1]{KaledinLehnSorger}), whose smooth locus admits a holomorphic symplectic form (\cite{Mukai}).

For simplicity, we make now the following definition.

\begin{defn}
Given two strictly positive integers $m,k\in\mathbb{N}$, a triple $(S,v,H)$ will be called $(m,k)$-\textit{triple} if $S$ is a projective K3 surface, $v$ is a Mukai vector on $S$ such that if $v=mw$, with $w$ primitive, then $w^{2}=2k$, and $H$ is a $v$-generic polarisation.
\end{defn}

\begin{rmk}
\label{rmk:gen}
If $(S,v,H)$ is an $(m,k)$-triple and $v=mw$, then $(S,w,H)$ is a $(1,k)$-triple. The reason for this is that if $H$ is $v$-generic, then it is $w$-generic. Indeed, if $w=(r,\xi,a)$ with $r>0$, $E$ is a $\mu_{H}$-semistable sheaf with Mukai vector $w$ and $F\subseteq E$ is a proper coherent subsheaf with $\mu_{H}(F)=\mu_{H}(E)$, then $E^{\oplus m}$ is a $\mu_{H}$-semistable sheaf with Mukai vector $v$ and $F^{\oplus m}$ is a proper subsheaf of $E^{\oplus m}$ such that $\mu_{H}(F^{\oplus m})=\mu_{H}(E^{\oplus m})$. But since $H$ is $v$-generic it follows that $c_{1}(E^{\oplus m})/\rk(E^{\oplus m})=c_{1}(F^{\oplus m})/\rk(F^{\oplus m})$, and hence $c_{1}(E)/\rk(E)=c_{1}(F)/\rk(F)$. A similar proof works for $r=0$. 
\end{rmk}

For our purposes it is useful to introduce an equivalence relation that identifies $(m,k)$-triples whose associated moduli spaces represent the same sheaves.

\begin{defn}
\label{defn:congruent}
Two $(m,k)$-triples $(S_{1},v_{1},H_{1})$ and $(S_{2},v_{2},H_{2})$ are called \textit{congruent} if $S_{1}=S_{2}=S$, $v_{1}=v_{2}=v$ and a coherent sheaf $F$ of Mukai vector $v$ on $S$ is $H_{1}$-(semi)stable if and only if it is $H_{2}$-(semi)stable. In particular there is an equality $M_v(S,H_1)=M_v(S,H_2)$ and we denote by 
\[ \chi_{H_1,H_2}\colon M_v(S,H_1)\longrightarrow M_v(S,H_2),\qquad F\mapsto F \]
the identity morphism.
\end{defn}

\begin{rmk}
\label{rmk:triple congruenti}
If $H$ and $H'$ are two $v$-generic polarisations that lie in the closure of the same $v$-chamber, then $(S,v,H)$ and $(S,v,H')$ are congruent by Remark~\ref{rmk:wall and chamber}.
\end{rmk}

The following result is the starting point of our paper.

\begin{thm}[\protect{\cite[Theorem~1.10]{PR:SingularVarieties}}]
\label{thm:pr}
Let $(S,v,H)$ be an $(m,k)$-triple. Then $M_v(S,H)$ is an irreducible symplectic variety of dimension $2km^2+2$.
\end{thm}

\begin{rmk}\label{rmk:strat M}
In the setting of Theorem~\ref{thm:pr}, one may prove that the smallest stratum of the stratification of the singularities of $M_{v}(S,H)$ is isomorphic to $M_{w}(S,H)$. In particular we get a natural closed embedding
\[ i_{w,m}\colon M_w(S,H)\longrightarrow M_v(S,H). \]
This has been done in the proof of \cite[Theorem~4.4]{KaledinLehnSorger}, but let us recall the main idea.

The singular locus of $M_{v}(S,H)$ coincides with the locus of strictly semi-stable sheaves; moreover, any strictly semistable sheaf $F$ is S-equivalent to a sheaf of the form $F_1\oplus F_2$, where $F_i\in M_{m_iw}(S,H)$, with $m_{1}+m_{2}=m$. In particular it belongs to the image of the morphism 
\[ M_{m_1w}(S,H)\times M_{m_2w}(S,H)\to M_v(S,H), \quad (F_1,F_2)\mapsto [F_1\oplus F_2] \]
which is an irreducible component of the strictly semistable locus. The intersection of all these components is then the locus 
\[ Y=\{ E^{\oplus m}\in M_v(S,H)\mid E\in M_w(S,H)\}\cong M_w(S,H). \]
\end{rmk}

As we recalled in the previous section, since $M_v(S,H)$ is an irreducible symplectic variety, the cohomology group $\oH^2(M_v(S,H),\ZZ)$ is a free $\mathbb{Z}$-module with both a pure weight two Hodge structure and a lattice structure. These structures have been made explicit in \cite{PR:v perp}.

\begin{thm}[\protect{\cite[Theorem~1.6]{PR:v perp}}]\label{thm:PR v perp}
Let $(S,v,H)$ be an $(m,k)$-triple. 
Then there exists a Hodge isometry
\[ \lambda_{(S,v,H)}\colon v^\perp\longrightarrow\oH^2(M_v(S,H),\ZZ), \]
where $v^\perp$ inherits the Hodge and lattice structures from those of $\widetilde{\oH}(S,\ZZ)$, and $\oH^2(M_v(S,H),\ZZ)$ is endowed with the Beauville--Bogomolov--Fujiki form.
\end{thm} 

It is implicit in \cite{PR:v perp} that the isometries $\lambda_{(S,v,H)}$ behave well under deformations of moduli spaces induced by deformations of K3 surfaces. We will expand this comment more precisely in Remark~\ref{rmk:lambda in famiglie}, after we will have carefully defined deformations of $(m,k)$-triples.

For future reference, let us notice that the isometry $\lambda_{(S,v,H)}$ induces an isomorphism between the orthogonal groups by conjugation, 
\begin{align}\label{eqn:lambda sharp}
	\lambda_{(S,v,H)}^{\sharp}\colon \Or(v^\perp) & \longrightarrow \Or(\oH^{2}(M_{v}(S,H),\mathbb{Z})) \\
	g & \longmapsto \lambda_{(S,v,H)}\circ 
        g\circ(\lambda_{(S,v,H)})^{-1}. \nonumber
\end{align}
\medskip

By Remark~\ref{rmk:strat M} there is a closed embedding 
\[ i_{w,m}\colon M_{w}(S,H)\longrightarrow M_{v}(S,H), \] and we get a morphism
\[ i_{w,m}^*\colon\oH^2(M_{v}(S,H),\ZZ)\longrightarrow\oH^2(M_w(S,H),\ZZ). \]

Thanks to this morphism, we may now describe the relation between $\lambda_{(S,w,H)}$ and $\lambda_{(S,v,H)}$.

\begin{prop}\label{prop:i come lambda}
Let $(S,v,H)$ be an $(m,k)$-triple and write $v=mw$. Then $i^{*}_{w,m}\circ\lambda_{(S,v,H)}=m\lambda_{(S,w,H)}$.
\end{prop}

\begin{proof}
 In this proof, for every $p>0$ we use the shortened notation $M_{pw}$ for the moduli space $M_{pw}(S,H)$  and $\lambda_{pw}$ for the morphism $\lambda_{(S,pw,H)}$.
 
 By \cite[Section~4.2]{PR:v perp}, for every $p>0$, we have a morphism $$g_{p}\colon M_{pw}\times M_{w}\longrightarrow M_{(p+1)w},\,\,\,\,\,\,\,\,g_{p}(F,G):=F\oplus G.$$We let $f_{2}:=g_{1}\colon M_{w}\times M_{w}\to M_{2w}$, and for every $p\geq 3$ we let $$f_{p}:=g_{p-1}\circ(f_{p-1}\times \id_{M_{w}})\colon M_{w}^{p}\longrightarrow M_{pw},$$so that $f_{p}(F_{1},\cdots,F_{p}):=F_{1}\oplus\cdots\oplus F_{p}$. In particular, for $p=m$ we get a morphism $f_{m}\colon M_{w}^{m}\longrightarrow M_{v}$ such that $f_{m}(F_{1},\cdots,F_{m}):=F_{1}\oplus\cdots\oplus F_{m}$.

We claim that for every $p>0$ the following diagram
\begin{equation}
\label{eq:commperm}
\begin{CD}
	(pw)^{\perp}=w^{\perp}                                                  @>{\lambda_{pw}}>>             \oH^{2}(M_{pw},\ZZ)\\
	@V{(\lambda_{w},\cdots,\lambda_{w})}VV                                               @VV{f_{p}^{*}}V\\
	\oplus_{i=1}^{p}\oH^{2}(M_{w},\ZZ)  @>>{\sum_{i=1}^{p}\pi_{i,p}^{*}}>  \oH^{2}(M_{w}^{p},\ZZ)
\end{CD}
\end{equation}
is commutative, where $\pi_{i,p}\colon M_{w}^{p}\to M_{w}$ is the projection onto the $i$-th factor. We prove it by induction on $p$.

First of all we remark that by \cite[Proposition 4.11(2)]{PR:v perp} for every $d>0$ we have a commutative diagram
\begin{equation}
\label{eq:commgp}
\begin{CD}
	((d+1)w)^{\perp}=w^{\perp}=(dw)^{\perp}                                                       @>{\lambda_{(d+1)w}}>>             \oH^{2}(M_{(d+1)w},\ZZ)\\
	@V{(\lambda_{dw},\lambda_{w})}VV                                                @VV{g_{d}^{*}}V\\
	\oH^{2}(M_{dw},\ZZ)\oplus \oH^{2}(M_{w},\ZZ)  @>>{q_{1,d}^{*}+q_{2,d}^{*}}>  \oH^{2}(M_{dw}\times M_{w},\ZZ)
\end{CD}
\end{equation}
where $q_{1,d}\colon M_{dw}\times M_{w}\to M_{dw}$ and $q_{2,d}\colon M_{dw}\times M_{w}\to M_{w}$ are the two projections. For $d=1$ we then get a commutative diagram

$$\begin{CD}
	w^{\perp}                                                  @>{\lambda_{2w}}>>             \oH^{2}(M_{2w},\ZZ)\\
	@V{(\lambda_{w},\lambda_{w})}VV                                               @VV{f_{2}^{*}}V\\
	\oH^{2}(M_{w},\ZZ)\oplus \oH^{2}(M_{w},\ZZ)  @>>{\pi_{1,1}^{*}+\pi_{2,1}^{*}}>  \oH^{2}(M_{w}\times M_{w},\ZZ)
\end{CD}$$
that proves the commutativity of diagram (\ref{eq:commperm}) for $p=2$, i.e.\ for the initial step of the induction.

Let us now consider any $p\geq 2$, and notice that we have a commutative diagram
\begin{equation}
\label{eq:comm2}
	\begin{CD}
		\oH^{2}(M_{(p-1)w},\ZZ)\oplus \oH^{2}(M_{w},\ZZ)                                                       @>{q^{*}_{1,p-1}+q^{*}_{2,p-1}}>>             \oH^{2}(M_{(p-1)w,\ZZ}\times M_{w},\ZZ)\\
		@V{f_{p-1}^{*}\times \id_{\oH^{2}(M_{w},\ZZ)}}VV                                                @VV{(f_{p-1}\times \id_{M_{w}})^{*}}V\\
		\oH^{2}(M_{w}^{p-1},\ZZ)\oplus \oH^{2}(M_{w},\ZZ)  @>>{\operatorname{pr}_{1,p}^{*}+\operatorname{pr}_{2,p}^{*}}>  \oH^{2}(M^{p}_{w},\ZZ)
	\end{CD}
\end{equation}
where $\operatorname{pr}_{1,p}\colon M_{w}^{p}=M_{w}^{p-1}\times M_{w}\to M_{w}^{p-1}$ and $\operatorname{pr}_{2,p}\colon M_{w}^{p}=M_{w}^{p-1}\times M_{w}\to M_{w}$ are the two projections.

Putting diagram (\ref{eq:commgp}) (for $d=p-1$) and diagram (\ref{eq:comm2}) together, we get a commutative diagram 
\begin{equation}
	\label{eq:comm3}
	\begin{CD}
		(pw)^{\perp}=w^{\perp}                                                       @>{\lambda_{pw}}>>             \oH^{2}(M_{pw},\ZZ)\\
		@V{(f_{p-1}^{*}\circ\lambda_{(p-1)w},\lambda_{w})}VV                                                @VV{f_{p}^{*}}V\\
		\oH^{2}(M_{w}^{p-1},\ZZ)\oplus \oH^{2}(M_{w},\ZZ)  @>>{\operatorname{pr}_{1,p}^{*}+\operatorname{pr}_{2,p}^{*}}>  \oH^{2}(M^{p}_{w},\ZZ)
	\end{CD}
\end{equation}
By induction, the commutativity of diagram (\ref{eq:commperm}) for $p-1$ reads as $$f_{p-1}^{*}\circ\lambda_{(p-1)w}=\bigg(\sum_{i=1}^{p-1}\pi_{i,p-1}^{*}\bigg)\circ(\lambda_{w},\cdots,\lambda_{w}),$$so we get a commutative diagram
\begin{equation}
	\label{eq:comm4}
	\begin{CD}
		(pw)^{\perp}=w^{\perp}                                                       @>{\lambda_{pw}}>>             \oH^{2}(M_{pw},\ZZ)\\
		@V{(\lambda_{w},\cdots,\lambda_{w})}VV                                                @VV{f_{p}^{*}}V\\
		\oplus_{i=1}^{p}\oH^{2}(M_{w},\ZZ)  @>>{(\operatorname{pr}_{1,p}^{*}+\operatorname{pr}_{2,p}^{*})\circ((\sum_{i=1}^{p-1}\pi_{i,p-1}^{*})\times \id)}>  \oH^{2}(M^{p}_{w},\ZZ)
	\end{CD}
\end{equation}
Notice that $$(\operatorname{pr}_{1,p}^{*}+\operatorname{pr}_{2,p}^{*})\circ\bigg(\bigg(\sum_{i=1}^{p-1}\pi_{i,p-1}^{*}\bigg)\times \id_{\oH^{2}(M_{w})}\bigg)=\sum_{i=1}^{p}\pi_{i,p}^{*},$$hence diagram (\ref{eq:commperm}) is commutative for $p$, proving the claim.

In particular, diagram (\ref{eq:commperm}) is commutative for $p=m$. As a consequence, for every $a\in v^{\perp}=w^{\perp}$ we have 
$$f_{m}^{*}(\lambda_{v}(a))=\sum_{i=1}^{m}\pi_{i,m}^{*}(\lambda_{w}(a)).$$

Let us now consider $\delta\colon M_{w}\to M_{w}^{m}$ the diagonal morphism, mapping $[F]$ to $([F],\cdots,[F])$. We then have $i_{w,m}:=f_{m}\circ\delta\colon M_{w}\to M_{v}$.

Now, let $\alpha\in \oH^{2}(M_{v},\mathbb{Z})$, so that there is a unique $a\in v^{\perp}$ such that $\alpha=\lambda_{v}(a)$. Then we have 
$$i_{w,m}^{*}(\alpha)=\delta^{*}f_{m}^{*}(\lambda_{v}(a))=\delta^{*}\bigg(\sum_{i=1}^{m}\pi_{i,m}^{*}(\lambda_{w}(a))\bigg)=\sum_{i=1}^{m}\lambda_{w}(a)=m\lambda_{w}(a).$$
This concludes the proof.
\end{proof}

In the following corollary we use the notation $q_v$ (resp.\ $q_w$) for the Beauville--Bogomolov--Fujiki form on $\oH^2(M_v(S,H),\ZZ)$ (resp.\ $\oH^2(M_w(S,H),\ZZ)$).

\begin{cor}\label{cor:w and mw}
Let $(S,v,H)$ be an $(m,k)$-triple and write $v=mw$. Then the restriction morphism 
\[ i_{w,m}^*\colon\oH^2(M_{v}(S,H),\ZZ)\longrightarrow\oH^2(M_w(S,H),\ZZ) \]
is a similitude of lattices, i.e., more explicitly, for every $\alpha,\beta\in \oH^{2}(M_{v}(S,H),\mathbb{Z})$ we have 
\[ q_w(i^*_{w,m}(\alpha),i^*_{w,m}(\beta))=m^{2}\, q_v(\alpha,\beta). \]
\end{cor}

Proposition~\ref{prop:i come lambda} and Corollary~\ref{cor:w and mw} allow us to get an isomorphism between the ortoghonal groups $\Or(\oH^2(M_v(S,H),\ZZ))$ and $\Or(\oH^2(M_w(S,H),\ZZ))$.
To this purpose let us denote by 
    \[ i_{w,m,\QQ}^*\colon\oH^2(M_v(S,H),\QQ) \longrightarrow\oH^2(M_w(S,H),\QQ) \]
    the $\QQ$-linear extension of $i^*_{w,m}$ and  by 
    \begin{align*}
    i_{w,m,\QQ}^{\sharp}\colon\Or(\oH^2(M_v(S,H),\QQ)) & \longrightarrow\Or(\oH^2(M_w(S,H),\QQ)) \\
    g & \longmapsto i^*_{w,m,\QQ}\circ g \circ (i^*_{w,m,\QQ})^{-1}. 
    \end{align*}
    the induced isomorphism.

\begin{lemma}\label{lemma:i sharp}
    Let $(S,v,H)$ be an $(m,k)$-triple and write $v=mw$. Then:
    \begin{enumerate}
        \item $i_{w,m,\QQ}^{\sharp}$ sends integral isometries to integral isometries bijectively, i.e.\ it restricts to an isomorphism 
        \[ i_{w,m}^{\sharp}\colon\Or(\oH^2(M_v(S,H),\ZZ))\longrightarrow\Or(\oH^2(M_w(S,H),\ZZ)); \]
        \item using the equality $v^\perp=w^\perp$, we have that 
        \[i_{w,m}^\sharp =\lambda_{(S,w,H)}^\sharp\circ (\lambda_{(S,v,H)}^\sharp)^{-1}\colon \Or(\oH^2(M_v(S,H),\ZZ))\longrightarrow \Or(\oH^2(M_w(S,H),\ZZ)), 
        \]
        where $\lambda_{(S,v,H)}^\sharp$ and $\lambda_{(S,w,H)}^\sharp$ are the isomorphisms defined in the formula (\ref{eqn:lambda sharp}).
    \end{enumerate}
\end{lemma}
\begin{proof}
    For any $(l,k)$-triple $(S,u,H)$,  let $$\lambda_{(S,u,H),\QQ}\colon u^\perp\otimes\QQ\to \oH^2(M_u(S,H),\QQ)$$
    be the $\QQ$-linear extension of $\lambda_{(S,u,H)}$ and let $$\lambda_{(S,u,H),\QQ}^\sharp\colon \Or(v^\perp\otimes\QQ)\longrightarrow\Or(\oH^2(M_u(S,H),\QQ))$$ be the group ismorphism sending a $\QQ$-linear isometry $g\in \Or(v^\perp\otimes\QQ)$  to $\lambda_{(S,u,H),\QQ}\circ g\circ (\lambda_{(S,u,H),\QQ})^{-1}$ so that $\lambda_{(S,u,H)}^{
    \sharp}$
    is the restriction of $\lambda_{(S,v,H),\QQ}^\sharp$ to $\Or(v^\perp)$.
    
    We claim that 
    \[  i_{w,m,\QQ}^\sharp =\lambda_{(S,w,H),\QQ}\circ (\lambda_{(S,v,H),\QQ}^\sharp)^{-1}
    \]
    Since $v^\perp=w^\perp$ and  the isomorphism $\lambda_{(S,v,H)}^\sharp$ and $\lambda_{(S,w,H)}^\sharp$ are the restrictions of $\lambda_{(S,v,H),\QQ}^\sharp$ and $\lambda_{(S,w,H),\QQ}^\sharp$, 
    our claim implies that  $i_{w,m,\QQ}^{\sharp}$ sends $\Or(\oH^2(M_v(S,H),\ZZ))$ onto $\Or(\oH^2(M_w(S,H),\ZZ))$ and its restriction 
    \[ i_{w,m}^{\sharp}\colon\Or(\oH^2(M_v(S,H),\ZZ))\longrightarrow\Or(\oH^2(M_w(S,H),\ZZ)) \]
  is an isomorphism and satisfies   
  \[ i_{w,m}^{\sharp}=\lambda_{(S,w,H)}^\sharp\circ(\lambda_{(S,v,H)}^\sharp)^{-1}. \]

    Let us then prove the claim. By Proposition~\ref{prop:i come lambda} we know that $i_{w,m,\QQ}^*=m\lambda_{(S,w,H),\QQ}\circ\lambda_{(S,v,H),\QQ}^{-1}$; it follows that $(i_{w,m,\QQ}^*)^{-1}=\frac{1}{m}\lambda_{(S,v,H),\QQ}\circ\lambda_{(S,w,H),\QQ}^{-1}$. Therefore, by definition, we have that for every $g\in\Or(\oH^2(M_v(S,H),\QQ))$
    \begin{align*}
     i_{w,m,\QQ}^\sharp(g) & = i^*_{w,m,\QQ}\circ g \circ (i^*_{w,m,\QQ})^{-1} \\
       & =(m\lambda_{(S,w,H),\QQ}(\circ\lambda_{(S,v,H),\QQ}^{-1})\circ g \circ (\frac{1}{m}\lambda_{(S,v,H),\QQ}\circ\lambda_{(S,w,H),\QQ}^{-1}) \\
       & = \lambda_{(S,w,H),\QQ}\circ(\lambda_{(S,v,H),\QQ}^{-1}\circ g \circ \lambda_{(S,v,H),\QQ})\circ\lambda_{(S,w,H),\QQ}^{-1} \\
       & = \lambda_{(S,w,H),\QQ}^\sharp\circ(\lambda_{(S,v,H),\QQ}^\sharp)^{-1}(g),
    \end{align*}
    where in the third equality we used that $g$ is $\QQ$-linear. The claim follows and the proof is concluded.
\end{proof}

For later use, let us extend the last results to the case in which we have two $(m,k)$-triples $(S_1,v_1,H_1)$ and $(S_2,v_2,H_2)$. As usual, let us write $v_i=mw_i$, so that $(S_1,w_1,H_1)$ and $(S_2,w_2,H_2)$ are $(1,k)$-triples. For sake of simplicity, in the following we use the notation $M_{v_1}$ for $M_{v_1}(S_1,H_1)$ and $M_{v_2}$ for $M_{v_2}(S_2,H_2)$; similarly we write $M_{w_1}$ for $M_{w_1}(S_1,H_1)$ and $M_{w_2}$ for $M_{w_2}(S_2,H_2)$.

Let us consider the following bijective map of sets
\begin{align*}
	i^{\sharp}_{w_1,w_2,m,\QQ}\colon \Or(\oH^{2} 
        (M_{v_1},\mathbb{Q}),\oH^{2}(M_{v_2},\mathbb{Q})) & \longrightarrow \Or(\oH^{2}(M_{w_1},\mathbb{Q}),\oH^{2}(M_{w_2},\mathbb{Q})) \\
	g & \longmapsto i^{*}_{w_2,m,\QQ}\circ 
        g\circ(i^{*}_{w_1,m,\QQ})^{-1}. \nonumber
    \end{align*}

\begin{lemma}\label{lemma:i sharp gen} The bijection
$i_{w_1,w_2,m,\QQ}^{\sharp}$ sends integral isometries to integral isometries bijectively, i.e.\ it restricts to a bijection
        \[ i_{w_1,w_2,,m}^{\sharp}\colon \Or(\oH^{2} (M_{v_1},\mathbb{Z}),\oH^{2}(M_{v_2},\mathbb{Z})) \longrightarrow \Or(\oH^{2}(M_{w_1},\mathbb{Z}),\oH^{2}(M_{w_2},\mathbb{Z}))  
 \]
        
    More esplicity, we have
    \[ i^{\sharp}_{w_1,w_2,m}(g)=(\lambda_{(S_2,w_2,H_2)}\circ\lambda_{(S_2,v_2,H_2)})\circ g\circ(\lambda_{(S_1,w_1,H_1)}\circ\lambda_{(S_1,v_1,H_1)})^{-1} \]
    for every $g\in\Or(\oH^{2} 
        (M_{v},\mathbb{Z}),\oH^{2}(M_{v'},\mathbb{Z}))$.
\end{lemma}    
\begin{proof}
    Let $g\in\Or(\oH^{2}(M_{v},\mathbb{Z}),\oH^{2}(M_{v'},\mathbb{Z}))$ be an isometry and $g_\QQ$ its rational extension. As in the proof of Lemma~\ref{lemma:i sharp}, by using Proposition~\ref{prop:i come lambda} and the $\QQ$-linearity of $g_\QQ$, we can see that 
    \[ i^{\sharp}_{w_1,w_2,m,\QQ}(g_\QQ)=(\lambda_{(S_2,w_2,H_2),\QQ}\circ\lambda_{(S_2,v_2,H_2),\QQ})\circ g_\QQ\circ(\lambda_{(S_1,w_1,H_1),\QQ}\circ\lambda_{(S_1,v_1,H_1),\QQ})^{-1}. \]
    Since all the isometries in the right hand side of the equality are rational extensions of integral isometries, we get that also $i^{\sharp}_{w_1,w_2,m,\QQ}(g_\QQ)$ is a rational extension of an integral isometry. By construction this isometry is 
    \[ i^{\sharp}_{w_1,w_2,m}(g):=(\lambda_{(S_2,w_2,H_2)}\circ\lambda_{(S_2,v_2,H_2)})\circ g\circ(\lambda_{(S_1,w_1,H_1)}\circ\lambda_{(S_1,v_1,H_1)})^{-1}, \]
    from which the claim follows.
\end{proof}


\section{A groupoid representation}\label{section:groupoid}

Aim of this section is to define a groupoid $\mathcal{G}^{m,k}$ of $(m,k)$-triples and a $\mathcal{G}^{m,k}$-representation with values in a 
groupoid of free $\ZZ$-modules. 

The groupoid $\mathcal{G}^{m,k}$ will be defined starting from two groupoids, $\mathcal{G}^{m,k}_{\operatorname{def}}$ and $\mathcal{G}^{m,k}_{\operatorname{FM}}$, whose constructions will be explained in Sections~\ref{section:families of K3} and~\ref{section:FM}, respectively. These two groupoids will have the same objects; moreover, the morphisms in $\mathcal{G}^{m,k}_{\operatorname{def}}$ come from deformations of $(m,k)$-triples, while the morphisms in $\mathcal{G}^{m,k}_{\operatorname{FM}}$ come from Fourier-Mukai transforms.

We start by quickly recalling the main definitions and notation about groupoids we will use.

\subsection{Groupoids}\label{section:groupoids}
We refer to the lecture notes \cite{Higgins} for the definitions and constructions about groupoids.
First of all a groupoid $\mathcal{G}$ is a (small) category whose morphisms are all isomorphisms, i.e.\ for any two objects $x,y\in\mathcal{G}$, any $f\in\Hom_{\mathcal{G}}(x,y)$ is an isomorphism.
If $\mathcal{G}$ is a groupoid and $x\in\mathcal{G}$ is an object, we let
\[ \Aut_{\mathcal{G}}(x):=\Hom_{\mathcal{G}}(x,x), \]
often called the isotropy group of the object $x$ in $\mathcal{G}$.
If the groupoid $\mathcal{G}$ is clear from the context, then we will simply write $\Aut(x)$. 

Moreover, if $\mathcal{G}$ and $\mathcal{H}$ are two groupoids, and $F\colon\mathcal{G}\to\mathcal{H}$ is a functor, for every object $x$ of $\mathcal{G}$ we let $$F_{x}\colon\Aut_{\mathcal{G}}(x)\longrightarrow \Aut_{\mathcal{H}}(F(x))$$
be the group morphism mapping an automorphism of $x$ in $\mathcal{G}$ to its image under the functor $F$.

If $\mathcal{G}$ is a groupoid, by a \emph{representation} of $\mathcal{G}$ we mean a functor $F\colon\mathcal{G}\to\mathcal{A}$ where $\mathcal{A}$ is a suitable groupoid of $\ZZ$-modules (or vector spaces).

Finally, let us recall the notion of free product of groupoids. Let $\mathcal{G}$ and $\mathcal{H}$ be two groupoids, and $Y$ a set of common objects of $\mathcal{G}$ and $\mathcal{H}$. In all the situations of interest for us $\mathcal{G}$ and $\mathcal{H}$ will have the same set of objects.
\begin{defn}\label{defn:free product of groupoids}
    The \emph{free product} of $\mathcal{G}$ and $\mathcal{H}$ along $Y$ is the groupoid $\mathcal{G}\ast_Y\mathcal{H}$ such that:
    \begin{itemize}
        \item the objects of $\mathcal{G}\ast_Y\mathcal{H}$ are the elements of $Y$;
        \item if $x,y\in Y$, then a morphism $f\in\Hom_{\mathcal{G}\ast_Y\mathcal{H}}(x,y)$ is a formal combination (with usual cancellation properties)
        \[ f=g_1\circ\cdots\circ g_k \]
        where each $g_i$ is a morphism from $x_i\in Y$ to $x_{i+1}\in Y$ in either $\mathcal{G}$ or $\mathcal{H}$ and such that $x_1=x$ and $x_{k+1}=y$.
    \end{itemize}
    If $\mathcal{G}$ and $\mathcal{H}$ have the same sets of objects, then we simply refer to $\mathcal{G}\ast\mathcal{H}$ as the free product of $\mathcal{G}$ and $\mathcal{H}$, assuming that $Y$ is the whole set of objects.
\end{defn}
The fact that free products exist is the content of \cite[Proposition~21]{Higgins}. In literature the free product is also defined as the pushout, in the category of small categories, of $\mathcal{G}$ and $\mathcal{H}$ along a common set $Y$ of objects of $\mathcal{G}$ and $\mathcal{H}$.

\subsection{Deformations of $(m,k)$-triples and their groupoid}\label{section:families of K3}

We start by recalling the construction of a deformation of a moduli space of sheaves on a K3 surface induced by the deformation of the surface itself, following \cite{PR:Deformations,PR:SingularVarieties}. 

Let $(S,v,H)$ be an $(m,k)$-triple, and write $v=(r,\xi,a)$. We let $L$ be a line bundle on $S$ such that $c_1(L)=\xi$. We will moreover use the following notation: if $T$ is a smooth, connected algebraic variety, $f\colon Y\to T$ is a morphism and $\mathcal{L}\in \Pic(Y)$, for every $t\in T$ we let $Y_{t}:=f^{-1}(t)$ and $\mathcal{L}_{t}:=\mathcal{L}_{|Y_{t}}$.

\begin{defn}\label{defn:def of m k triple}
	Let $(S,v,H)$ be an $(m,k)$-triple, and $T$ a smooth, connected algebraic variety. A \textit{deformation of $(S,v,H)$ along $T$} is a triple $(f\colon \mathcal{S}\to T,\mathcal{L},\mathcal{H})$, where:
	\begin{enumerate}
		\item $f\colon \mathcal{S}\to T$ is a smooth, projective deformation of $S$, and we let $0\in T$ be such that $\mathcal{S}_{0}\simeq S$;
		\item $\mathcal{L}$ is a line bundle on $\mathcal{S}$ such that $\mathcal{L}_{0}\simeq L$.
		\item $\mathcal{H}$ is a line bundle on $\mathcal{S}$ such that $\mathcal{H}_{t}$ is a $v_{t}$-generic polarisation on $\mathcal{S}_{t}$ for every $t\in T$ and such that $\mathcal{H}_{0}\simeq H$
	\end{enumerate}
	where for every $t\in T$ we let $v_{t}:=(r,c_{1}(\mathcal{L}_{t}),a)$.
\end{defn}

\begin{rmk}
\label{rmk:defomk}
Notice that if $(f\colon \mathcal{S}\to T,\mathcal{L}^{\otimes m},\mathcal{H})$ is a deformation of an $(m,k)$-triple $(S,v,H)$ along $T$, and if we let $v=mw$, then $(f\colon \mathcal{S}\to T,\mathcal{L},\mathcal{H})$ is a deformation of $(S,w,H)$ along $T$. 

Conversely, if $(f\colon \mathcal{S}\to T,\mathcal{L},\mathcal{H})$ is a deformation of a $(1,k)$-triple $(S,w,H)$ along $T$, and if we let $v=mw$, then by \cite{PR:SingularVarieties} there is a Zariski-closed subset $Z$ of $T$ for which $(f\colon \mathcal{S}_{|T'}\to T',\mathcal{L}_{|T'}^{\otimes m},\mathcal{H}_{|T'})$ is a deformation of $(S,v,H)$ along $T':=T\setminus Z$ (here $Z$ is the subset of $T$ of those $t\in T$ for which $H_{t}$ is not $v_{t}$-generic). Notice that it may happen that $Z=T$.
\end{rmk}

If $(S,v,H)$ is an $(m,k)$-triple and $(f\colon\mathcal{S}\to T,\mathcal{L},\mathcal{H})$ is a deformation of $(S,v,H)$ along a smooth, connected algebraic variety $T$, we let $p_{v}\colon\mathcal{M}_{v}\to T$ be the relative moduli space of semistable sheaves so that for every $t\in T$ we have $\mathcal{M}_{v,t}=M_{v_{t}}(\mathcal{S}_{t},\mathcal{H}_{t})$. As shown in \cite[Lemma~2.21]{PR:SingularVarieties}, the family $p_{v}\colon\mathcal{M}_{v}\to T$ is a locally trivial deformation of $M_{v}(S,H)$ along $T$.

\begin{rmk}\label{rmk:lambda in famiglie}
    The isometry $\lambda_{(S,v,H)}$ in Theorem~\ref{thm:PR v perp} behaves well in families of $(m,k)$-triples, i.e.\ it extends to an isometry of local systems as we will now explain. 
    
    Let $(f\colon \mathcal{S}\to T,\mathcal{L},\mathcal{H})$ be a deformation of $(m,k)$-triples and $p_{v}\colon\mathcal{M}_{v}\to T$ the associated locally trivial family of moduli spaces.
   
    Let $\mathcal{M}_v^s$ be the smooth locus of $\mathcal{M}_v$ and and let $p_v^s:\mathcal{M}_v^s\rightarrow T$  be the restriction of $p_v$. By \cite[Proposition 3.5(2)]{PR:v perp}, the inclusion
    $\mathcal{M}_v^s\subset \mathcal{M}_v$ induces an isomorphism $\iota\colon R^2p_{v*}^s\ZZ
    \rightarrow  R^2p_{v*}\ZZ $
    of local systems. By
    \cite[Proposition 4.4(2)]{PR:v perp},  a relative quasi-universal family for $\mathcal{M}_v^s$ induces an isomorphism $\lambda^s\colon \mathsf{v}^\perp\rightarrow R^2p^{s}_{v*}\ZZ$  such that the morphism over a point $t\in T$ of the composition 
    \[ \mathsf{\lambda}:=\iota\circ \lambda^s\colon \mathsf{v}^\perp\longrightarrow R^2p_{v*}\ZZ \] is just $\lambda_{(S_t,v_t,H_t)}$.
\end{rmk}

\begin{defn}
Let $(S_{1},v_{1},H_{1})$ and $(S_{2},v_{2},H_{2})$ be two $(m,k)$-triples. A \textit{deformation path} from $(S_{1},v_{1},H_{1})$ to $(S_{2},v_{2},H_{2})$ is a $6$-tuple $$\alpha:=(f\colon\mathcal{S}\to T,\mathcal{L},\mathcal{H},t_{1},t_{2},\gamma)$$where: \begin{itemize}
    \item the triple $(f\colon\mathcal{S}\to  T,\mathcal{L},\mathcal{H})$ is a deformation of both the $(m,k)$-triples $(S_{1},v_{1},H_{1})$ and $(S_{2},v_{2},H_{2})$;
    \item for $i=1,2$ the point $t_{i}\in T$ is such that $(\mathcal{S}_{t_{i}},v_{t_{i}},H_{t_{i}})=(S_{i},v_{i},H_{i})$;
    \item we have that $\gamma$ is a continuous path in $T$ such that $\gamma(0)=t_{1}$ and $\gamma(1)=t_{2}$.
\end{itemize}
\end{defn}

Given two $(m,k)$-triples $(S_{1},v_{1},H_{1})$ and $(S_{2},v_{2},H_{2})$ and $\alpha=(f\colon\mathcal{S}\to T,\mathcal{L},\mathcal{H},t_{1},t_{2},\gamma)$ a deformation path from $(S_{1},v_{1},H_{1})$ to $(S_{2},v_{2},H_{2})$, let $p_{v}\colon\mathcal{M}_{v}\to T$ be the relative moduli space associated to the deformation $(f\colon\mathcal{S}\to T,\mathcal{L},\mathcal{H})$. There are two natural locally trivial parallel transport operators that can be defined starting from $\alpha$.
\begin{enumerate}
    \item The first one is the parallel transport operator $$p_{\alpha}\colon\widetilde{\oH}(S_{1},\mathbb{Z})\longrightarrow\widetilde{\oH}(S_{2},\mathbb{Z})$$along $\gamma$ inside the local system $R^{\bullet}f_{*}\mathbb{Z}$.
    \item The second one is the locally trivial parallel transport operator $$g_{\alpha}\colon\oH^{2}(M_{v_{1}}(S_{1},H_{1}),\mathbb{Z})\longrightarrow \oH^{2}(M_{v_{2}}(S_{2},H_{2}),\mathbb{Z})$$along $\gamma$ inside the local system $R^{2}p_{v*}\mathbb{Z}$.
\end{enumerate}

\begin{defn}\label{defn:equivalent}
Let $(S_{1},v_{1},H_{1})$ and $(S_{2},v_{2},H_{2})$ be two $(m,k)$-triples. Two deformations paths $\alpha$ and $\alpha'$ from $(S_{1},v_{1},H_{1})$ to $(S_{2},v_{2},H_{2})$ are \textit{equivalent} if $p_{\alpha}=p_{\alpha'}$ and $g_{\alpha}=g_{\alpha'}$. The equivalence class of $\alpha$ will be denoted $\overline{\alpha}$.
\end{defn}

\begin{rmk}\label{rmk:solo p alpha}
    In Definition~\ref{defn:equivalent} we used both the parallel transport operators $p_{\alpha}$ in the local system $R^\bullet f_*\ZZ$ and the parallel transport operators $g_{\alpha}$ in the local system $R^2p_{v*}\ZZ$. In fact, only $p_{\alpha}$ is needed. This is because the local system $R^\bullet f_*\ZZ$ comes with a flat section $\mathsf{v}$ corresponding to the Mukai vectors on the fibres. We can then consider the sub-local system 
    \[ \mathsf{v}^\perp\subset R^\bullet f_*\ZZ. \]
    The parallel transport operator $p_{\alpha}$ is constant along $\mathsf{v}$ by definition. The restriction $p_{\alpha}|_{v^\perp}$ can then be seen as the parallel transport operator inside the local system $\mathsf{v}^\perp$. Since by Remark~\ref{rmk:lambda in famiglie} there is an isomorphism of local systems 
    \[ \lambda\colon \mathsf{v}^\perp\longrightarrow R^2p_{v*}\ZZ, \] and $g_{\alpha}$ and $p_{\alpha}|_{v^\perp}$ are parallel transport operators over the same path $\gamma$, the morphism $g_{\alpha}$ is uniquely determined by $p_{\alpha}$ and vice versa.
\end{rmk}

Consider three $(m,k)$-triples $(S_{1},v_{1},H_{1})$, $(S_{2},v_{2},H_{2})$ and $(S_{3},v_{3},H_{3})$, and let $\alpha=(f\colon\mathcal{S}\to T,\mathcal{L},\mathcal{H},t_{1},t_{2},\gamma)$ be a deformation path from $(S_{1},v_{1},H_{1})$ to $(S_{2},v_{2},H_{2})$ and $\alpha'=(f'\colon\mathcal{S}'\to T',\mathcal{L}',\mathcal{H}',t'_{1},t'_{2},\gamma')$ a deformation path from $(S_{2},v_{2},H_{2})$ to $(S_{3},v_{3},H_{3})$.

\begin{defn}
The \textit{concatenation} of $\alpha$ with $\alpha'$ is the $6$-tuple $$\alpha'\star\alpha:=(f''\colon\mathcal{S}''\to T'',\mathcal{L}'',\mathcal{H}'',t''_{1},t''_{2},\gamma'')$$where
\begin{itemize}
    \item $T''$ is obtained by glueing $T$ and $T'$ along $t_{2}$ and $t'_{1}$
    \item $\mathcal{S}''$ is obtained by glueing $\mathcal{S}$ and $\mathcal{S}'$ along $\mathcal{S}_{t_{2}}$ and $\mathcal{S}'_{t'_{1}}$
    \item $f''$ is obtained by glueing $f$ and $f'$ along $\mathcal{S}_{t_{2}}$ and $\mathcal{S}'_{t'_{1}}$
    \item $\mathcal{L}''$ is obtained by glueing $\mathcal{L}$ and $\mathcal{L}'$ along $\mathcal{S}_{t_{2}}$ and $\mathcal{S}'_{t'_{1}}$
    \item $\mathcal{H}''$ is obtained by glueing $\mathcal{H}$ and $\mathcal{H}'$ along $\mathcal{S}_{t_{2}}$ and $\mathcal{S}'_{t'_{1}}$
    \item $t''_{1}$ is the image of $t_{1}$ in $T''$
    \item $t''_{2}$ is the image of $t'_{2}$ in $T''$
    \item $\gamma''$ is the concatenation of the image of the path $\gamma$ in $T''$ with the image of the path $\gamma'$ in $T''$.
\end{itemize}
\end{defn}

It is immediate to notice that if $\alpha$ is a deformation path from $(S_{1},v_{1},H_{1})$ to $(S_{2},v_{2},H_{2})$ and $\alpha'$ is a deformation path from $(S_{2},v_{2},H_{2})$ to $(S_{3},v_{3},H_{3})$, then $\alpha\star\alpha'$ is a deformation path from $(S_{1},v_{1},H_{1})$ to $(S_{3},v_{3},H_{3})$.

\begin{rmk}
\label{rmk:starok}
Notice that $$p_{\alpha\star\alpha'}=p_{\alpha'}\circ p_{\alpha},\,\,\,\,\,\,\,\,g_{\alpha\star\alpha'}=g_{\alpha'}\circ g_{\alpha},$$so if $\alpha$ is equivalent to $\beta$ and $\alpha'$ is equivalent to $\beta'$, then $\alpha\star\alpha'$ is equivalent to $\beta\star\beta'$.
\end{rmk}

The previous notions allow us to define the groupoid $\mathcal{G}^{m,k}_{\operatorname{def}}$ of deformations of $(m,k)$-triples. Before doing this, we need to introduce two groupoids $\widetilde{\mathcal{G}}^{m,k}_{\operatorname{def}}$ and $\mathcal{P}^{m,k}$ that we will use to define $\mathcal{G}^{m,k}_{\operatorname{def}}$.

We start with the groupoid $\widetilde{\mathcal{G}}^{m,k}_{\operatorname{def}}$.

\begin{defn}\label{defn:G tilde m k def}
Given two strictly positive integers $m$ and $k$, the groupoid $\widetilde{\mathcal{G}}^{m,k}_{\operatorname{def}}$ is defined as follows:
\begin{itemize}
    \item the objects of $\widetilde{\mathcal{G}}^{m,k}_{\operatorname{def}}$ are the $(m,k)$-triples;
    \item for every two $(m,k)$-triples $(S_{1},v_{1},H_{1}),(S_{2},v_{2},H_{2})$, a morphism from $(S_{1},v_{1},H_{1})$ to $(S_{2},v_{2},H_{2})$ in $\widetilde{\mathcal{G}}^{m,k}_{\operatorname{def}}$ is an equivalence class of deformation paths from $(S_{1},v_{1},H_{1})$ to $(S_{2},v_{2},H_{2})$;
    \item for every $(m,k)$-triple $(S,v,H)$, where $v=(r,c_{1}(L),a)$, the identity of $(S,v,H)$ is the equivalence class of the deformation path $(S\to\{p\},L,H,p,p,\kappa_{p})$ where $\kappa_{p}$ is the constant path;
    \item for every three $(m,k)$-triples $(S_{1},v_{1},H_{1})$, $(S_{2},v_{2},H_{2})$ and $(S_{3},v_{3},H_{3})$ and every pair of morphisms 
    \[ \overline{\alpha}\colon(S_{1},v_{1},H_{1})\to (S_{2},v_{2},H_{2})\quad\mbox{and}\quad\overline{\alpha'}\colon(S_{2},v_{2},H_{2})\to(S_{3},v_{3},H_{3}), \] 
    the composition of $\alpha$ with $\alpha'$ is $\overline{\alpha'}\circ\overline{\alpha}:=\overline{\alpha\star\alpha'}$.
\end{itemize}
\end{defn}

\begin{rmk}
The fact that the composition is well defined follows from Remark \ref{rmk:starok}, and the fact that the identity is the prescribed one is immediate. Moreover, the fact that $\widetilde{\mathcal{G}}^{m,k}_{\operatorname{def}}$ is a groupoid may be proved as follows: given two $(m,k)$-triples $(S_{1},v_{1},H_{1})$ and $(S_{2},v_{2},H_{2})$ and a morphism $\overline{\alpha}$ between them, where $\alpha=(f\colon\mathcal{S}\to T,\mathcal{L},\mathcal{H},t_{1},t_{2},\gamma)$, then $\overline{\alpha}^{-1}$ is $\overline{\alpha^{-1}}$, where $$\alpha^{-1}:=(f\colon\mathcal{S}\to T,\mathcal{L},\mathcal{H},t_{2},t_{1},\gamma^{-1}).$$Indeed $p_{\alpha^{-1}}=p_{\alpha}^{-1}$ and $g_{\alpha^{-1}}=g^{-1}_{\alpha}$, so $\overline{\alpha^{-1}}=\overline{\alpha}^{-1}$ by Remark \ref{rmk:starok}. Finally, the associativity property also holds since the operation of concatenating paths is associative. 
\end{rmk}

We now define the groupoid $\mathcal{P}^{m,k}$. Recall that two $(m,k)$-triples $(S_{1},v_{1},H_1)$ and $(S_{2},v_{2},H_2)$ are congruent (see Definition \ref{defn:congruent}) when $S_{1}=S_{2}=S$, $v_{1}=v_{2}=v$ and  a coherent sheaf $F$ on $S$ with Mukai vector $v$ is $H_1$-(semi)stable sheaf if and only if it is $H_2$-(semi)stable. In this case we denote by 
\[ \chi_{H_1,H_2}\colon M_v(S,H_1)\longrightarrow M_v(S,H_2), \qquad F\mapsto F \]
the identity map.

\begin{defn}
\label{defn:P m k}
Given two strictly positive integers $m$ and $k$, the groupoid $\mathcal{P}^{m,k}$ is defined as follows:
\begin{itemize}
    \item the objects of $\mathcal{P}^{m,k}$ are the $(m,k)$-triples;
    \item for every two $(m,k)$-triples $(S_{1},v_{1},H_{1})$, $(S_{2},v_{2},H_{2})$ we set $$\Hom_{\mathcal{P}^{m,k}}((S_{1},v_{1},H_{1}),(S_{2},v_{2},H_{2})):=\{\chi_{H_{1},H_{2}}\}$$if $(S,v,H_{1})$ and $(S_{2},v_{2},H_{2})$ are congruent, and otherwise $$\Hom_{\mathcal{P}^{m,k}}((S_{1},v_{1},H_{1}),(S_{2},v_{2},H_{2})):=\emptyset.$$
\end{itemize}
\end{defn}

We are now finally in the position to define the groupoid $\mathcal{G}^{m,k}_{\operatorname{def}}$.

\begin{defn}\label{defn:G m k def}
Given two strictly positive integers $m$ and $k$, the groupoid 
\[ \mathcal{G}^{m,k}_{\operatorname{def}}:=\widetilde{\mathcal{G}}^{m,k}_{\operatorname{def}}\ast\mathcal{P}^{m,k}  \]
is the free product of $\widetilde{\mathcal{G}}^{m,k}_{\operatorname{def}}$ and $\mathcal{P}^{m,k}$ (see Definition~\ref{defn:free product of groupoids}).
\end{defn}

\subsection{Fourier--Mukai equivalences and their groupoid}\label{section:FM}

The purpose of this section is to define the groupoid $\mathcal{G}^{m,k}_{\operatorname{FM}}$. To do so, we need to recall some isomorphisms of moduli spaces of sheaves induced by Fourier--Mukai transforms.

Let us start with the tensorization with a line bundle. Let $S$ be a projective K3 surface and $L\in\Pic(S)$ a line bundle. Consider the derived equivalence
\begin{align}\label{eqn:L}
\mathsf{L}\colon\operatorname{D}^b(S) & \longrightarrow\operatorname{D}^b(S) \\
F^\bullet & \longmapsto F^\bullet\otimes L. \nonumber
\end{align}
If a sheaf $F$ has Mukai vector $v(F)=v$, then $v(\mathsf{L}(F))=v\cdot\operatorname{ch}(L)=:v_L$. With an abuse of notation we keep denoting by
\[ \mathsf{L}\colon M_v(S,H)\longrightarrow M_{v_L}(S,H) \]
the induced morphism between moduli spaces of sheaves. This morphism is known to be an isomorphism in some cases, as the following states.

\begin{lemma}[\protect{\cite[Lemma~2.24]{PR:SingularVarieties}}]\label{lemma:2.20}
Let $S$ be a projective K3 surface, $v=(r,\xi,a)$ a Mukai vector and $H$ an ample line bundle.
\begin{enumerate}
\item For any $d\in\mathbb{Z}$, the morphism $\mathsf{dH}\colon M_v(S,H)\to M_{v_{dH}}(S,H)$ is an isomorphism.
\item If $r>0$ and $H$ is $v$-generic, the morphism $\mathsf{L}\colon M_v(S,H)\to M_{v_{L}}(S,H)$ is an isomorphism for any $L\in\Pic(S)$.
\end{enumerate}
\end{lemma}

We now consider the Fourier--Mukai transform whose kernel is the ideal of the diagonal. Let $S$ be a projective K3 surface and $\Delta\subset S\times S$ be the diagonal. If $I_\Delta\in\operatorname{Coh}(S\times S)$ denotes the ideal sheaf of $\Delta$, then we consider the Fourier-Mukai equivalence
$$\operatorname{FM}_\Delta\colon\operatorname{D}^b(S) \longrightarrow\operatorname{D}^b(S)\,\,\,\,\,\,\,\,\,\operatorname{FM}_{\Delta}(F^\bullet):=R\pi_{2*}(I_\Delta\stackrel{L}{\otimes}\pi_1^*F^\bullet),$$
where $\pi_1$ and $\pi_2$ are the two projections from $S\times S$ to $S$, and the composite equivalence
$$\operatorname{FM}^\vee_\Delta\colon\operatorname{D}^b(S) \longrightarrow\operatorname{D}^b(S)\,\,\,\,\,\,\,\,\,\operatorname{FM}_{\Delta}(F^\bullet):=\left(R\pi_{2*}(I_\Delta\stackrel{L}{\otimes}\pi_1^*F^\bullet)\right)^\vee,$$

If $F$ is a sheaf on $S$ with Mukai vector $v(F)=(r,\xi,a)=v$, then by a direct check one can see that 
\[ v(\operatorname{FM}_\Delta(F))=(a,-\xi,r)=:\tilde{v}\quad\mbox{ and }\quad v(\operatorname{FM}^\vee_\Delta(F))=(a,\xi,r)=:\hat{v} \]

\begin{lemma}[\cite{Yoshioka:FM,PR:SingularVarieties}]\label{lemma:Y-PR}
Let $S$ be a projective K3 surface, $H$ an ample line bundle and $n,a\in\mathbb{Z}$.
\begin{enumerate}
\item Suppose that $\Pic(S)=\ZZ\,H$. Put $v=(r,nH,a)$, with $r>0$. Then there exists an integer $n_0\gg0$ such that for every $n> n_0$ the Fourier--Mukai transform $\operatorname{FM}_{\Delta}$ induces an isomorphism 
\[ \operatorname{FM}_{\Delta,v}\colon M_v(S,H)\longrightarrow M_{\tilde{v}}(S,H). \]
\item Put $v=(0,\xi,a)$ and suppose that $H$ is both $v$ and $\tilde{v}$-generic. Then there exists an integer $a_0\gg0$ such that for every $a>a_0$ the Fourier--Mukai transform $\operatorname{FM}_{\Delta}$ induces an isomorphism 
\[ \operatorname{FM}_{\Delta,v}\colon M_v(S,H)\longrightarrow M_{\tilde{v}}(S,H). \]
\item In any of the two cases above, the dual Fourier--Mukai transform $\operatorname{FM}^\vee_{\Delta}$ induces an isomorphism
\[ \operatorname{FM}^\vee_{\Delta,v}\colon M_v(S,H)\longrightarrow M_{\hat{v}}(S,H). \]
\end{enumerate}
\end{lemma}

\begin{proof}
    The proof of the first two items is \cite[Proposition~2.29, Proposition~2.33]{PR:SingularVarieties} (but see also \cite[Theorem~3.18]{Yoshioka:FM} for a more general statement). Moreover in \cite{PR:SingularVarieties} it is further shown that the sheaf $\operatorname{FM}_\Delta(F)$ is locally free (see \cite[Lemma~2.28]{PR:SingularVarieties}). The proof of the third item now follows from the first two just by noticing that a locally free sheaf $F$ is (semi)stable if and only if its dual $F^\vee$ is (semi)stable. 
\end{proof}

The last Fourier--Mukai transform we consider is associated to elliptic K3 surfaces and it has been studied in \cite{Bridgeland}; the result we need is contained in \cite{Yoshioka:ModuliAbelian}.

Let $p\colon S\to\PP^1$ be an elliptic K3 surface. We assume that there is a section $s\colon\PP^1\to S$ and we denote by $\ell\in\Pic(S)$ its cohomology class. If we put $f=p^*\cO(1)$, then the lattice spanned by $\ell$ and $f$ is the unimodular hyperbolic plane $U$; more precisely, $\ell^2=-2$, $(\ell,f)=1$ and $f^2=0$. Let us assume that $\Pic(S)$ coincides with this hyperbolic plane (i.e.\ that $S$ is generic with this property). It is known that the class $H=\ell+tf$ is ample for $t\gg0$ and we fix once and for all such an ample class.

Consider the moduli space $M_{(0,f,0)}(S,H)$. For $t\gg0$, the polarisation $H$ is generic with respect to $(0,f,0)$ and $M_{(0,f,0)}(S,H)\cong S$ (see \cite[Section~4]{Bridgeland}). Let $\mathcal{P}$ be the relative Poincar\'e bundle that we see as a coherent sheaf over the product $S\times S$. The Fourier--Mukai transform
\[ \operatorname{FM}_{\mathcal{P}}\colon\operatorname{D}^b(S) \longrightarrow\operatorname{D}^b(S),\,\,\,\,\,\,\,\,\,\operatorname{FM}_{\mathcal{P}}(F^\bullet):=(R\pi_{2*}(\mathcal{P}\stackrel{L}{\otimes}\pi_1^*F^\bullet))[1] \]
is proved to be an equivalence in \cite[Theorem~1.2]{Bridgeland}.

\begin{lemma}[\protect{\cite[Theorem~3.15]{Yoshioka:ModuliAbelian}}]\label{prop:Poincare}
Let $H=\ell+tf$ with $t\gg0$. Then $\operatorname{FM}_{\mathcal{P}}$ induces an isomorphism
\[ \operatorname{FM}_{\mathcal{P}}\colon M_{(m,0,-mk)}(S,H)\longrightarrow M_{(0,m(\ell+(k+1)f),m)}(S,H) \]
for every $m,k>0$.
\end{lemma}

We conclude this section by defining the groupoid $\mathcal{G}^{m,k}_{\operatorname{FM}}$.
\begin{defn}\label{defn:G m k FM}
 Given two strictly positive integers $m$ and $k$, the groupoid $\mathcal{G}^{m,k}_{\operatorname{FM}}$ is defined as follows:
\begin{itemize}
	\item the objects of $\mathcal{G}^{m,k}_{\operatorname{FM}}$ are the $(m,k)$-triples;

 \item given two $(m,k)$-triples $(S,v_1,H)$ and $(S,v_2,H)$, an \textit{elementary morphism} between them is one of the following:
 \begin{itemize}
     \item the equivalence $\mathsf{L}$, if $(S,v_1,H)=(S,v,H)$ and $(S,v_2,H)=(S,v_L,H)$ are as in Lemma~\ref{lemma:2.20};
     \item the equivalence $\operatorname{FM}_\Delta$, if $(S,v_1,H)=(S,v,H)$ and $(S,v_2,H)=(S,\tilde{v},H)$ verify the hypotheses of  Lemma~\ref{lemma:Y-PR} (1) or (2);
     \item the equivalence $\operatorname{FM}^\vee_\Delta$, if $(S,v_1,H)=(S,v,H)$ and $(S,v_2,H)=(S,\hat{v},H)$ verify the hypotheses of  Lemma~\ref{lemma:Y-PR} (3);
     \item the equivalence $\operatorname{FM}_{\mathcal{P}}$, if $(S,v_1,H)=(S,(m,0,m-mk),H)$ and $(S,v_2,H)=(S,(0,m(\ell+kf),m),H)$ verify the hypotheses of  Lemma~\ref{prop:Poincare};
 \end{itemize}
 
	\item for every two $(m,k)$-triples $(S_{1},v_{1},H_{1}),(S_{2},v_{2},H_{2})$, a morphisms from $(S_{1},v_{1},H_{1})$ to $(S_{2},v_{2},H_{2})$ is a formal concatenation of elementary morphisms and their formal inverses, subject to the usual cancellation rules.
\end{itemize}
\end{defn}

\subsection{The groupoid $\mathcal{G}^{m,k}$ and its representations}\label{section:representation}

For every two strictly positive integers $m,k$ we now define the groupoid $\mathcal{G}^{m,k}$ as the groupoid generated by $\mathcal{G}^{m,k}_{\operatorname{def}}$ and $\mathcal{G}^{m,k}_{\operatorname{FM}}$. 
\begin{defn}\label{defn:G m k}
Given two strictly positive integers $m$ and $k$, the groupoid 
\[ \mathcal{G}^{m,k}:=\mathcal{G}^{m,k}_{\operatorname{def}}\ast \mathcal{G}^{m,k}_{\operatorname{FM}}  \]
is the free product of $\mathcal{G}^{m,k}_{\operatorname{def}}$ and $\mathcal{G}^{m,k}_{\operatorname{FM}}$ (see Definition~\ref{defn:free product of groupoids}).
\end{defn}

\begin{rmk}\label{rmk:non empty}
Let $(S_1,v_1,H_1),(S_2,v_2,H_2)\in\mathcal{G}^{m,k}$ be two objects; then  
\[ \Hom_{\mathcal{G}^{m,k}}((S_1,v_1,H_1),(S_2,v_2,H_2))\neq\emptyset. \]

More precisely, in the proof of \cite[Theorem~1.7]{PR:SingularVarieties} the authors show that one can go from $(S_1,v_1,H_1)$ to $(S_2,v_2,H_2)$ only using the following morphisms of $\mathcal{G}^{m,k}$:
\begin{itemize}
    \item equivalence classes of deformation paths of $(m,k)$-triples;
    \item the morphisms $\chi_{H,H'}$ of $\mathcal{P}^{m,k}$;
    \item derived equivalences of the form $\mathsf{L}$, for some line bundle $L$;
    \item the Fourier--Mukai transform $\operatorname{FM}_\Delta$.
\end{itemize}
\end{rmk}

\subsubsection{The $\widetilde{\cH}$-representation $\widetilde{\Phi}^{m,k}$ of $\mathcal{G}^{m,k}$}

Let $\widetilde{\Lambda}$ be an even unimodular lattice of signature $(4,20)$. Notice that the isometry class of $\widetilde{\Lambda}$ is uniquely determined, so that $\widetilde{\Lambda}$ is isometric to the lattice
\[ U^{\oplus4}\oplus E_8(-1)^{\oplus2}, \]
where $U$ is the unimodular even rank 2 lattice and $E_8(-1)$ is the negative definite Dynkin lattice of type $E_8$. 

If $\widetilde{\mathcal{C}}\subset\widetilde{\Lambda}\otimes_{\ZZ}\RR$ is the cone of (strictly) positive classes, then by \cite[Lemma~4.1]{Markman:Survey} we have that $\oH^3(\widetilde{\mathcal{C}},\ZZ)=\ZZ$. The choice of a generator of $\oH^3(\widetilde{\mathcal{C}},\ZZ)$ is an \emph{orientation} of $\widetilde{\Lambda}$. Notice that there are only two orientations, corresponding to the generators $\pm 1$ of $\ZZ$. Again by \cite[Lemma~4.1]{Markman:Survey}, if $W\subset\widetilde{\Lambda}\otimes_{\ZZ}\RR$ is a positive real subspace of dimension $4$, then the space $W\setminus\{0\}$ is a deformation retract of $\widetilde{\mathcal{C}}$, so that an orientation on $\widetilde{\Lambda}$ corresponds to an orientation on a positive real subspace of maximal dimension.

If $g\in\Or(\widetilde{\Lambda})$ is an isometry, then $g$ induces an action on $\oH^3(\widetilde{\mathcal{C}},\ZZ)$ that either preserves a generator or it maps it to its opposite. Therefore we get an orientation character
\begin{equation}\label{eqn:or}
\operatorname{or}\colon\Or(\widetilde{\Lambda})\longrightarrow\ZZ/2\ZZ.\end{equation}
The subgroup $\Or^+(\widetilde{\Lambda})=\ker(\operatorname{or})$ is the group of orientation preserving isometries.  

Similarly, if $(\widetilde{\Lambda}_1,\epsilon_1)$ and $(\widetilde{\Lambda}_2,\epsilon_2)$ are two pairs composed by a unimodular even lattice $\widetilde{\Lambda}_i$ of signature $(4,20)$ and an orientation $\epsilon_i\in\oH^3(\widetilde{\mathcal{C}}_i,\ZZ)$ on $\widetilde{\Lambda}_i$, then we get an orientation map
\begin{equation}\label{eqn:or tra due}
\operatorname{or}\colon\Or((\widetilde{\Lambda}_1,\epsilon_1),(\widetilde{\Lambda}_2,\epsilon_2))\longrightarrow\ZZ/2\ZZ. \end{equation}
Again, $\Or^+((\widetilde{\Lambda}_1,\epsilon_1),(\widetilde{\Lambda}_2,\epsilon_2))\coloneqq\operatorname{or}^{-1}(0)$ is the set of orientation preserving isometries.

\begin{defn}\label{defn:H tilde m k}
Given two strictly positive integers $m$ and $k$, the groupoid $\widetilde{\cH}^{m,k}$ is defined as follows:
\begin{itemize}
    \item the objects are triples $(\widetilde{\Lambda},v,\epsilon)$, where $\widetilde{\Lambda}$ is a unimodular even lattices of signature $(4,20)$, $v\in\widetilde{\Lambda}$ is of the form $v=mw$, where $w$ is primitive and $w^2=2k$, and where $\epsilon$ is an orientation on $\widetilde{\Lambda}$;
    \item for any two objects $(\widetilde{\Lambda}_1,v_1,\epsilon_1)$ and $(\widetilde{\Lambda}_2,v_2,\epsilon_2)$, the set $$\Hom_{\widetilde{\cH}^{m,k}}((\widetilde{\Lambda}_1,v_1,\epsilon_1),(\widetilde{\Lambda}_1,v_1,\epsilon_2))$$ is the set of isometries $g\colon\widetilde{\Lambda}_1\to\widetilde{\Lambda}_2$ such that $g(v_1)=v_2$.
\end{itemize}
\end{defn}

\begin{rmk}
    Notice that morphisms in $\widetilde{\cH}^{m,k}$ are not necessarily orientation preserving.
\end{rmk}

\begin{example}\label{example:Mukai lattice}
    If $S$ is a K3 surface, then the Mukai lattice $\widetilde{\oH}(S,\ZZ)$ is a unimodular even lattice of signature $(4,20)$.
    As it is explained in \cite[Section~4.1]{Markman:Monodromy}, $\widetilde{\oH}(S,\ZZ)$ comes with a distinguished orientation, which we denote by $\epsilon_S$. Such an orientation is associated to the distinguished positive real $4$-space $W$ with a chosen basis given by a K\"ahler class, the real and imaginary parts of a symplectic form and the vector $(1,0,-1)$.
    
    We also point out that, as we have seen in the proof of Lemma~\ref{lemma:Mon is O+}, the positive real $3$-space with basis given by a K\"ahler class and the real and imaginary part of a symplectic form determines an orientation on the lattice $\oH^2(S,\ZZ)$.
\end{example}

Before continuing, let us recall the following definition.
\begin{defn}
    Let $f\colon\mathcal{S}\to T$ be a smooth family of K3 surfaces. Take $t_1,t_2\in T$ and a path $\gamma$ from $t_1$ to $t_2$. A \emph{parallel transport operator} 
    $$g\colon\widetilde{\oH}(\mathcal{S}_{t_1},\ZZ)\to\widetilde{\oH}(\mathcal{S}_{t_2},\ZZ)$$
    is an isometry induced by parallel transport inside the local system $R^\bullet f_*\ZZ$.
\end{defn}

This allows us to define the following representation of the groupoid $\mathcal{G}^{m,k}_{\operatorname{def}}$.

\begin{defn}\label{defn:Phi tilde def m k}
Given two strictly positive integers $m$ and $k$, the representation
\[ \widetilde{\Phi}^{m,k}_{\operatorname{def}}\colon\mathcal{G}^{m,k}_{\operatorname{def}}\longrightarrow\widetilde{\cH}^{m,k} \]
is defined as follows:
\begin{itemize}
    \item if $(S,v,H)\in\mathcal{G}^{m,k}_{\operatorname{def}}$ is an object, then $\widetilde{\Phi}^{m,k}_{\operatorname{def}}(S,v,H)=(\widetilde{\oH}(S,\ZZ),v,\epsilon_S)$ (see Example~\ref{example:Mukai lattice});
    \item if $(S_1,v_1,H_1)$ and $(S_{2},v_{2},H_{2})$ are two objects and    $\alpha=(f\colon\mathcal{S}\to T,\mathcal{L},\mathcal{H},t_{1},t_{2},\gamma)$ is a deformation path from $(S_{1},v_{1},H_{1})$ to $(S_{2},v_{2},H_{2})$, then 
    \[ \widetilde{\Phi}^{m,k}_{\operatorname{def}}(\overline{\alpha}):=p_{\alpha}, \] the parallel transport operator in the local system $R^\bullet f_*\ZZ$ along the path $\gamma$;
    \item if $(S_1,v_1,H_1)$ and $(S_{2},v_{2},H_{2})$ are congruent, then  
    \[ \widetilde{\Phi}^{m,k}_{\operatorname{def}}(\chi_{H_{1},H_{2}}):=\id_{\widetilde{\oH}(S,\mathbb{Z})}, \] 
    where $\chi_{H_1,H_2}$ is the identification $M_{v_{1}}(S_{1},H_{1})=M_{v_{2}}(S_{2},H_{2})$ (see Definitions \ref{defn:congruent} and \ref{defn:P m k}).
\end{itemize}
\end{defn}

\begin{rmk}
Notice that if $\alpha=(f\colon\mathcal{S}\to T,\mathcal{L},\mathcal{H},t_{1},t_{2},\gamma)$ is a deformation path from $(S_1,v_1,H_1)$ to $(S_2,v_2,H_2)$, then by definition the Mukai vectors $v_1$ and $v_2$ belong to the same flat section of the local system $R^\bullet f_*\ZZ$, so the parallel transport operator $p_{\alpha}$ maps $v_1$ to $v_2$ and the representation $\widetilde{\Phi}^{m,k}_{\operatorname{def}}$ is well defined.
\end{rmk}

\begin{rmk}\label{rmk:pto on H tilde}
    Because of Lemma~\ref{lemma:Mon is O+} and the fact that the vector $(1,0,-1)$ is preserved, we see that a parallel transport operator $g\colon\widetilde{\oH}(S_1,\ZZ)\to\widetilde{\oH}(S_2,\ZZ)$ is orientation preserving. In particular the morphisms in the image of $\widetilde{\Phi}^{m,k}_{\operatorname{def}}$ are always orientation preserving.
\end{rmk}

Let us now define the representation
\[ \widetilde{\Phi}^{m,k}_{\operatorname{FM}}\colon\mathcal{G}^{m,k}_{\operatorname{FM}}\longrightarrow\widetilde{\cH}^{m,k}. \]
\begin{defn}\label{defn:Phi tilde FM m k}
Given two strictly positive integers $m$ and $k$, the representation $\widetilde{\Phi}^{m,k}_{\operatorname{FM}}$ is defined as follows:
\begin{itemize}
    \item if $(S,v,H)\in\mathcal{G}^{m,k}_{\operatorname{FM}}$ is an object, then $\widetilde{\Phi}^{m,k}_{\operatorname{FM}}(S,v,H)=(\widetilde{\oH}(S,\ZZ),v,\epsilon_S)$;
    \item if $\phi\in\Aut(\operatorname{D}^b(S))$ corresponds to a morphism in $\mathcal{G}^{m,k}_{\operatorname{FM}}$, then 
    $\widetilde{\Phi}^{m,k}_{\operatorname{FM}}(\phi)=\phi^{\oH}$ is the isometry induced by $\phi$ on the Mukai lattice.
\end{itemize}
\end{defn}
\begin{rmk}
    The equivalence $\phi$ is composition of the equivalences introduced in Section~\ref{section:FM} and, by definition, it induces an isomorphism between the corresponding moduli spaces. It follows in particular that the Mukai vector must be preserved, so that again $\widetilde{\Phi}^{m,k}_{\operatorname{FM}}$ is well defined.
\end{rmk}

\begin{rmk}
    Notice that morphisms in the image of $\widetilde{\Phi}^{m,k}_{\operatorname{FM}}$ are not necessarily orientation preserving.
\end{rmk}

\begin{defn}\label{defn: Phi tilde m k}
    Define the representation
    \[ \widetilde{\Phi}^{m,k}\colon\mathcal{G}^{m,k}\longrightarrow\widetilde{\cH}^{m,k} \]
    as the unique representation restricting to  $\widetilde{\Phi}^{m,k}_{\operatorname{def}}$ on $\mathcal{G}^{m,k}_{\operatorname{def}}$ and to $\widetilde{\Phi}^{m,k}_{\operatorname{FM}}$ on $\mathcal{G}^{m,k}_{\operatorname{FM}}$.
\end{defn}

\begin{rmk}
The existence and uniqueness of $\widetilde{\Phi}^{m,k}$ are a consequence of the fact that the objects of $\mathcal{G}^{m,k}$ are the same as those of $\mathcal{G}^{m,k}_{\operatorname{def}}$ and $\mathcal{G}^{m,k}_{\operatorname{FM}}$, and the representations $\widetilde{\Phi}^{m,k}_{\operatorname{def}}$ and $\widetilde{\Phi}^{m,k}_{\operatorname{FM}}$ coincide on the objects. Moreover, as morphisms in $\mathcal{G}^{m,k}$ are formal concatenations of morphisms in $\mathcal{G}^{m,k}_{\operatorname{def}}$ and $\mathcal{G}^{m,k}_{\operatorname{FM}}$, there is a unique way to define $\widetilde{\Phi}^{m,k}$ on morphisms.
\end{rmk}

\subsubsection{The $\mathcal{A}_k$-representation $\mathsf{pt}^{m,k}$ of $\mathcal{G}^{m,k}$}\label{section:Phi=mon}
\begin{defn}\label{defn: A k}
For every $k>0$ we define the groupoid $\mathcal{A}_k$ as follows: 
\begin{itemize}
	\item the objects of $\mathcal{A}_{k}$ are even lattices $\Lambda$ of signature $(3,20)$ isometric to the lattice $U^{\oplus 3}\oplus E_{8}(-1)^{\oplus 2}\oplus\langle -2k\rangle$;
	\item if $\Lambda_1$ and $\Lambda_2$ are two objects, then 
 \[ \Hom_{\mathcal{A}_k}(\Lambda_1,\Lambda_2):=\Or(\Lambda_1,\Lambda_2). \]
\end{itemize}
\end{defn}

As before we first define $\mathcal{A}_k$-representations for both $\mathcal{G}^{m,k}_{\operatorname{def}}$ and $\mathcal{G}^{m,k}_{\operatorname{FM}}$. We start by defining the representation $\mathsf{pt}^{m,k}_{\operatorname{def}}\colon\mathcal{G}^{m,k}_{\operatorname{def}}\to\mathcal{A}_{k}$.

\begin{defn}\label{defn:pt m k def}
    Given two strictly positive integers $m$ and $k$ we define the functor $\mathsf{pt}^{m,k}_{\operatorname{def}}$ as follows:
    \begin{itemize}
        \item if $(S,v,H)\in\mathcal{G}^{m,k}_{\operatorname{def}}$ is an object, then 
        \[ \mathsf{pt}^{m,k}_{\operatorname{def}}((S,v,H))=\oH^2(M_v(S,H),\ZZ); \]
        \item if $(S_1,v_1,H_1)$ and $(S_2,v_2,H_2)$ are two objects and $\alpha=(f\colon\mathcal{S}\to T,\mathcal{L},\mathcal{H},t_{1},t_{2},\gamma)$ is a deformation path from $(S_{1},v_{1},H_{1})$ to $(S_{2},v_{2},H_{2})$, then 
        \[ \mathsf{pt}^{m,k}_{\operatorname{def}}(\overline{\alpha}):=g_\alpha, \]
        the parallel transport operator in the local system $R^2p_{v,*}\ZZ$ along the path $\gamma$;
        \item if $(S_1,v_1,H_1)$ and $(S_2,v_2,H_2)$ are congruent, then $$\mathsf{pt}^{m,k}_{\operatorname{def}}(\chi_{H_{1},H_{2}}):=\id_{\oH^{2}(M_{v}(S,H_{1}),\mathbb{Z})},$$
        where $\chi_{H_1,H_2}$ is the identification $M_{v_{1}}(S_{1},H_{1})=M_{v_{2}}(S_{2},H_{2})$ (see Definitions \ref{defn:congruent} and \ref{defn:P m k}).
    \end{itemize}
\end{defn}
\begin{rmk}
    By Theorem~\ref{thm:PR v perp} there is an isometry 
    $\oH^2(M_v(S,H),\ZZ)\cong v^\perp$;
    since $v=mw$, with $w^2=2k$, it follows that $\oH^2(M_v(S,H),\ZZ)$ is isometric to the lattice $U^{\oplus 3}\oplus E_{8}(-1)^{\oplus 2}\oplus\langle -2k\rangle$.
\end{rmk}

Next, let us define the functor 
\[ \mathsf{pt}^{m,k}_{\operatorname{FM}}\colon\mathcal{G}^{m,k}_{\operatorname{FM}}\to\mathcal{A}_k. \]
\begin{defn}\label{defn:pt m k FM}
Given two strictly positive integers $m$ and $k$, we define the functor $\mathsf{pt}^{m,k}_{\operatorname{FM}}$ as follows:
\begin{itemize}
    \item if $(S,v,H)\in\mathcal{G}^{m,k}_{\operatorname{FM}}$ is an object, then \[ \mathsf{pt}^{m,k}_{\operatorname{FM}}(S,v,H)=\oH^2(M_v(S,H),\ZZ); \]
    
    \item if $\phi\in\Hom_{\mathcal{G}^{m,k}_{\operatorname{FM}}}((S,v_1,H),(S,v_2,H))$ is an elementary morphism, then by definition $\phi$ induces an isomorphism
    \[ \phi_{v_1}\colon M_{v_1}(S,H)\to M_{v_2}(S,H) \]
    of the moduli spaces. We then define 
    \[ \mathsf{pt}^{m,k}_{\operatorname{FM}}(\phi)=\phi_{v_1,*}, \]
    where the latter is the pushforward action on the second integral cohomology groups of the moduli spaces;
    
    \item if $\phi\in\Hom_{\mathcal{G}^{m,k}_{\operatorname{FM}}}((S,v_1,H),(S,v_2,H))$ is a morphism, i.e.\ a concatenation of elementary morphisms and their inverses, then we define $\mathsf{pt}^{m,k}_{\operatorname{FM}}(\phi)$ as the composition of the corresponding isometries.
\end{itemize}
\end{defn}

Similarly to the case of $\widetilde{\Phi}^{m,k}$, we now give the following definition.

\begin{defn}\label{defn:pt m k}
    Define 
    \begin{equation}\label{eqn:pt} \mathsf{pt}^{m,k}\colon\mathcal{G}^{m,k}\to\mathcal{A}_k 
    \end{equation}
    as the unique representation restricting to $\mathsf{pt}^{m,k}_{\operatorname{def}}$ on $\mathcal{G}^{m,k}_{\operatorname{def}}$ and to $\mathsf{pt}^{m,k}_{\operatorname{FM}}$ on $\mathcal{G}^{m,k}_{\operatorname{FM}}$.
\end{defn}

\begin{prop}\label{prop:Im of pt in Mon}
    Let $A_1=(S_1,v_1,H_1),A_2=(S_2,v_2,H_2)\in\mathcal{G}^{m,k}$ be two objects. Then 
    \[ \mathsf{pt}^{m,k}(\Hom_{\mathcal{G}^{m,k}}(A_1,A_2))\subset\mathsf{PT}_{\operatorname{lt}}(M_{v_1}(S_1,H_1),M_{v_2}(S_2,H_2)). \]
\end{prop}
\begin{proof}
    First of all, we notice that 
    \[ \mathsf{pt}^{m,k}_{\operatorname{def}}(\Hom_{\mathcal{G}^{m,k}_{\operatorname{def}}}(A_1,A_2))\subset\mathsf{PT}_{\operatorname{lt}}(M_{v_1}(S_1,H_1),M_{v_2}(S_2,H_2)) \]
    by definition. On the other hand, by Proposition~\ref{prop:iso is mon} we also have that
    \[ \mathsf{pt}^{m,k}_{\operatorname{FM}}(\Hom_{\mathcal{G}^{m,k}_{\operatorname{FM}}}(A_1,A_2))\subset\mathsf{PT}_{\operatorname{lt}}(M_{v_1}(S_1,H_1),M_{v_2}(S_2,H_2)) \]
    so that the claim follows.
\end{proof}

Using the notation set in Section~\ref{section:groupoids}, we state the following corollary of Proposition \ref{prop:Im of pt in Mon}.
\begin{cor}\label{cor:Im of pt in Mon}
    Let $(S,v,H)\in\mathcal{G}^{m,k}$ be an object. Then 
    \[ \Im\left( \mathsf{pt}^{m,k}_{(S,v,H)}\colon\Aut(S,v,H)\to\Aut(\oH^2(M_v(S,H),\ZZ) \right)\subset\Mon^2_{\operatorname{lt}}(M_v(S,H)). \]
\end{cor}

\subsubsection{Relation between the two representations $\widetilde{\Phi}^{m,k}$ and $\mathsf{pt}^{m,k}$}
We now connect the representation $\mathsf{pt}^{m,k}$ of $\mathcal{G}^{m,k}$ in $\mathcal{A}_{k}$ with the the representation $\widetilde{\Phi}^{m,k}$ of $\mathcal{G}^{m,k}$ in $\widetilde{\mathcal{H}}^{m,k}$ we defined before. 

\begin{defn}\label{defn:Psi}
Given two strictly positive integers $m$ and $k$, define the functor 
\[ \Psi\colon\widetilde{\cH}^{m,k}\longrightarrow\mathcal{A}_k \]
as follows:
\begin{itemize}
    \item if $(\widetilde{\Lambda},v,\epsilon)$ is an object in $\widetilde{\cH}^{m,k}$, then 
    \[ \Psi(\widetilde{\Lambda},v)=v^\perp; \]
    \item if $(\widetilde{\Lambda}_i,v_i,\epsilon_i)$ are two objects and $g\colon\widetilde{\Lambda}_1\to\widetilde{\Lambda}_2$ an isometry such that $g(v_1)=v_2$, then 
    \[ \Psi(g)=(-1)^{\operatorname{or}(g)} g|_{v_1^\perp}\colon v_1^\perp\longrightarrow v_2^\perp, \]
    where $\operatorname{or}$ is the orientation character (\ref{eqn:or tra due}).
\end{itemize}
\end{defn}

Put 
\[ \Phi^{m,k}=\Psi\circ\widetilde{\Phi}^{m,k}\colon\mathcal{G}^{m,k}\longrightarrow\mathcal{A}_k.
\] 


\begin{prop}\label{prop:iso of functors}
    There exists an isomorphism of functors
    \[ \lambda\colon\Phi^{m,k}\longrightarrow \mathsf{pt}^{m,k}. \]
\end{prop}
We recall that if $\mathcal{A}$ and $\mathcal{B}$ are two categories and $F,G\colon\mathcal{A}\to\mathcal{B}$ are two functors, an isomorphism of functors $\lambda\colon F\to G$ is a natural transformation such that for each $A\in\mathcal{A}$ the morphism $\lambda(A)\colon F(A)\to G(A)$ is an isomorphism in $\mathcal{B}$.
\begin{proof}
    First of all, let us define the representations
    \[ \Phi^{m,k}_{\operatorname{def}}=\Psi\circ\widetilde{\Phi}^{m,k}_{\operatorname{def}}\colon\mathcal{G}^{m,k}_{\operatorname{def}}\longrightarrow\mathcal{A}_k 
    \]
    and
    \[ \Phi^{m,k}_{\operatorname{FM}}=\Psi\circ\widetilde{\Phi}^{m,k}_{\operatorname{FM}}\colon\mathcal{G}^{m,k}_{\operatorname{FM}}\longrightarrow\mathcal{A}_k. 
    \]
    We will prove the existence of two isomorphisms of functors 
    \[ \lambda_{\operatorname{def}}\colon \Phi^{m,k}_{\operatorname{def}}\longrightarrow \mathsf{pt}^{m,k}_{\operatorname{def}}\qquad\mbox{ and }\qquad \lambda_{\operatorname{FM}}\colon \Phi^{m,k}_{\operatorname{FM}}\longrightarrow \mathsf{pt}^{m,k}_{\operatorname{FM}}, \]
    from which the statement will follow by definition.

    Let us start with $\lambda_{\operatorname{def}}$. For any object $(S,v,H)\in\mathcal{G}^{m,k}_{\operatorname{def}}$, Theorem~\ref{thm:PR v perp} provides an isometry
    \[ \lambda_{(S,v,H)}\colon v^\perp\longrightarrow\oH^2(M_v(S,H),\ZZ). \]
    For $i=1,2$, let $A_i=(S_i,v_i,H_i)\in\mathcal{G}^{m,k}_{\operatorname{def}}$ be two objects, and let us take a morphism $h\in\Hom_{\mathcal{G}^{m,k}_{\operatorname{def}}}(A_1,A_2)$. To conclude the proof of this first step, we need to show that there is a commutative diagram
    \begin{equation}\label{eqn:nat iso def} 
    \xymatrix{
    v_1^\perp\ar@{->}[r]^-{\lambda_{A_1}}\ar@{->}[d]_{\Phi^{m,k}_{\operatorname{def}}(h)} & \oH^2(M_{v_1}(S_1,H_1),\ZZ)\ar@{->}[d]^{\mathsf{pt}^{m,k}_{\operatorname{def}}(h)} \\
    v_2^\perp\ar@{->}[r]_-{\lambda_{A_2}} & \oH^2(M_{v_2}(S_2,H_2),\ZZ)\, .
    }
    \end{equation}
    We first prove the claim when $h=\overline{\alpha}$, where $\alpha=(f\colon\mathcal{S}\to T,\mathcal{L},\mathcal{H},t_{1},t_{2},\gamma)$ is a deformation path. In this case, by Remark~\ref{rmk:pto on H tilde}, we have that $\widetilde{\Phi}^{m,k}_{\operatorname{def}}(\alpha)=p_{\alpha}$ is orientation preserving.
    
    Let $p\colon\mathcal{M}\to T$ be the family of moduli spaces induced by $(f\colon\mathcal{S}\to T,\mathcal{L},\mathcal{H})$. By definition there exists a flat section $\mathsf{v}$ of the local system $R^\bullet f_*\ZZ$ such that $\mathsf{v}_{t_1}=v_1$ and $\mathsf{v}_{t_2}=v_2$. We denote by $\mathsf{v}^\perp$ the corresponding sub-local system of $R^\bullet f_*\ZZ$. By Remark~\ref{rmk:lambda in famiglie},
    the isometries $\lambda_{A_1}$ and $\lambda_{A_2}$ fit in an isomorphism of local systems
    \[ \lambda_{\mathsf{v}}\colon \mathsf{v}^\perp\longrightarrow R^2\phi_*\ZZ. \] 
    
    The commutativity of the diagram (\ref{eqn:nat iso def}) then follows since both $\Phi^{m,k}_{\operatorname{def}}(\alpha)=p_{\alpha}$ and $\mathsf{pt}^{m,k}_{\operatorname{def}}(\alpha)=g_{\alpha}$ are parallel transport operators in the two families associated to $\alpha$ (see also Remark~\ref{rmk:solo p alpha}).

    Let us now turn to the case where $h=\chi_{H_{1},H_{2}}$, i.e.\ when $(S_{1},v_{1},H_{1})$ and $(S_{2},v_{2},H_{2})$ are congruent: then $S_{1}=S_{2}=S$, $v_{1}=v_{2}=v$ and $\chi_{H_1,H_2}$ is the identification $M_v(S,H_1)=M_v(S,H_2)$ (see Definitions \ref{defn:congruent} and \ref{defn:P m k}). In this case we see that $\Phi^{m,k}_{\operatorname{def}}(h)=\id_{v^{\perp}}$, $\mathsf{pt}^{m,k}_{\operatorname{def}}(h)=\id_{\oH^{2}(M_{v}(S,H_{1}))}$, and we have an identification $\oH^{2}(M_{v}(S,H_{1}),\mathbb{Z})=\oH^{2}(M_{v}(S,H_{2}),\mathbb{Z})$ and an identification $\lambda_{A_{1}}=\lambda_{A_{2}}$: diagram (\ref{eqn:nat iso def}) is then commutative.

    By Definition~\ref{defn:G m k def}, we then deduce the existence of the isomorphism $\lambda_{\operatorname{def}}$.

    We are now left with $\lambda_{\operatorname{FM}}$, whose construction follows in the same way. First of all, we define it on objects as in the previous case. Let $A_1=(S,v_1,H)$ and $A_2=(S,v_2,H)$ be two objects and let $\phi\in\Aut(\operatorname{D}^b(S))$ be an equivalence of categories that induces an isomorphism $\mathsf{\phi}_{v_{1}}\colon M_{v_1}(S,H)\to M_{v_2}(S,H)$. We need to prove that there exists a commutative diagram
    \begin{equation}\label{eqn:nat iso FM} 
    \xymatrix{
    v_1^\perp\ar@{->}[r]^-{\lambda_{A_1}}\ar@{->}[d]_{\Phi^{m,k}_{\operatorname{FM}}(\phi)} & \oH^2(M_{v_1}(S_1,H_1),\ZZ)\ar@{->}[d]^{\mathsf{pt}^{m,k}_{\operatorname{FM}}(\phi)} \\
    v_2^\perp\ar@{->}[r]_-{\lambda_{A_2}} & \oH^2(M_{v_2}(S_2,H_2),\ZZ)\, .
    }
    \end{equation}
    Recall that, by definition, $\mathsf{pt}^{m,k}_{\operatorname{FM}}(\phi)=\phi_{v_1,*}$ and $\Phi^{m,k}_{\operatorname{FM}}(\phi)=(-1)^{\operatorname{or}(\phi^{\oH})} \phi^{\oH}|_{v_1^\perp}$, where $\phi^{\operatorname{H}}$ is the action of $\phi$ on the Mukai lattice and $\operatorname{or}$ is the orientation character (\ref{eqn:or tra due}).
    
    First of all, we notice that it is enough to prove the result only when $\phi=\mathsf{L},\operatorname{FM}_\Delta,\operatorname{FM}^\vee_\Delta,\operatorname{FM}_{\mathcal{P}}$ are as in Definition~\ref{defn:G m k FM}.

    Now, if $\phi=\mathsf{L},\operatorname{FM}_\Delta$, then $\phi^{\oH}$ is orientation preserving (see \cite[Remark~5.4]{HS2005}) and by \cite[Proposition~2.4]{Yoshioka:ModuliAbelian} 
    \[ \lambda_{A_2}^{-1}\circ \mathsf{\phi}_{v_1,*} \circ \lambda_{A_1}=\phi^{\oH}|_{v_1^\perp}. \]

    If $\phi=\operatorname{FM}^\vee_\Delta$, then $\phi^{\oH}$ is orientation reversing. In fact the duality equivalence $(-)^\vee\in\Aut(\operatorname{D}^b(S))$ induces in cohomology the isometry 
    \[ \delta\colon (r,c,s)\mapsto (r,-c,s) \]
    which is orientation reversing: it changes the sign to a K\"ahler and a symplectic form, but it is the identity on the hyperbolic plane generated by $\oH^0(S,\ZZ)$ and $\oH^4(S,\ZZ)$. In this case, by \cite[Proposition~2.5]{Yoshioka:ModuliAbelian} we have that
    \[ \lambda_{A_2}^{-1}\circ \mathsf{\phi}_{v_1,*} \circ \lambda_{A_1}=-\phi^{\oH}|_{v_1^\perp}. \] 

    Finally , if $\phi=\operatorname{FM}_{\mathcal{P}}$, then by \cite[Remark~5.4, Proposition~5.5.]{HS2005} $\phi^{\oH}$ is orientation preserving and by \cite[Proposition~2.4]{Yoshioka:ModuliAbelian} 
    \[ \lambda_{A_2}^{-1}\circ \mathsf{\phi}_{v_1,*} \circ \lambda_{A_1}=\phi^{\oH}|_{v_1^\perp}. \]    
    
    Those are exactly the commutativity condition needed to define the isomorphism $\lambda_{\operatorname{FM}}$ and we are done.
\end{proof}


\section{Polarised monodromy of K3 surfaces and its lift to moduli spaces}\label{section:mon S pol}

The first part of this section is dedicated to show that the monodromy group of a K3 surface can be generated only by polarised monodromy operators. This result is well-known to experts but a rigorous proof is lacking in literature. 

In the second and last part of the section we will instead show how to lift polarised monodromy operators on a K3 surface to some moduli spaces of sheaves on the same K3 surface.

\subsection{The monodromy group of a K3 surface}\label{section:Mon S}

The aim of this section is to prove the following result.
\begin{thm}\label{thm:mon S generated by ppto}
    Let $S$ be a projective K3 surface. Then the monodromy group $\Mon^2(S)$ is generated by polarised parallel transport operators.
\end{thm}
In the statement above we mean that there exists a set of generators of $\Mon^2(S)$ each of which is composition of polarised parallel transport operators (see Definition~\ref{defn:ppto} for the notion of polarised parallel transport operator). 

By deforming the K3 surface via polarised families, it is enough to prove the statement for a special example. We will work with a projective elliptic K3 surface $p\colon S\to\PP^1$ with a section and with Picard rank $2$. In particular, if we denote by $f$ the class of the fibre and by $\ell$ the class of the section, then
\[ \Pic(S)=\ZZ.\ell\oplus\ZZ.f=
\left(\begin{matrix}
-2 & 1 \\
1  & 0
\end{matrix}\right) \]
is the unimodular hyperbolic plane. Let us put $e=\ell+f$, so that $e$ and $f$ form the standard basis of the hyperbolic plane.

Let us recall that in this case a class $h=\alpha e + \beta f$ is ample when the ratio $\beta/\alpha>1$. Let us now choose three positive integers $r,k,p\gg0$, and consider the following ample classes,
\[ h_1=e+rf,\qquad h_2=e+(r-1)f,\qquad h_3=se+pf,\quad\mbox{and}\quad h_4=(s-1)e+pf. \]

We will prove the following result, which will imply Theorem~\ref{thm:mon S generated by ppto}.

\begin{prop}\label{prop:mon S pol}
    Let $S$ be a projective elliptic K3 surface with a section and Picard rank $2$. Let $H_1$, $H_2$, $H_3$ and $H_4$ be four polarisations whose classes are the classes $h_1$, $h_2$, $h_3$ and $h_4$ above. Then 
    \[ \Mon^2(S)=\langle \Mon^2(S,H_1),\Mon^2(S,H_2),\Mon^2(S,H_3),\Mon^{2}(S,H_{4})\rangle. \]
\end{prop}

First of all, let us recall the following well-known result due to \cite[Theorem~3.4.1 and Corollary~3.4.2]{Peters}. The following statement can be found also in \cite[Proposition~6.8]{Markman:Monodromy}.
\begin{lemma}[\cite{Peters}]\label{lemma:pol}
Let $S$ be a projective K3 surface, $H\in \Pic(S)$ an ample line bundle and $h\in\oH^2(S,\ZZ)$ its class. 
Then 
\[ \Mon^2(S,H)=\Or^+(\oH^2(S,\ZZ))_h, \]
where $\Or^+(\oH^2(S,\ZZ))_h$ is the group of orientation preserving isometries $g$ such that $g(h)=h$.

Moreover, the whole group $\Or^+(\oH^2(S,\ZZ))_h$ arises as the monodromy group of a projective polarised family $f\colon\mathcal{S}\to T$ of K3 surface over a smooth and quasi-projective base $T$.
\end{lemma} 
The smooth and quasi-projective base $T$ of the family above is the image under the period map of the moduli space of polarised K3 surfaces with fixed degree (see \cite[Section~6.3]{Markman:Monodromy} for more details).

The last ingredient we need is the following lattice-theoretic result, which is a straighforward application of the Eichler criterion. We remark that a very similar result already appeared in the proofs of \cite[Lemma~3.5]{MonRap} and \cite[Theorem~5.4]{Onorati:Monodromy} for similar purposes. 

Let $L$ be an even lattice and let us assume that $L$ contains at least three copies of the unimodular hyperbolic plane $U$. We denote by $U$ one distinguished such copy, with basis $\{e,f\}$, and we write $L=U\oplus L_1$. Given $s,r,p>1$, we use the following notation,
\[ h_1=e+rf,\qquad h_2=e+(r-1)f,\qquad h_3=se+pf,\quad\mbox{and}\quad h_4=(s-1)e+pf. \] 

\begin{lemma}\label{lemma:eichler}
With the notation above, we have
\[ \Or^+(L)=\langle \Or^+(L)_{h_1},\Or^+(L)_{h_2},\Or^+(L)_{h_3},\Or^+(L)_{h_4}\rangle. \]
\end{lemma}
Here $\Or^+(L)_{h_i}$ is the subgroup of $\Or^+(L)$ of the elements fixing $h_i$. 
\begin{proof}
First of all, by \cite[Proposition~3.3 (iii)]{GHS09}, we have that 
\[ \Or^+(L)=\langle \Or^+(L_1), E_{\bar{U}}(L_1)\rangle, \]
where $\Or^+(L_1)$ is embedded in $\Or^+(L)$ by extending any isometry as the identity on $U$, and where $E_{U}(L_1)$ is the group generated by transvections of the form $t(e,a)$ and $t(f,a)$ for all $a\in L_1$. 
We recall that a transvection is defined in the following way: for any $z\in L$ with $z^2=0$ and $a\in z^\perp$, then
\[ t(z,a)\colon x\mapsto x-(a,x)z+(z,x)a-\frac{1}{2}(a,a)(z,x)z. \]
We refer to \cite[Section~3]{GHS09} for backgrounds and main properties, but we point out that if $a^2\neq0$, then the extension of $t(z,a)$ over $\mathbb{Q}$ is the composition of the reflection $R_{a}$ with the reflection $R_{a+\frac{1}{2}(a,a)z}$. Any transvection $t(z,a)$ is an orientation preserving isometry with determinant $1$ (acting trivially on the discriminant group of $L$).

As $h_i\in U$, it follows that $\Or^+(L_1)\subset\Or^+(L)_{h_i}$ for $i=1,2,3,4$. Therefore, if we denote by $G$ the group generated by the $\Or^+(L)_{h_i}$'s, we only need to show that $t(e,a),t(f,a)\in G$ for all $a\in L_1$.

By assumption $L_1$ contains at least two copies of the hyperbolic plane, so that we can apply \cite[Proposition~3.3~(ii),]{GHS09} and find an isometry $g\in\Or^+(L_1)$ such that $g(a)\in \bar{U}$, where $\bar{U}$ is a distinguished copy of the hyperbolic plane in $L_1$.
By \cite[Relation~(6) in Section~3]{GHS09} we have that 
\[ g\circ t(z,a)\circ g^{-1}=t(z,g(a)) \]
for any isotropic element $z$. In particular, since we already remarked that $\Or^+(L_1)\subset G$, it is enough to prove that $t(e,a),t(f,a)\in G$ for all $a\in \bar{U}$. 

As already remarked, $t(e,a)$ and $t(f,a)$ are orientation preserving isometries with determinant $1$; moreover, since $a\in \bar{U}$, $t(e,a)$ and $t(f,a)$ both act as the identity on $(U\oplus \bar{U})^\perp$, and without loss of generality they can then be viewed as elements in $\operatorname{SO}^+(U\oplus \bar{U})$. 

If we denote by $\{\bar{e},\bar{f}\}$ a standard basis for $\bar{U}$, then by \cite[Lemma~3.2]{GHS09} we know that $\operatorname{SO}^+(U\oplus \bar{U})$ is generated by the four transvections
\[ t(\bar{e},e),\qquad t(\bar{e},f), \qquad t(\bar{f},e)\qquad\mbox{and}\qquad t(\bar{f},f). \]
So it is enough to prove that each of them is in $G$. Let us show that $t(\bar{e},f)\in G$, the others will follow similarly.

We use the following remark, which follows directly from the definition of transvection: if $a\in U^\perp$, then $t(e,a)$ and $t(f,a)$ are the identity on the sublattice $U^\perp\cap a^\perp$. In particular, for example, $t(\bar{e},e-tf)$ acts as the identity on $e+tf$.

By \cite[Relation~(5) in Section~3]{GHS09}, we have
in general that $t(z,a)\circ t(z,b)=t(z,a+b)$ and $t(z,-a)=t(z,a)^{-1}$. In our situation this says that
\[ t(\bar{e},e-rf)\circ t(\bar{e},-e+(r-1)f)=t(\bar{e},f)^{-1}, \]
but as remarked $t(\bar{e},e-rf)\in\Or^+(L)_{h_1}$ and similarly $t(\bar{e},-e+(r-1)f)=t(\bar{e},e-(r-1)f)^{-1}\in\Or^+(L)_{h_1}$, concluding the proof.
\end{proof}

\proof[Proof of Proposition~\ref{prop:mon S pol}]
Let us keep the notations as in the statement of Proposition~\ref{prop:mon S pol}. By Lemma~\ref{lemma:pol} we have that $\Mon^2(S,H_i)=\Or^+(\oH^2(S,\ZZ))_{h_i}$, for $i=1,2,3,4$. In particular, by Lemma~\ref{lemma:eichler}, we have that 
\[ \Or^+(\oH^2(S,\ZZ))=\langle \Mon^2(S,H_1),\Mon^2(S,H_2),\Mon^2(S,H_3),\Mon^{2}(S,H_{4})\rangle. \]
Since $\Mon^2(S)\subset\Or^+(\oH^2(S,\ZZ))$ (see Lemma~\ref{lemma:Mon is O+}), the claim follows.
\endproof
Let us conclude with the following corollary, which recover a well-known result of Borel.
\begin{cor}[\cite{Bor}]\label{corollary:Borel}
    If $S$ is a K3 surface, then
    \[ \Mon^2(S)=\Or^+(\oH^2(S,\ZZ)). \]
\end{cor}

\subsection{Lift of the polarised monodromy of a K3 surface to moduli spaces of sheaves}\label{section:lift}

In this section we exhibit some special cases in which we can lift polarised monodromy operators of a K3 surface $S$ to monodromy operators on a moduli space $M_v(S,H)$, where $(S,v,H)$ is a suitable $(m,k)$-triple. We will achieve this by using the groupoid representations defined in Section~\ref{section:representation}.
\medskip

Since $\oH^2(S,\ZZ)\subset\widetilde{\oH}(S,\ZZ)$, the group $\Or(\oH^2(S,\ZZ))$ is identified with the subgroup of $\Or(\widetilde{\oH}(S,\ZZ))$ of isometries acting as the identity on $\oH^0(S,\ZZ)\oplus\oH^4(S,\ZZ)$.

On the other hand, if $(S,v,H)$ is an $(m,k)$-triple, using the notation set in Section~\ref{section:groupoids}, we consider the homomorphism of groups
\[ \widetilde{\Phi}^{m,k}_{\operatorname{def},(S,v,H)}\colon\Aut_{\mathcal{G}_{\operatorname{def}}^{m,k}}(S,v,H)\longrightarrow\Aut_{\widetilde{\mathcal{H}}^{m,k}}(\tilde{\oH}(S,\ZZ),v,\epsilon_S) \]
induced by the representation $\widetilde{\Phi}^{m,k}_{\operatorname{def}}\colon\mathcal{G}_{\operatorname{def}}^{m,k}\to\widetilde{\mathcal{H}}^{m,k}$ defined in Definition~\ref{defn:Phi tilde def m k}. Since by definition
\[ \Aut_{\widetilde{\mathcal{H}}^{m,k}}(\tilde{\oH}(S,\ZZ),v,\epsilon_S)=\Or(\widetilde{\oH}(S,\ZZ))_v, \]
it follows that also
\[ \Im(\widetilde{\Phi}^{m,k}_{\operatorname{def},(S,v,H)})\subset\Or(\widetilde{\oH}(S,\ZZ)). \]
In the following lemma, we relate the subgroup $\Mon^2(S,H)\subset\Or(\oH^2(S,\ZZ))$ with $\Im(\widetilde{\Phi}^{m,k}_{\operatorname{def},(S,v,H)})$, at least when the Mukai vector is $(m,0,-mk)$. 

We will also relate $\Mon^2(S,H)$ with the image of the homomorphism of groups
\[ \mathsf{pt}^{m,k}_{\operatorname{def},(S,v,H)}\colon\Aut_{\mathcal{G}_{\operatorname{def}}^{m,k}}(S,v,H)\longrightarrow\Aut_{\mathcal{A}_k}(\oH^2(M_v(S,H),\ZZ)) \]
induced by the representation $\mathsf{pt}^{m,k}_{\operatorname{def}}\colon\mathcal{G}_{\operatorname{def}}^{m,k}\to\mathcal{A}_k$ defined in Definition~\ref{defn:pt m k def}.

\begin{lemma}\label{lemma:pto in Im(Phi tilde) when c=0}
    Let $m,k\geq1$ be two integers and $v=(m,0,-mk)$ a Mukai vector on a projective K3 surface $S$.
    If $(S,v,H)\in\mathcal{G}^{m,k}_{\operatorname{def}}$ is an object, then there is an inclusion of groups
        \[ \Mon^2(S,H)\subset\Im(\widetilde{\Phi}^{m,k}_{\operatorname{def},(S,v,H)}) \]
    and an injective morphism of groups
        \[  \Mon^2(S,H)\hookrightarrow\Im(\mathsf{pt}^{m,k}_{\operatorname{def},(S,v,H)}).
        \]
\end{lemma}

\begin{proof}
    Let us start with the inclusion of $\Mon^2(S,H)$ in $\Im(\widetilde{\Phi}^{m,k}_{\operatorname{def},(S,v,H)})$.
    Recall from Lemma~\ref{lemma:pol} that there exists a projective polarised family $f\colon\mathcal{S}\to T$ of projective K3 surfaces such that any isometry in $\Mon^2(S,H)$ arises as a parallel transport operator in the local system $R^2 f_*\ZZ$, and such that the base $T$ is smooth and quasi-projective. More precisely, the family $f$ comes with the following properties: 
    \begin{itemize}
       \item there exists $\bar{t}\in T$ such that $\mathcal{S}_{\bar{t}}=S$;
       \item there exists a relatively ample line bundle $\mathcal{H}$ on $\mathcal{S}$ such that $\mathcal{H}_{\bar{t}}=H$.
    \end{itemize}    
    Any $g\in\Mon^2(S,H)$, is then associated to a loop $\gamma$ in $T$ centred in $\bar{t}$.
    
    As in the discussion before the lemma, we can view $g$ as an element of $\Aut_{\widetilde{\mathcal{H}}^{m,k}}(\widetilde{\oH}(S,\ZZ),v,\epsilon_S)$, and we have to prove that there is a deformation path $\alpha$ from $(S,v,H)$ to itself such that $g=\widetilde{\Phi}_{\operatorname{def}}^{m,k}(\bar{\alpha})$.
    
    To do so, we first notice that the subset $Z\subset T$ of the points $t\in T$ such that $\mathcal{H}_t$ is not $v_{t}$-generic is a closed subset of $T$ (see Remark~\ref{rmk:defomk}). Moreover, since by assumptions $\mathcal{H}_{\bar{t}}$ is $v$-generic, $Z$ is strictly contained in $T$ and hence it must have real codimension at least $2$. In particular, since the base $T$ is smooth, there exists a loop $\gamma'$ in $T':=T\setminus Z$ centred in $\bar{t}$ that is homotopic to $\gamma$ (see for example \cite[Théorème~2.3 in Chapter~X]{Godbillon}). 

    Since parallel transport operators do not depend on the homotopy class of the path, we may suppose without loss of generality that $\mathcal{H}_t$ is $v$-generic for every $t\in T$.
    
    By assumption the Mukai vector is of the form $(m,0,-mk)$, therefore it follows that 
    \[ \alpha=(f\colon\mathcal{S}\to T, \mathcal{O}_{\mathcal{S}},\mathcal{H},\bar{t},\bar{t},\gamma) \]
    defines a morphism $\overline{\alpha}$ in $\Aut_{\mathcal{G}^{m,k}_{\operatorname{def}}}(S,v,H)$. By definition, we get that $$g=\widetilde{\Phi}^{m,k}_{\operatorname{def}}(\overline{\alpha}),$$
    concluding the first part of the proof.

    Now, since by Remark~\ref{rmk:pto on H tilde} the isometries of $\Mon^2(S,H)$ preserve the orientation of $\widetilde{\oH}(S,\ZZ)$, it follows that the inclusion just proved gives an injective map
    \[ \Mon^2(S,H)\hookrightarrow\Im(\Phi^{m,k}_{\operatorname{def},(S,v,H)}). \]
    Combining this with the isomorphism of functors $\lambda_{\operatorname{def}}\colon\Phi^{m,k}_{\operatorname{def}}\to\mathsf{pt}^{m,k}_{\operatorname{def}}$ constructed in the proof of Proposition~\ref{prop:iso of functors} finishes the proof.
    \end{proof}

\begin{rmk}
    The claim of Lemma~\ref{lemma:pto in Im(Phi tilde) when c=0} holds more generally when the Mukai vector has the form $(r,0,s)$, but we will only use it in the special case of the statement.
\end{rmk}

\begin{rmk}
    If $v=(m,0,-mk)$, then $\oH^2(S,\ZZ)\subset v^\perp$ and again the group $\Or(\oH^2(S,\ZZ))$ is naturally embedded in $\Or(v^\perp)$ by extending with the identity. Using the isometry 
    $\lambda_{(S,v,H)}\colon v^\perp\longrightarrow\oH^2(M_v(S,H),\ZZ)$,
    we have a natural injective morphism
    \[ \Or(\oH^2(S,\ZZ))\hookrightarrow\Or(\oH^2(M_v(S,H),\ZZ) \]
    of which 
    \[  \Mon^2(S,H)\hookrightarrow\Im(\mathsf{pt}^{m,k}_{\operatorname{def},(S,v,H)})
    \]
    is the restriction.     
\end{rmk}

The following is the main result of this section. In what follows by a \emph{very general projective elliptic K3 surface with a section} we mean a projective elliptic K3 surface having a section and having Picard rank $2$. As in Section~\ref{section:Mon S}, we denote by $f$ the class of a fibre and by $\ell$ the class of the section. Then $\Pic(S)$ is isometric to the unimodular hyperbolic plane with standard basis $e=\ell+f$ and $f$.

\begin{prop}\label{prop:Mon S in m k}
Assume that $v=(m,0,-mk)$ and $S$ is a very general projective elliptic K3 surface with a section. There exists an integer $t\gg0$ and a $v$-generic polarisation $H=e+tf$ such that 
\[ \Mon^2(S)\subset\Im(\widetilde{\Phi}^{m,k}_{\operatorname{def},(S,v,H)}). \]
\end{prop}
\begin{proof}
    By Proposition~\ref{prop:mon S pol} we know that there are four polarisations $H_1$, $H_2$, $H_3$ and $H_4$ on $S$ such that $\Mon^2(S)$ is generated by the $\Mon^2(S,H_i)$'s.

    Now, we claim that the four polarisations $H_1$, $H_2$, $H_3$ and $H_4$ can be chosen to belong to the same $v$-chamber (see Remark~\ref{rmk:wall and chamber}), so that they are $v$-generic. 

    In particular, by Lemma~\ref{lemma:pto in Im(Phi tilde) when c=0} we will have that 
    \[ \Mon^2(S,H_i)\subset\Im(\widetilde{\Phi}^{m,k}_{\operatorname{def},(S,v,H_i)}), \]
    and by Remark~\ref{rmk:triple congruenti} the groups $\Im(\widetilde{\Phi}^{m,k}_{\operatorname{def},(S,v,H_i)})$ can all be identified: the proposition will then follow from Proposition~\ref{prop:mon S pol}.

    Indeed, let us recall how the polarisations $H_i$ are chosen. If $f$ is the class of the fibration of the elliptic surface $S$ and $\ell$ is the class of the section, then we put $e=\ell+f$ and we choose
    \[ H_1=e+rf,\quad H_2=e+(r-1)f,\quad H_3=se+pf,\quad\mbox{and}\quad H_4=(s-1)e+pf \]
    for $r,k,p\gg0$. Now, by \cite[Lemma~I.0.3]{OGrady:WeightTwo} (see also \cite[Definition~2.37 and Lemma~2.38]{PR:SingularVarieties}), if $r$ is big enough and $s$ and $p$ are chosen such that the quotient $p/s$ is big enough, then all the $H_i$'s are contained in the unique $v$-chamber whose closure contains the class $f$. 

    We may now conclude the proof: choose a polarisation $H$ on $S$ that lies in the same $v$-chamber of $H_{1},\cdots,H_{4}$, i.e.\ $H$ belongs to the unique $v$-chamber whose closure contains $f$. By definition of $\mathcal{G}^{m,k}_{\operatorname{def}}$, as $H$ and $H_{i}$ lie in the closure of the same $v$-chamber for $i=1,2,3,4$, we have an isomorphism $\chi_{H,H_{i}}\colon(S,v,H)\to(S,v,H_{i})$ in $\mathcal{G}^{m,k}$ and hence an isomorphism of groups $$\chi_{H,H_{i}}^{\sharp}\colon\Aut_{\mathcal{G}^{m,k}_{\operatorname{def}}}(S,v,H)\longrightarrow \Aut_{\mathcal{G}^{m,k}_{\operatorname{def}}}(S,v,H_{i})$$
    given by conjugation with $\chi_{H,H_{i}}$. Since $\widetilde{\Phi}^{m,k}_{\operatorname{def}}(\chi_{H,H_{i}})$ is the identity, we get that $\Im(\widetilde{\Phi}^{m,k}_{\operatorname{def},(S,v,H_i)})=\Im(\widetilde{\Phi}^{m,k}_{\operatorname{def},(S,v,H)})$ for every $i=1,2,3,4$, and this concludes the proof.
    \end{proof}

\begin{rmk}
    If $v$ is primitive, then the same result was proved in \cite[Corollary~6.7]{Markman:Monodromy}.
    We wish to point out that even though the idea of our proof and Markman's one is the same, they differ inasmuch as Markman uses non-polarised deformations of K3 surfaces (compare \cite[Definition~6.2]{Markman:Monodromy} with our Definition~\ref{defn:G m k def}). If $v$ is primitive this is justified by the fact that $M_{(1,0,-k)}(S,H)\cong\Hilb^{k+1}(S)$. On the other hand, if $v$ is not primitive we are forced to use polarised families.
\end{rmk}

\begin{cor}
    Assume that $v=(m,0,-mk)$ and $S$ is a very general projective elliptic K3 surface with a section. If $H$ is a $v$-generic polarisation contained in the unique $v$-chamber whose closure contains the nef class $f$, then there is an injective morphism
    \[ \Mon^2(S)\hookrightarrow\Mon^2_{\operatorname{lt}}(M_v(S,H)). \]
\end{cor}
\begin{proof}
     As in the last part of the proof of Lemma~\ref{lemma:pto in Im(Phi tilde) when c=0}, by Remark~\ref{rmk:pto on H tilde} and Proposition~\ref{prop:iso of functors} the inclusion of Proposition~\ref{prop:Mon S in m k} gives an injective morphism
     \[ \Mon^2(S)\hookrightarrow\Im(\mathsf{pt}^{m,k}_{\operatorname{def},(S,v,H)}). \]
     Then the claim follows from Corollary~\ref{cor:Im of pt in Mon}.
\end{proof}

Finally, we want to exhibit an analogous version of Lemma~\ref{lemma:pto in Im(Phi tilde) when c=0} when the K3 surface is general enough, which will be useful later in Section~\ref{section:W in m k}.

\begin{lemma}\label{lemma:Mon S pol con H}
    Let $S$ be a projective K3 surface with $\Pic(S)=\ZZ.H$ and $v=(r,mH,s)$ a Mukai vector. If $(S,v,H)$ is an $(m,k)$-triple, then there exists an inclusion of groups
        \[ \Mon^2(S,H)\subset\Im(\widetilde{\Phi}^{m,k}_{\operatorname{def},(S,v,H)}) \]
    and an injective morphism of groups
        \[  \Mon^2(S,H)\hookrightarrow\Im(\mathsf{pt}^{m,k}_{\operatorname{def},(S,v,H)}).
        \]
\end{lemma}
\begin{proof}
    The proof runs verbatim as the proof of Lemma~\ref{lemma:pto in Im(Phi tilde) when c=0}, the only difference being now that, with the same notation, the deformation path is
    \[ \alpha=(f\colon\mathcal{S}\to T,\mathcal{H}^{\otimes m},\mathcal{H},\bar{t},\bar{t},\gamma). \]
\end{proof}

\section{The locally trivial monodromy group}\label{section:main}
Let $(S,v,H)$ be an $(m,k)$-triple. As usual we write $v=mw$, with $w$ primitive, and we consider the $(1,k)$-triple $(S,w,H)$ (cf.\ Remark~\ref{rmk:gen}).

By Remark~\ref{rmk:strat M} we can identify the (smooth) moduli space $M_w(S,H)$ with the most singular locus of the (singular) moduli space $M_v(S,H)$. In particular we consider the closed embedding 
\[ i_{w,m}\colon M_w(S,H)\longrightarrow M_v(S,H). \]

The aim of this section is to relate the locally trivial monodromy group of $M_v(S,H)$ to the monodromy group of $M_w(S,H)$ by mean of the morphism $i_{w,m}^\sharp$ defined in Lemma~\ref{lemma:i sharp}. 

In Section~\ref{section:W in m k} we construct locally trivial monodromy operators and show that those operators form a distinguished group.

In Section~\ref{subsection:i sharp} we prove that $i_{w,m}^\sharp$ is injective.

In Section~\ref{section:main result} we show that the monodromy operators constructed before are all and only by using the morphism $i_{w,m}^\sharp$ as a constraint; as a corollary we get the fact that $i_{w,m}^\sharp$ is an isomorphism.

\subsection{The group $\mathsf{W}(v^\perp)$ as a subgroup of $\Mon^2_{\operatorname{lt}}(M_{v}(S,H))$}\label{section:W in m k}

Let us recall the following notation. If $L$ is an abstract lattice, the nondegenerate bilinear form embeds $L$ into its dual $L^*=\Hom(L,\ZZ)$: the group $L^*/L$ is denoted by $A_L$ and called \emph{discriminant group} of $L$. Any isometry $g\in\Or(L)$ extends to an isometry on $L^*$ and hence it induces an automorphism of $A_L$. We denote by $\sfW(L)\subset\Or^+(L)$ the subgroup of orientation preserving isometries acting as $\pm\id$ on $A_{L}$. 

When $L=\oH^2(X,\ZZ)$ for an irreducible symplectic variety $X$, then we simply write $\mathsf{W}(X)$ instead of $\mathsf{W}(\oH^2(X,\ZZ))$.

\begin{rmk}\label{rmk:W as ext}
Let $S$ be a projective K3 surface and $v\in\widetilde{\oH}(S,\ZZ)$ an element. The group $\sfW(v^\perp)$ consists of orientation preserving isometries of $v^{\perp}$ that extend to isometries of $\widetilde{\oH}(S,\mathbb{Z})$. More precisely, if $g\in\sfW(v^\perp)$, then there exists $\tilde{g}\in\Or^+(\widetilde{\oH}(S,\ZZ))$ such that $\tilde{g}(v)=\pm v$ and $\tilde{g}|_{v^\perp}=g$ (see \cite[Lemma~4.10]{Markman:Monodromy}).
\end{rmk}

The main result of this section is the following.
\begin{thm}\label{thm:W in m k}
    Let $(S,v,H)$ be an $(m,k)$-triple. Then 
    \[ \mathsf{W}(M_v(S,H))\subset\Mon^2_{\operatorname{lt}}(M_{v}(S,H)). \]
\end{thm}

\begin{rmk}
    When $v$ is primitive, this statement is the main result in \cite{Markman:Monodromy}. In fact our proof parallels Markman's one and only differs from his in some technical details arising from the non-primitivity of the Mukai vector.
\end{rmk}

Our tool to produce locally trivial monodromy operators is via the representation $\mathsf{pt}^{m,k}\colon\mathcal{G}^{m,k}\to\mathcal{A}_k$ defined in Definition~\ref{defn:pt m k} (cf.\ Corollary~\ref{cor:Im of pt in Mon}); on the other end, because of the isomorphism $\lambda\colon\mathsf{\Phi}^{m,k}\to\mathsf{pt}^{m,k}$ (cf.\ Proposition~\ref{prop:iso of functors}), we will instead look at the representation $\Phi^{m,k}\colon\mathcal{G}^{m,k}\to\mathcal{A}_k$ defined in Definition~\ref{defn: Phi tilde m k}. 

Since $\Phi^{m,k}=\Psi\circ\widetilde{\Phi}^{m,k}$, our most technical result involve the representation $\widetilde{\Phi}^{m,k}\colon\mathcal{G}^{m,k}\to\widetilde{\mathcal{H}}^{m,k}$ defined in Definition~\ref{defn: Phi tilde m k}. 

If $(S,v,H)$ is an $(m,k)$-triple, then we have the induced homomorphism
\[ \widetilde{\Phi}^{m,k}_{(S,v,H)}\colon\Aut_{\mathcal{G}^{m,k}}(S,v,H)\longrightarrow\Aut_{\widetilde{\mathcal{H}}^{m,k}}(\widetilde{\oH}(S,\ZZ),v) \]
(cf.\ Section~\ref{section:groupoids} for the notation used).

\begin{prop}\label{prop:Im of Phi tilde = O}
    If $(S,v,H)$ is an $(m,k)$-triple, then 
    \[ \Im(\widetilde{\Phi}^{m,k}_{(S,v,H)})=\Or(\widetilde{\oH}(S,\ZZ))_v. \]
\end{prop}

We will now show how the theorem follows from the proposition and then we will spend the rest of the section proving the proposition.

\proof[Proof of Theorem~\ref{thm:W in m k} assuming Proposition~\ref{prop:Im of Phi tilde = O}]
By definition, if $(\Lambda,v,\epsilon)$ is an object of $\widetilde{\mathcal{H}}^{m,k}$, then we have an homomorphism of groups 
\[ \Psi_{(\Lambda,v,\epsilon)}\colon\Or(\Lambda)_v\longrightarrow\Or^+(v^\perp),\qquad g\mapsto(-1)^{\operatorname{or}(g)}g|_{v^\perp} \]
induced by the representation $\Psi\colon\widetilde{\mathcal{H}}^{m,k}\to\mathcal{A}_k$ defined in Definition~\ref{defn:Psi}.
Moreover, by \cite[Lemma~4.10]{Markman:Monodromy} it follows that 
\[ \Im(\Psi_{(\Lambda,v,\epsilon)})=\mathsf{W}(v^\perp). \]

On the other hand, if $(S,v,H)$ is an $(m,k)$-triple, by Proposition~\ref{prop:Im of Phi tilde = O} the homomorphism
\[ \widetilde{\Phi}^{m,k}\colon\Aut_{\mathcal{G}^{m,k}}(S,v,H)\longrightarrow\Aut_{\widetilde{\mathcal{H}}^{m,k}}(\widetilde{\oH}(S,\ZZ),v,\epsilon_S) \]
is surjective. Since by definition $\Phi^{m,k}=\Psi\circ\widetilde{\Phi}^{m,k}$, it follows that
\begin{equation}\label{eqn:Im of Phi = W}
\Im(\Phi^{m,k}_{(S,v,H)})=\mathsf{W}(v^\perp). 
\end{equation}

Now, let us consider the isometry
\[ \lambda_{(S,v,H)}\colon v^\perp\longrightarrow\oH^2(M_v(S,H),\ZZ) \]
in Theorem~\ref{thm:PR v perp} and the induced isomorphism
\[ \lambda_{(S,v,H)}^\sharp\colon \Or(v^\perp)\longrightarrow\Or(\oH^2(M_v(S,H),\ZZ)) \]
defined in equation (\ref{eqn:lambda sharp}). Then we have a chain of equalities:
\begin{align*}
\Im(\mathsf{pt}^{m,k}_{(S,v,H)}) & =\lambda_{(S,v,H)}^\sharp(\Im(\Phi^{m,k}_{(S,v,H)})) & \mbox{(by Proposition~\ref{prop:iso of functors})}\\
 & =\lambda_{(S,v,H)}^\sharp(\mathsf{W}(v^\perp)) & \mbox{(by equality~(\ref{eqn:Im of Phi = W}))} \\
 & =\mathsf{W}(M_v(S,H)). &  
 \end{align*}

Therefore the claim follows at once by Corollary~\ref{cor:Im of pt in Mon}.
\endproof

Let us now prove Proposition~\ref{prop:Im of Phi tilde = O}.

\proof[Proof of Proposition~\ref{prop:Im of Phi tilde = O}]
Notice that 
\[ \Im(\widetilde{\Phi}^{m,k}_{(S,v,H)})\subset\Aut_{\widetilde{\cH}^{m,k}}(\widetilde{\oH}(S,\ZZ),v)=\Or(\widetilde{\oH}(S,\ZZ))_v, \]
where the last equality follows by definition.

To prove the opposite inclusion,  we first notice that it is sufficient to prove it for a special choice of the $(m,k)$-triple $(S,v,H)$. 
Indeed, if $(S',v',H')$ is any other $(m,k)$-triple, then by Remark~\ref{rmk:non empty} there exists a non-zero element 
\[ \eta\in\Hom_{\mathcal{G}^{m,k}}((S,v,H),(S',v',H')) \] 
and 
\[ \Im(\widetilde{\Phi}^{m,k}_{(S',v',H')})=\widetilde{\Phi}^{m,k}(\eta)\circ \Im(\widetilde{\Phi}^{m,k}_{(S,v,H)} )\circ\widetilde{\Phi}^{m,k}(\eta)^{-1}. \]
Therefore if $\Im(\widetilde{\Phi}^{m,k}_{(S,v,H)})=\Or(\widetilde{\oH}(S,\ZZ))_v$, as $\widetilde{\Phi}^{m,k}(\eta)$  maps $v$ to $v'$, it follows that $\Im(\widetilde{\Phi}^{m,k}_{(S',v',H')})=\Or(\widetilde{\oH}(S',\ZZ))_{v'}$.

Because of this, we will prove the proposition in the case when $v=(m,0,-mk)$ and $S$ is a projective elliptic K3 surface with a section and having Picard rank $2$. For short, in the following we will refer to such a K3 surface as a \emph{very general projective elliptic K3 surface with a section}. If we denote by $p\colon S\to\PP^1$ the elliptic structure of $S$ and by $s\colon\PP^1\to S$ the section, then we put $f=p^*\cO_{\PP^1}(1)$ and $\ell=[s(\PP^1)]$.
In this case 
\[ \Pic(S)=\ZZ.\ell\oplus\ZZ.f=
\left(\begin{matrix}
-2 & 1 \\
1  & 0
\end{matrix}\right),\]  
and we will always denote by $e=\ell+f$ the other isotropic generator.

Finally, the polarisation $H$ will be chosen in the unique $v$-chamber (see Remark~\ref{rmk:wall and chamber}) whose closure contains the nef class $f$. 
\medskip

As a first step, let us recall the following lattice-theoretic result due to Markman, which exhibits a set of generators for the group $\Or^+(\widetilde{\oH}(S,\ZZ))_v$, when $S$ is now any K3 surface.

\begin{prop}[\protect{\cite[Lemma~8.1, Proposition~8.6]{Markman:Monodromy}}]\label{prop:boh}
Let $S$ be a K3 surface, $v=(m,0,-mk)\in\widetilde{\oH}(S,\mathbb{Z})$ and $\xi\in\oH^2(S,\ZZ)$ any primitive class such that $\xi^2=2k-2$. The group $\Or^+(\widetilde{\oH}(S,\ZZ))_v$ is generated by $\Or^+(\oH^2(S,\ZZ))$ and the reflection $R_u$ around the class $u=(1,\xi,k)\in v^\perp$. 
\end{prop}

Here, as usual, we consider   the group $\Or^+(\oH^2(S,\ZZ))$  as a subgroup of $\Or^+(\widetilde{\oH}(S,\ZZ))$ by extending by the identity on $\oH^0(S,\ZZ))\oplus \oH^4(S,\ZZ)$, and since $v\in \oH^0(S,\ZZ))\oplus \oH^4(S,\ZZ)$,  we actually view the group  $\Or^+(\oH^2(S,\ZZ))$ as  a subgroup of $
\Or^+(\widetilde{\oH}(S,\ZZ))_v$.

\begin{proof}
By \cite[Proposition~8.6]{Markman:Monodromy}, $\Or^+(\widetilde{\oH}(S,\ZZ))_v$ is generated by the group $\Or^+(\oH^2(S,\ZZ))$ and all the reflections $R_u$ as in the statement. The reason why only one is in fact needed is a consequence of \cite[Theorem~1.14.4, Theorem~1.17.1]{Nikulin} as it is explained in the proof of \cite[Lemma~8.1]{Markman:Monodromy}: if $u'=(1,\xi',k)$ is another $(-2)$-class, then there exists an isometry $g\in\Or^+(\oH^2(S,\ZZ))$ such that $g(\xi')=\xi$ and by extending $g$ to $\Or^+(\widetilde{\oH}(S,\ZZ))$ by the identity, it follows that $g\circ R_{u'}\circ g^{-1}=R_u$. In other words, any two such reflections are conjugated by elements of $\Or^+(\oH^2(S,\ZZ))$.
\end{proof}

Therefore, in order to prove Proposition~\ref{prop:Im of Phi tilde = O}, we need to exhibit examples $(m,k)$-triples  such that the generators in the Proposition \ref{prop:boh} arise as morphisms in $\Im(\widetilde{\Phi}^{m,k}_{(S,v,H))})$. 

By Proposition~\ref{prop:Mon S in m k}, there exists $t\gg0$ and a polarisation $v$-generic polarisation $H=e+tf$ on the very generic elliptic surface $S$ such that 
\[ \Mon^2(S)\subset\Im(\widetilde{\Phi}^{m,k}_{(S,v,H))}). \]

Since by Corollary~\ref{corollary:Borel} we have that $\Mon^2(S)=\Or^+(\oH^2(S,\ZZ))$, it follows that 
\begin{equation}\label{eqn:O + in Phi tilde}
\Or^+(\oH^2(S,\ZZ))\subset \Im(\widetilde{\Phi}^{m,k}_{(S,v,H))}). 
\end{equation}

Now, let $\beta\in\oH^2(S,\ZZ)$ be a class such that $\beta^2=2k-2$ and such that $\beta$ is orthogonal to both $f$ and $\ell$. Let us consider the class $u=(1,\beta-f,k)\in v^\perp$. Notice that $u^2=-2$.

\begin{lemma}[\protect{\cite[Proposition~7.1]{Markman:Monodromy}}]\label{lemma:FM P in Im Phi tilde}
If $H=\ell+tf$ with $t\gg0$, then 
\[ R_u\in\Im(\widetilde{\Phi}^{m,k}_{(S,v,H)}). \]
\end{lemma}

\begin{proof} 
Since $H=\ell+tf$ with $t\gg0$, by Proposition~\ref{prop:Poincare} the derived equivalence $\operatorname{FM}_{\mathcal{P}}$ induces an isomorphism
\[ \operatorname{FM}_{\mathcal{P}}\colon M_{(m,0,-mk)}(S,H)\longrightarrow M_{(0,m(\ell+(k+1)f),m)}(S,H). \]
For simplicity, let us put 
\[ v=(m,0,-mk)\quad\mbox{ and }\quad \bar{v}=(0,m(\ell+(k+1)f),m). \]
Then 
\[ \operatorname{FM}_{\mathcal{P}}\in\Hom_{\mathcal{G}^{m,k}_{\operatorname{FM}}}((S,v,H),(S,\bar{v},H)). \]
The cohomological action $\operatorname{FM}_{\mathcal{P}}^{\oH}$ has been studied in \cite[Lemma~7.2]{Markman:Monodromy} and 
    \[ \operatorname{FM}_{\mathcal{P}}^{\oH}(1,\beta-f,k)=(0,\ell-(k-1)f-\beta,0). \]
    Put $a=\ell-(k-1)f-\beta\in\oH^2(S,\ZZ)$. Then $a^2=-2$ and $R_{a}\in\Mon^2(S)=\Or^+(\oH^2(S,\ZZ))$.
    Notice also that $(0,a,0)\in\bar{v}^\perp$.

    Since 
    \[ R_u=\operatorname{FM}^{\operatorname{H}}_{\mathcal{P}}\circ R_{(0,a,0)}\circ(\operatorname{FM}^{\operatorname{H}}_{\mathcal{P}})^{-1}\]
    the proof will be concluded as soon as we prove the following claim.
    
    \begin{claim}    $R_{(0,a,0)}\in\Im(\widetilde{\Phi}^{m,k}_{\operatorname{def},(S,\bar{v},H)})$.
    \end{claim}
    
    First of all, let us consider a deformation path \[ \alpha=(f\colon\mathcal{S}\to T,\mathcal{L},\mathcal{H},t_1,t_2,\gamma) \]
    where 
    \[ (\mathcal{S}_{t_1},\mathcal{L}_{t_1},\mathcal{H}_{t_1})=(S,m(\ell+(k+1)f),H)\quad\mbox{and}\quad(\mathcal{S}_{t_2},\mathcal{L}_{t_2},\mathcal{H}_{t_2})=(S',L',H')
    \]
    with $\Pic(S')=\ZZ H'$ and $L'=m H'$. Put $v'=v_{t_2}$.
    
    Up to conjugating with $\widetilde{\Phi}(\bar{\alpha})$, it is enough to prove that 
    $R_{b}$ is an element of $\Im(\widetilde{\Phi}^{m,k}_{\operatorname{def},(S',v',H')})$, where $b=\widetilde{\Phi}(\bar{\alpha})(0,a,0)$. On the other hand the last claim follows from Lemma~\ref{lemma:Mon S pol con H}: in fact the triple $(S',L',H')$ satisfies its hypothesis and $R_b\in\Or^+(\oH^2(S,\ZZ))_{H'}$ by construction.
\end{proof}

    Combining equality~(\ref{eqn:O + in Phi tilde}) and Lemma~\ref{lemma:FM P in Im Phi tilde}, by Proposition~\ref{prop:boh} we have that 
    \[ \Or^+(\widetilde{\oH}(S,\ZZ))_v\subset\Im(\widetilde{\Phi}^{m,k}_{(S,v,H)}), \]
    where $H$ is a $v$-generic polarisation inside the $v$-chamber whose closure contains the class $f$.
    
    To conclude the proof of Proposition~\ref{prop:Im of Phi tilde = O} it is then enough to prove that $\Im(\widetilde{\Phi}^{m,k}_{(S,v,H)})$ contains an orientation reversing isometry.
    
     First of all, let $(S',v',H')$ be an $(m,k)$-triple such that $\Pic(S')=\ZZ.H'$ and let us pick a non-trivial element
     \[ \eta\in\Hom_{\mathcal{G}_{\operatorname{def}}^{m,k}}((S,v,H),(S',v',H')). \]
     
     By Lemma~\ref{lemma:2.20} we know that the derived autoequivalence $\mathsf{L}$ is a morphism in $\Hom_{\mathcal{G}^{m,k}_{\operatorname{FM}}}((S',v',H'),(S',v'_L,H'))$. If we take $L=nH'$ for $n\gg0$, then by Lemma~\ref{lemma:Y-PR} we also know that 
     \[ \operatorname{FM}^\vee_{\Delta}\in\Hom_{\mathcal{G}^{m,k}_{\operatorname{FM}}}((S',v'_L,H'),(S',\hat{v}'_L,H')), \] 
     so that 
    \[ \phi=\operatorname{FM}^\vee_{\Delta}\circ\mathsf{L}\circ\eta\in\Hom_{\mathcal{G}^{m,k}}((S,v,H),(S',\hat{v}'_L,H')). \]
    Notice that $\phi^{\oH}$ is orientation reversing (cf.\ proof of Proposition~\ref{prop:iso of functors}). 

    By Remark~\ref{rmk:non empty} there is a morphism $\psi\in\Hom_{\mathcal{G}^{m,k}}((S',\hat{v}'_L,H'),(S,v,H))$ such that:
    \begin{itemize}
        \item $\psi$ is obtained by concatenating morphisms in $\mathcal{G}^{m,k}_{\operatorname{def}}$ with morphisms in $\mathcal{G}^{m,k}_{\operatorname{FM}}$ of the form $\mathsf{L}$ and $\operatorname{FM}_{\Delta}$;
        \item $\psi^{\oH}$ is orientation preserving (cf.\ proof of Proposition~\ref{prop:iso of functors} again).
    \end{itemize}
    Then 
    \[ (\psi\circ\phi)^{\oH}\in\Im(\widetilde{\Phi}^{m,k}_{(S,v,H)}) \]
    is orientation reversing and we are done.
\endproof

\subsection{An injective morphism from $\Mon^2_{\operatorname{lt}}(M_{v}(S,H))$ to $\Mon^2(M_{w}(S,H))$}\label{subsection:i sharp}

Let $(S,v,H)$ be an $(m,k)$-triple. Write $v=mw$, with $w$ primitive, and consider the $(1,k)$-triple $(S,w,H)$.

The closed embedding 
\[ i_{w,m}\colon M_w(S,H)\longrightarrow M_v(S,H) \] is defined in Remark~\ref{rmk:strat M} and by Corollary~\ref{cor:w and mw} the pullback morphism
\[ i_{w,m}^*\colon \oH^2(M_v(S,H),\ZZ)\longrightarrow\oH^2(M_w(S,H),\ZZ) \]
is a similitude of lattices. Moreover, by Lemma~\ref{lemma:i sharp}, $i^*_{w,m}$ induces an isomorphism of orthogonal groups

\begin{equation}\label{eqn:i sharp}
i_{w,m}^{\sharp}\colon\Or(\oH^2(M_v(S,H),\ZZ))\longrightarrow\Or(\oH^2(M_w(S,H),\ZZ)).
\end{equation}

More generally, if $(S_1,v_1,H_1)$ and $(S_2,v_2,H_2)$ are two $(m,k)$-triples, and if we write $v_1=mw_1$ and $v_2=mw_2$, then we consider the bijection (see Lemma~\ref{lemma:i sharp gen})

\begin{equation}\label{eqn:i sharpgen}
    i^{\sharp}_{w_1,w_2,m}\colon \Or(\oH^{2}(M_{v_1},\mathbb{Z}),\oH^{2}(M_{v_2},\mathbb{Z}))\longrightarrow \Or(\oH^{2}(M_{w_1},\mathbb{Z}),\oH^{2}(M_{w_2},\mathbb{Z})).
\end{equation}

The next proposition shows that $i_{w_1,w_2,m}^{\sharp}$ sends locally trivial parallel transport operators to parallel transport operators. In particular it follows that the restriction of $i_{w_1,w_2,m}^{\sharp}$ to the subset of locally trivial parallel transport operators is an injection.

\begin{prop}\label{prop:first inclusiongen}
Let $(S_1,v_1,H_1)$ and $(S_2,v_2,H_2)$ be two $(m,k)$-triples with $m>1$, and write $v_1=mw_1$ and $v_2=mw_2$. Then the bijection (\ref{eqn:i sharpgen}) maps $\mathsf{PT}_{\operatorname{lt}}^2(M_{v_1}(S_1,H_1),M_{v_2}(S_2,H_2))$ to $\mathsf{PT}^2(M_{w_1}(S_1,H_1),M_{w_2}(S_2,H_2))$, i.e.\ it  restricts to an injective 
function
\[ 
i_{w_1,w_2,m}^\sharp\colon\mathsf{PT}_{\operatorname{lt}}^2(M_{v_1}(S_1,H_1),M_{v_2}(S_2,H_2))\rightarrow\mathsf{PT}^2(M_{w_1}(S_1,H_1),M_{w_2}(S_2,H_2)). 
\]
\end{prop}
\proof
Let $p\colon \mathcal{X}\to T$ be a locally trivial family of primitive symplectic varieties with fibres $\cX_{t_{1}}\cong M_{v_1}(S_1,H_1)$ and $\cX_{t_{2}}\cong M_{v_2}(S_2,H_2)$, then by Remark~\ref{rmk:strat} we have a relative stratification $\cX\supset\cX_1\supset\cdots\supset\cX_{s}=:\cY$
that restricted to each fiber $X_{t}$ of $\mathcal{X}$ is the stratification of the singularities given by Proposition~\ref{prop:stratification}. By the local triviality of the family and the fact that a subspace of a K\"ahler space is again K\"ahler (see \cite[II,1.3.1(i)~Proposition]{Varouchas}), the smallest stratum $q\colon\mathcal{Y}\to T$ of this relative stratification is a smooth family of irreducible holomorphic symplectic manifolds; by Remark \ref{rmk:strat M} we have $\mathcal{Y}_{t_{1}}\cong M_{w_1}(S_1,H_1)$ and $\mathcal{Y}_{t_{2}}\cong M_{w_2}(S_2,H_2)$. 

Let $\gamma$ be a continuous path in $T$ with initial point in $t_1$ and final point in $t_2$ and let $g\colon \oH^2(M_{v_1}(S_1,H_1),\ZZ)\to\oH^2(M_{v_2}(S_2,H_2),\ZZ)$ be the corresponding parallel transport operator in the family $p\colon\mathcal{X}\to T$. The following claim concludes the proof of the proposition by interpreting $i_{w_1,w_2,m}^\sharp(g)$ as a parallel transport operator from $M_{w_1}(S_1,H_1)$ to $M_{w_2}(S_2,H_2)$.

\begin{claim}
The isometry $i_{w_1,w_2,m}^\sharp(g)$ is the parallel transport operator in the family $q\colon\mathcal{Y}\to T$ along the path $\gamma$.
\end{claim}
Let $g_{\QQ}$ be the $\QQ$-linear extension of $g$, i.e. let $g_{\QQ}$ be the parallel transport operator 
$\mathsf{PT}_p(\gamma)$  along $\gamma$ in the local system $R^2q_*\QQ$.
By definition of $i_{w_1,w_2,m}^\sharp$ (see Lemma \ref{lemma:i sharp gen}), it is enough to show that the $i_{w_1,w_2,m,\QQ}^\sharp(g_{\QQ})$ is the parallel transport operator $\mathsf{PT}_q(\gamma)$ along $\gamma$ in the local system $R^2q_*\QQ$.

By local triviality, the inclusion $i\colon\mathcal{Y}\to\mathcal{X}$ induces a morphism of local systems
\[ i_{\QQ}^*\colon R^2p_{*}\QQ\longrightarrow R^2q_{*}\QQ \]
such that $i^*_{t_{1},\QQ}=i_{w_1,m,\QQ}^{*}$ and $i^{*}_{t_{2},\QQ}=i_{w_2,m,\QQ}^{*}$.
Notice that since $i^*_{t_{1},\QQ}$ (and $i^*_{t_{2},\QQ}$) is an isometry (cf.\ Corollary~\ref{cor:w and mw}), it follows that $i_{\QQ}^*$ is an isomorphism of local systems.

The claim follows  since
$$i_{w_1,w_2,m,\QQ}^\sharp(g_\QQ)  = i_{w_1,w_2,m,\QQ}^\sharp(\mathsf{PT}_{p}(\gamma))
  =i^{*}_{w_2,m,\QQ}\circ \mathsf{PT}_{p}(\gamma)\circ(i^{*}_{w_1,m,\QQ})^{-1}, $$ by definition of $i_{w_1,w_2,m,\QQ}^\sharp$, and 
 $$
  i^{*}_{w_2,m,\QQ}\circ \mathsf{PT}_{p}(\gamma)\circ(i^{*}_{w_1,m,\QQ})^{-1}= \mathsf{PT}_{q}(\gamma)$$ because $i_{\QQ}^*\colon R^2p_{*}\QQ\rightarrow R^2q_{*}\QQ$ is an isomorphism of local systems restricting to $i_{w_1,m,\QQ}^*$ at $t_1$ and to $i_{w_2,m,\QQ}^*$ at $t_2$.
\endproof

\begin{rmk}
Let $\Omega(T,t_{1},t_{2})$ be the set of homotopy classes of paths in $T$ from $t_{1}$ to $t_{2}$, and $\mathsf{PT}_{p}$ (resp.\ $\mathsf{PT}_{q}$) be the morphism mapping the class of a path $[\gamma]$ to the parallel transport operator associated to $\gamma$. With this notation the claim in the proof of Proposition \ref{prop:first inclusiongen} states that  the diagram
\[ 
\xymatrix{
	\Omega(T,t_{1},t_{2})\ar@{->}[r]^-{\mathsf{PT}_{p}}\ar@{->}[dr]_{\mathsf{PT}_{q}} & \Or(\oH^2(M_v,\QQ),\oH^{2}(M_{v'},\QQ))\ar@{->}[d]^{i_{w,w',m}^\sharp}_{\cong} \\
	&  \Or(\oH^2(M_w,\QQ),\oH^{2}(M_{w'},\QQ))
}
\]
is commutative.
\end{rmk}

In the particular case of $(S,v,H)=(S',v',H')$ we get the following.

\begin{cor}\label{prop:first inclusion}
Let $(S,v,H)$ be an $(m,k)$-triple with $m>1$, and write $v=mw$. Then the isomorphism (\ref{eqn:i sharp}) restricts to an injective group morphism
\[ i_{w,m}^\sharp\colon\Mon^2_{\operatorname{lt}}(M_{v}(S,H))\longrightarrow\Mon^2(M_w(S,H)). \]
\end{cor}

\subsection{The main results}\label{section:main result}

\begin{thm}\label{thm:Mon 2 lt}
    Let $(S,v,H)$ be an $(m,k)$-triple and suppose that $m>1$. Then 
    \[ \Mon^2_{\operatorname{lt}}(M_{v}(S,H))=\mathsf{W}(M_v(S,H)). \]
\end{thm}

The case when $v$ is primitive, i.e.\ $m=1$, has been proved by Markman.
\begin{thm}[\protect{\cite[Theorem~1.1]{Markman:IntegralConstraints}}]\label{thm:Markman}
    Let $(S,w,H)$ be an $(1,k)$-triple. Then 
    \[ \Mon^2(M_w(S,H))=\mathsf{W}(M_w(S,H)). \]
\end{thm}

We will use Markman's result in our proof, this is the reason why we have omitted this case from our statement.

In the proof we use the following notation: if $H$ is a subgroup of a group $G$, then $[G:H]$ stands for its index. 

\begin{proof}
    We will use the simplified notation $M_{w}$ (resp.\ $M_{v}$) instead of $M_{w}(S,H)$ (resp.\ $M_{v}(S,H)$). 
    
    By Theorem~\ref{thm:W in m k} we have that 
    \[  \mathsf{W}(M_v)\subset\Mon^2_{\operatorname{lt}}(M_v)\subset \Or(\oH^2(M_v,\ZZ)) \] 
    so that 
    \[ [\Or(\oH^2(M_v,\ZZ)):\mathsf{W}(M_v)]\geq[\Or(\oH^2(M_v,\ZZ)):\Mon^2_{\operatorname{lt}}(M_v)]. \]
    On the other hand since $m>1$, by Lemma \ref{lemma:i sharp} and  Corollary~\ref{prop:first inclusion} there is an isomorphism
    \[ i_{w,m}^\sharp\colon
    \Or(\oH^2(M_v,\ZZ))\longrightarrow\Or(\oH^2(M_w,\ZZ))
\]
 such that $$i_{w,m}^\sharp (\Mon^2_{\operatorname{lt}}(M_{v}))\subset \Mon^2(M_{w})$$ and, by Theorem~\ref{thm:Markman}, $\Mon^2(M_w)=\mathsf{W}(\oH^2(M_w,\ZZ))$. It follows that 
    \[ [\Or(\oH^2(M_v,\ZZ)):\Mon^2_{\operatorname{lt}}(M_v)]\geq[\Or(\oH^2(M_w,\ZZ)):\mathsf{W}(M_w)]. \]
    Finally, since $\oH^2(M_v,\ZZ)$ and $\oH^2(M_w,\ZZ)$ are abstractly isometric as lattices, and since the groups $\mathsf{W}(M_v)$ and $\mathsf{W}(M_w)$ are defined only in lattice-theoretic terms, we get the equality
    \begin{equation*}
        [\Or(\oH^2(M_v,\ZZ)):\mathsf{W}(M_v)]=[\Or(\oH^2(M_w,\ZZ)):\mathsf{W}(M_w)],
    \end{equation*}
    which allows us to deduce that
    \[ [\Or(\oH^2(M_v,\ZZ)):\Mon^2_{\operatorname{lt}}(M_v)]=[\Or(\oH^2(M_v,\ZZ)):\mathsf{W}(M_v)]. \]
    Since by Theorem~\ref{thm:W in m k} we have that 
    \[  \mathsf{W}(M_v)\subset\Mon^2_{\operatorname{lt}}(M_v),\] we conclude  that $\Mon^2_{\operatorname{lt}}(M_v)=\mathsf{W}(M_v)$.
\end{proof}

\begin{cor}[\protect{\cite[Lemma~4.2]{Markman:IntegralConstraints}}]\label{cor:index}
    Let $(S,v,H)$ be a $(m,k)$-triple. Then 
    \[ \Mon^2_{\operatorname{lt}}(M_v(S,H))\subset\Or^+(\oH^2(M_v(S,H),\ZZ)) \] 
    has index $2^{\rho(k)-1}$, where $\rho(k)$ is the number of distinct primes in the factorisation of $k$.
\end{cor}

\begin{rmk}
    Let $(S,v,H)$ be an $(m,k)$-triple.
    It is interesting to point out that the index of $\Mon^2_{\operatorname{lt}}(M_v(S,H))$ in $\Or^+(\oH^2(M_v(S,H),\ZZ))$ does not depend on $m$. In particular we have that if  $k$ is a power of a prime, then 
    \[ \Mon^2_{\operatorname{lt}}(M_v(S,H))=\Or^+(\oH^2(M_v(S,H),\ZZ)). \]
\end{rmk}


Theorem~\ref{thm:Mon 2 lt} can be analogously formulated in terms of the inclusion $i_{w,m}\colon M_w(S,H)\rightarrow M_v(S,H)$. 

\begin{thm}\label{thm:i sharp is iso}
  Let $(S,v,H)$ be an $(m,k)$-triple, $v=mw$ with $m>1$ and $(S,w,H)$ the corresponding $(1,k)$-triple. Then the isomorphism
  \[ i_{w,m}^\sharp\colon\Or(\oH^2(M_{v}(S,H),\ZZ))\longrightarrow\Or(\oH^2(M_{w}(S,H),\ZZ)) \]
 induces by restriction an isomorphism 
 \[ i_{w,m}^\sharp\colon\Mon^2_{\operatorname{lt}}(M_{v}(S,H))\longrightarrow\Mon^2(M_w(S,H)). \]  
 \end{thm}

\begin{proof}
    First of all recall that $i_{w,m}^\sharp =\lambda_{(S,w,H)}^{\sharp}\circ(
    \lambda_{(S,v,H)}^\sharp)^{-1}$ by Lemma~\ref{lemma:i sharp}.(2).

   Since the groups $\mathsf{W}(M_{v}(S,H))$ and $\mathsf{W}(M_w(S,H))$ are defined in lattice-theoretic terms, and $\lambda_{(S,w,H)}$ and $\lambda_{(S,v,H)}$ are isometries, it follows that 
    \begin{equation}\label{eqn:qui} i_{w,m}^\sharp(\mathsf{W}(M_{v}(S,H)))=\mathsf{W}(M_w(S,H)). 
    \end{equation}

    On the other hand, by Corollary~\ref{prop:first inclusion} the isometry $i_{w,m}^\sharp$ sends monodromy operators in monodromy operators, i.e.\ 
    \begin{equation}\label{eqn:quo} i_{w,m}^\sharp(\Mon^2_{\operatorname{lt}}(M_{v}(S,H)))\subset\Mon^2(M_w(S,H)). 
    \end{equation}

    Combining this with Theorem~\ref{thm:W in m k} and Theorem~\ref{thm:Markman}, we eventually get
    \begin{align*}
    \mathsf{W}(M_w(S,H)) & =i_{w,m}^\sharp(\mathsf{W}(M_{v}(S,H))) & \mbox{(by equality (\ref{eqn:qui}))} \\
     & \subseteq i_{w,m}^\sharp(\Mon^2_{\operatorname{lt}}(M_{v}(S,H))) & \mbox{(by Theorem~\ref{thm:W in m k})} \\
      & \subseteq\Mon^2(M_w(S,H)) & \mbox{(by inclusion (\ref{eqn:quo}))} \\
       & =\mathsf{W}(M_w(S,H)) & \mbox{(by Theorem~\ref{thm:Markman})}
      \end{align*}
      from which it follows that 
      \[ i_{w,m}^\sharp(\Mon^2_{\operatorname{lt}}(M_{v}(S,H)))=\Mon^2(M_w(S,H)) \]
      and therefore the claim.
    \end{proof}

 \begin{rmk}
    Suppose that $m'\neq m$ are two strictly positive integers and put $v'=m'w$ and $v=mw$. There is then a natural isomorphism
    \[ (i^\sharp_{w,m'})^{-1}\circ i^\sharp_{w,m}\colon\Mon^2_{\operatorname{lt}}(M_v(S,H))\stackrel{\sim}{\longrightarrow}\Mon^2_{\operatorname{lt}}(M_{v'}(S,H)).  \]
    It follows that the isomorphism class of $\Mon^2_{\operatorname{lt}}(M_v(S,H))$ does not depend on $m$, but only on $w$.
\end{rmk}

Since the locally trivial monodromy group is invariant along locally trivial families of primitive symplectic varieties (see Definition~\ref{def:lt}), we get the following corollary.
\begin{cor}\label{cor:ogni X}
    Let $X$ be an irreducible symplectic variety that is locally trivially deformation equivalent to a moduli space $M_v(S,H)$, where $(S,v,H)$ is an $(m,k)$-triple.

    Then $\Mon^2_{\operatorname{lt}}(X)$ is the subgroup $\mathsf{W}(X)\subset\Or(\oH^2(X,\ZZ))$ of the orientation preserving isometries acting as $\pm\id$ on the discriminant group.
\end{cor}

\proof 
Since $X$ is locally trivially deformation equivalent to a moduli space $M_v(S,H)$ as in the statement, there exists a locally trivial parallel transport operator 
\[ g\colon\oH^{2}(M_{v}(S,H),\mathbb{Z})\longrightarrow \oH^{2}(X,\mathbb{Z}). \] 
Conjugation with $g$ gives an isomorphism of orthogonal groups 
\[ g^\sharp\colon\Or(\oH^{2}(M_{v}(S,H),\mathbb{Z}))\longrightarrow \Or(\oH^{2}(X,\mathbb{Z})) \]
which maps $\Mon^{2}_{\operatorname{lt}}(M_{v}(S,H))$ to $\Mon^{2}_{\operatorname{lt}}(X)$ by definition.

Moreover, since $g^\sharp$ is induced by an isometry, it also has to map $\mathsf{W}(M_{v}(S,H))$ to $\mathsf{W}(X)$.

Since $\mathsf{W}(M_{v}(S,H))=\Mon^{2}_{\operatorname{lt}}(M_{v}(S,H))$ by Theorem~\ref{thm:Mon 2 lt}, it follows that 
\[ \mathsf{W}(X)=\Mon^{2}_{\operatorname{lt}}(X,\ZZ). \]
\endproof

Finally, let us observe that the analogue of Theorem~\ref{thm:i sharp is iso} holds for any symplectic variety $X$ locally trivial deformation equivalent to $M_v(S,H)$. 

In order to state the result, we recall that the most singular locus $Y$ of $X$ is an irreducible symplectic manifold (see Proposition~\ref{prop:stratification}). Let us formally define the morphism 
\[ i_{Y,X}^\sharp\colon\Mon^2_{\operatorname{lt}}(X) \longrightarrow \Mon^2(Y) \] induced by the inclusion $i_{Y,X}\colon Y\to X$, as follows. For every monodromy operator $g\in \Mon^2_{\operatorname{lt}}(X)$, there exists a locally trivial family of irreducible symplectic varieties $p\colon \mathcal{X}\to T$, a point $\bar{t}\in T$ and loop $\gamma$ centered at $\bar{t}$ such that $\mathcal{X}_{\bar{t}}=X$ and $g$ is the parallel transport operator $\mathsf{PT}_p(\gamma)$ associated with the family $p$ and the loop $\gamma$. By local triviality of $p$, the restriction of  $p$ to the  most singular locus  $\mathcal{Y}$ of $\mathcal{X}$ gives a smooth family $q\colon\mathcal{Y}\to T$ of irreducible symplectic manifolds and we define $i_{Y,X}^{\sharp}(g)\colon \oH^{2}(Y,\mathbb{Z})\to \oH^{2}(Y,\mathbb{Z})$ as the parallel transport operator $\mathsf{PT}_q(\gamma)$ associated with  the family $q$ and the loop $\gamma$. We remark that  $i_{Y,X}^{\sharp}(g)$ is well defined and its definition immediately implies 
\begin{equation}\label{penultima} i_{Y,X}^{\sharp}(g) \circ i_{Y,X}^{*}= i_{Y,X}^{*}\circ g\colon\oH^2(X,\mathbb{Z})
\longrightarrow \oH^2(Y,\mathbb{Z})
\end{equation}

\begin{cor}\label{cor:i sharp is iso gen}
    Let $X$ be an irreducible symplectic variety that is locally trivially deformation equivalent to a moduli space $M_v(S,H)$, where $(S,v,H)$ is an $(m,k)$-triple. Let $Y\subset X$ be the most singular locus 
    and $i_{Y,X}\colon Y\to X$ be the closed embedding. Then the morphism 
    \[ i_{Y,X}^\sharp\colon\Mon^2_{\operatorname{lt}}(X)\stackrel{\sim}{\longrightarrow}\Mon^2(Y) \] is an isomorphism.
\end{cor}
\begin{proof}
    Let $p\colon\mathcal{X}\to T$ be a locally trivial family of irreducible symplectic varieties such that there exists two points $t_1,t_2\in T$ with $\mathcal{X}_{t_1}=X$ and $\mathcal{X}_{t_2}=M_v(S,H)$, respectively. Here $M_v(S,H)$ is the irreducible symplectic variety associated to an $(m,k)$-triple $(S,v,H)$. Let $\mathcal{Y}$ be the most singualar locus of $\mathcal{X}$ and let $q\colon\mathcal{Y}\to T$ be the restriction of $p$; by Remark~\ref{rmk:strat} the family $q$ is a family  of irreducible holomorphic symplectic manifolds. Let $\gamma$ be a path contained in  $T$ from $t_1$ to $t_2$ and denote by $\mathsf{PT}_p(\gamma)\colon\oH^{2}(X,\mathbb{Z})\to \oH^{2}(M_{v}(S,H),\mathbb{Z})$ and $\mathsf{PT}_q(\gamma )\colon\oH^{2}(Y,\mathbb{Z})\to\oH^{2}(M_{w}(S,H),\mathbb{Z})$ the parallel transport operators associated with the path $\gamma$ and the families $p$ and $q$, respectively. By construction we have \begin{equation}\label{ultima}
    i_{w,m}^{*}\circ \mathsf{PT}_p(\gamma)= \mathsf{PT}_q(\gamma)\circ i_{Y,X}^*\colon \oH^{2}(X,\mathbb{Z})\to \oH^{2}(M_{w}(S,H),\mathbb{Z}).
    \end{equation} 
    
    Finally we define the group isomorphisms 
    \begin{align}
\mathsf{PT}_p^{\sharp}(\gamma)\colon \Mon^2_{\operatorname{lt}}(M_{v}(S,H)) & \longrightarrow \Mon^2_{\operatorname{lt}}(X) \nonumber\\
	g & \longmapsto (\mathsf{PT}_p (\gamma))^{-1}\circ g \circ \mathsf{PT}_p(\gamma)\nonumber 
\end{align}
and
 \begin{align}
\mathsf{PT}_q^{\sharp}(\gamma)\colon \Mon^2(M_{w}(S,H)) & \longrightarrow \Mon^2(Y) \nonumber\\
	g & \longmapsto (\mathsf{PT}_q (\gamma))^{-1}\circ g \circ \mathsf{PT}_q(\gamma).\nonumber 
\end{align}

By formulae (\ref{penultima}) and (\ref{ultima}) we deduce 
that 
$$i_{Y,X}^{\sharp}= \mathsf{PT}_{q}^{\sharp}(\gamma)\circ i_{w,m}^{\sharp}\circ (\mathsf{PT}_{p}^{\sharp}(\gamma))^{-1} $$ and, since
$\mathsf{PT}_{q}^{\sharp}(\gamma)$, $i_{w,m}^{\sharp}$ and $ (\mathsf{PT}_{p}^{\sharp}(\gamma))^{-1}$ are isomorphisms, the morphism  $i_{Y,X}^{\sharp}$ is an isomorphism too.
\end{proof}


\end{document}